\documentclass[final]{siamltex}

\usepackage{amsmath}
\usepackage{amssymb}
\usepackage{graphicx}
\usepackage{subfigure}
\usepackage{longtable}
\newtheorem{algorithm}[theorem]{Algorithm}

\renewcommand{\u}{\ensuremath{\begin{pmatrix} \rho \\ \rho u \\ E \\ \end{pmatrix}}}
\newcommand{\f}{\ensuremath{\begin{pmatrix} \rho u \\ \rho u^2 + P \\  (E + P)u \\ \end{pmatrix}}}

\newcommand{\EE}{\mathbb{E}}
\newcommand{\PP}{\mathbb{P}}
\newcommand{\QQ}{\mathbb{Q}}
\newcommand{\RR}{\mathbb{R}}
\newcommand{\NN}{\mathbb{N}}
\newcommand{\calI}{\mathcal{I}}
\newcommand{\calA}{\mathcal{A}}
\newcommand{\calB}{\mathcal{B}}
\newcommand{\bfA}{\mathbf{A}}
\newcommand{\bfB}{\mathbf{B}}
\newcommand{\bfN}{\mathbf{N}}
\newcommand{\vecu}{\tilde{\mathbf{u}}_{in}}
\newcommand{\vecuSS}{\tilde{\mathbf{u}}_{in}^{**}}

\newcommand{\Std}{\mathbf{Std}}
\newcommand{\eps}{\epsilon}

\title{Probability of Failure in Hypersonic Engines Using Large Deviations}
\author{
George Papanicolaou\thanks{Mathematics Department, Stanford University ({\tt papanicolaou@stanford.edu})}
\and 
Nicholas West\thanks{Institute for Computational and Mathematical Engineering (ICME), Stanford University,
({\tt nickwest@stanford.edu})}
\and 
Tzu-Wei Yang\thanks{Institute for Computational and Mathematical Engineering (ICME), Stanford University 
({\tt twyang@stanford.edu})}}

\begin{document}

\maketitle

\begin{abstract}
	We consider a reduced order model of an air-breathing hypersonic engine with a time-dependent stochastic 
	inflow that may cause the failure of the engine. The probability of failure is analyzed by the 
	Freidlin-Wentzell theory, the large deviation principle for finite dimensional stochastic differential 
	equations. We compute the asymptotic failure probability by numerically solving the constrained 
	optimization related to the large deviation problem. A large-deviation-based importance sampling 
	suggested by the most probable inflow perturbation is also implemented to compute the probability of 
	failure of the engine. The numerical simulations show that the importance sampling method is much more 
	efficient than the basic Monte Carlo method.
\end{abstract}



\section{Introduction}

Operation of a scramjet (supersonic combustion ramjet) is difficult to accurately model with 
state-of-the-art codes due to the complex physical and chemical systems governing the flow of air through 
and the combustion in the engine. When coupled with environmental uncertainties, a priori design of a 
safety, high-performance operation plan (fuel schedule) is a lofty ambition. This paper numerically analyzes 
the probability of failure for a scaled engine operating at about Mach $2$. While there are many 
uncertainties outside of the engine that can effect the operability, we focus only on those within the core 
engine-system: the isolator, combustor and nozzle. The combustor is the location of the most complex (and 
least certain) chemistry: the amount of heat released here directly effects the performance of the system.

Thermal choking can result from excessive fueling, which decreases immediate performance and produces a 
shock that may travel through the isolator. If the shock reaches the entrance of the isolator, the engine 
will stall; this is called unstart. While there are many other causes of unstart (for example, thermal 
deformation of the engine \cite{Buchmann1979}), we focus on the unstart directly related to the inflow 
perturbations \cite{Sato1992} and the fueling of the engine. Iaccarino et al \cite{Iaccarino2011} studied 
the relationship between the amount of heat released and the operability of the engine over short time 
scales and with a low-fidelity model of the stochastic nature of the fueling. This paper studies the 
resulting uncertainty of the flow in the engine with a stochastic inflow Mach number.

The Euler equations are frequently employed to model compressible flows in aerospace models, especially in 
steady-state as a numerically tractable model when designing airfoils (see \cite{Jameson1994}). The 
time-dependent, quasi-one-dimensional form of the equations can be used to capture the geometry of a 
compression-expansion engine (see e.g. \cite{Shapiro1953}) and replicates many of the physical features of 
actual flows in scramjet engines. When used to model scramjets, a forcing term is added that models the heat 
release mechanism of fueling. It is well documented that these models capture the physical shock that 
results from thermal choking. Due to the one-dimensional nature of these equations, they can be solved 
quickly and are ideal as a reduced order model.

The large deviation principle is used to analyze events with exponentially small probability. The 
Freidlin-Wentzell theory, the large deviation principle for finite dimensional stochastic differential 
equations is the mathematical tool to compute the probability of failure: when the random perturbation in 
the stochastic differential equation is small, the probability of failure decreases exponentially fast and 
the rate of decay of the probability is governed by the minimum of the rate function over the event of 
interest. We numerically solve this constrained optimization problem to obtain the asymptotic probability of 
failure and the most probable path causing unstart under several interesting cases. This so-called the 
\textit{minimum action method} has been successfully applied to different model problems (see 
\cite{E2004,Zhou2008}).

Another obvious way to compute the probability of unstart is to use the Monte Carlo simulations. It has been 
extensively used in the engineering community to consider more elaborate scramjet models such as the 
two-dimensional model with the second order discretization, and it can be accelerated by using the 
adjoint-based sampling method \cite{Wang2012}. When the targeted probability is small, however, because of 
the natural limitation of the basic Monte Carlo method, one needs a excessively large number of samples to 
accurately estimate the probability, and such large amount of computations leads to the inefficiency of the 
basic Monte Carlo method. We use the large-deviation-based importance sampling technique, the importance 
sampling Monte Carlo method whose change of measure based on the minimizer of the large deviation principle. 
Our numerical results show that the large-deviation-based importance sampling outperforms the basic Monte 
Carlo method.

This paper is organized as follows: in Section \ref{sec:model} we discuss the equations that govern the 
flow, the geometry of the engine and the definition of the unstart; Section \ref{sec:general LDP} briefly 
introduces the classical Freidlin-Wentzell theory, the large deviation principle for finite dimensional 
stochastic differential equations, and the large deviations for the related Euler schemes. In Section 
\ref{sec:LDP for unstart}, we explain how to formulate the unstart of the scramjet as a large deviation 
problem. Section \ref{sec:numerical result of LD} shows the numerical results of the large deviation 
problems in Section \ref{sec:LDP for unstart} under different settings. In Section 
\ref{sec:importance sampling} we use the importance sampling Monte Carlo method based on the solution of the 
large deviation problems in Section \ref{sec:numerical result of LD} to directly estimate the probability of 
the unstart. Section \ref{sec:conclusion} concludes this paper. The table of parameters and the numerical 
PDE method for the governing equation are in the appendices.

\section{Model Problem}
\label{sec:model}

\subsection{Governing Equations and Engine Geometry}

The quasi-1D compressible Euler equations serve as our reduced model of the engine-combustion system and 
capture the phenomena of unstart due to fueling. This model was developed by Iaccarino et al 
\cite{Iaccarino2011} and is similar to the model developed by Bussing and Murmam \cite{Bussing1983}.  The 
quasi-1D compressible Euler equations are the following hyperbolic system:
\begin{equation}
	\label{eq:qce}
	\u_t + \f_x 
	= \frac{A'(x)}{A(x)} \left( \begin{pmatrix} 0 \\ P \\ 0 \\ \end{pmatrix} - \f\right) 
	+ \begin{pmatrix} 0 \\ 0 \\ f(x,t) \\ \end{pmatrix},
\end{equation}
where $\rho$ is the density of the fluid, $\rho u$ is the momentum and $E$ is the total energy. The pressure 
$P$ can be derived from an equation of state and we take $P=(\gamma-1)(E-\rho u^2/2)$ where $\gamma$ is the 
ratio of specific heats, taken to be $1.4$. The function $A(x)$ describes the cross-sectional area of the 
engine; we assume that the width is constant and thus the area varies as the height; see 
Figure \ref{fig:geom} for a sample height profile and the following mathematical definition:
\begin{equation*}
	A(x) = 
	\begin{cases} 
		A_0 -x \sin \theta_I, & -L_I < x < 0, \\ 
		A_0 + x \sin \theta_C, & 0 \leq x \leq L_C, \\
		A_0 +  L_C \sin \theta_C + (x - L_c) \sin \theta_E, & L_C < x \leq L_C + L_E.
	\end{cases}
\end{equation*}
The term $f(x,t)$ models the heat release due to fueling and takes the form:
\begin{equation}
	\label{eq:fuel}
	f(x,t) = 
	\begin{cases}
		f(t) f(x) \cdot \phi \cdot f_{stoch} \cdot H_{prop} \cdot A_0 
		\cdot \rho_0 \cdot u_0 / (L^2_C \cdot A(x)), & x \in [0, L_C],\\
		0,\quad &\text{otherwise.}
	\end{cases}
\end{equation}
where $f(x) = x^{1/3}$ following Iaccarino et al \cite{Iaccarino2011} and Riggins et al \cite{Riggins2006} 
and O'Byrne et al \cite{OByrne2000}. $\phi$ is the equivalence ratio and governs the amount of heat released 
into the system per unit time and $f(t)$ is an indicator for when the engine is fueling.

\begin{figure}
	\centering
	\includegraphics[trim = 0in 3.5in 0in 3.5in, clip, width=\textwidth]{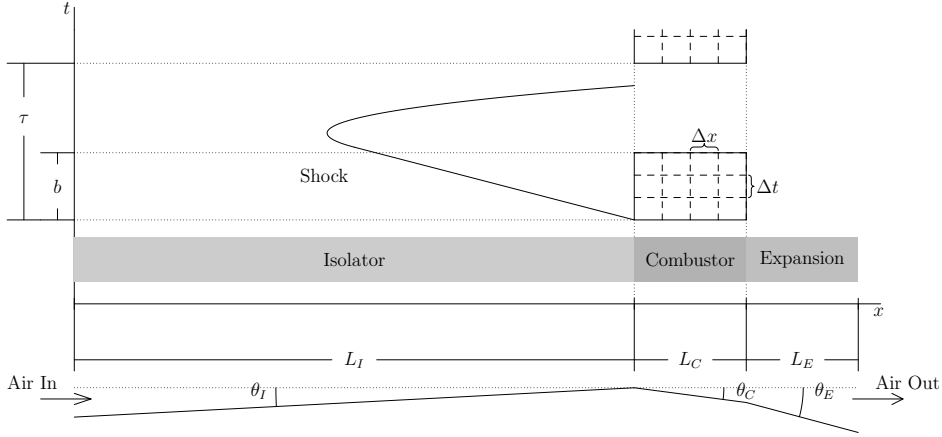}
	\caption{
	\label{fig:geom}
	Geometry of 1D engine model.  This figure is a schematic of both the engine geometry (shown in the lower 
	half of the image) and of the fueling profile and a resulting shock. $L_I$, $L_C$ and $L_E$ are the 
	lengths of the isolator, combustor and expansion region; $\theta_I$, $\theta_C$ and $\theta_E$ are their 
	respective angles; the air flows in from the left and out at the right. In the top half, the $x$ axis 
	gives the location in the engine and the $y$ axis is time. $T$ is the full length of a fueling period 
	and $b$ is the burst length.}	
\end{figure}

\subsection{Unstart of the Engine}

The Mach number, defined as $M = u / \sqrt{\gamma P / \rho}$, characterizes the behavior of the flow.  The 
Mach number of the flow is plotted Figures \ref{fig:low}-\ref{fig:stall} for different values of $\phi$, 
when the engine is fueled from $t = 0.5$ms to $t= 1.5$ms.  In all cases a shock develops, from a supersonic 
Mach number (Mach 2, green) to a subsonic Mach number (0.4, blue).  This shock extends into the isolator of 
the engine and persists after the fueling has stopped.  The distance into the isolator that the shock 
travels before receding and the amount of time it takes for the engine to return to a ``normal'' idle state 
are functions of how much head is injected.  In Figure \ref{fig:stall} the shock reaches the left boundary 
of the inlet; this is called ``unstart'' and the engine has failed and ceases to produce thrust. To 
determine if the engine has unstarted, the shock location, defined as:
\begin{equation*}
	x_{shock} = \sup \{ x \in [-L_c, 0] : M(x) \geq 1\},
\end{equation*}
is tracked. When the engine is in a idle state $x_{shock} = 0$; when the engine has failed $x_{shock}=-L_c$.

\begin{figure}
	\centering
	\subfigure[Low Fueling Rate]{\includegraphics[width=.48\textwidth]{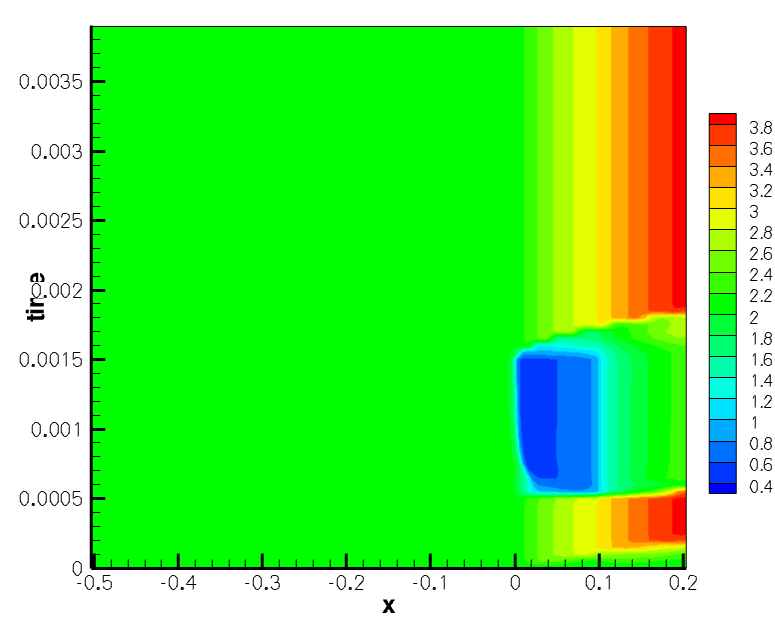} 
	\label{fig:low}}
	\subfigure[Moderate Fueling Rate]{\includegraphics[width=.48\textwidth]{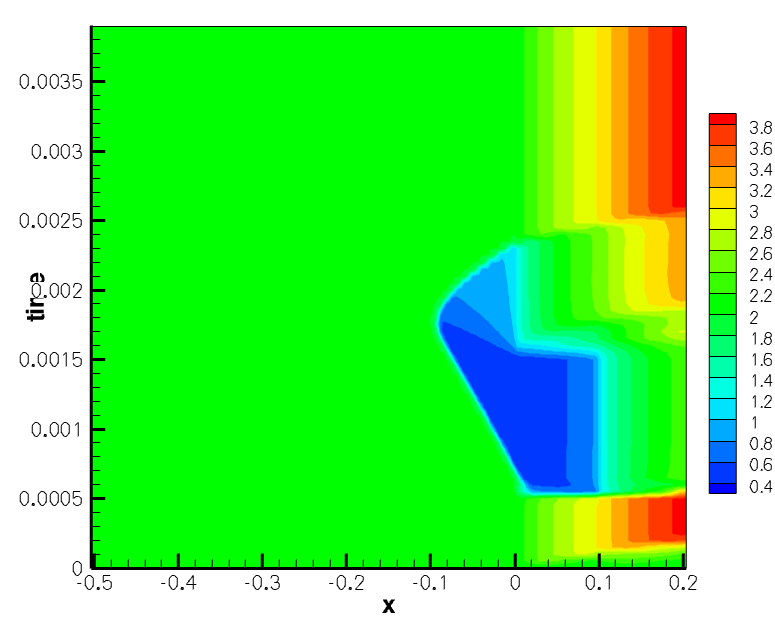} 
	\label{fig:mod}}
	
	\subfigure[Excessive/Stalling]{\includegraphics[width=.48\textwidth]{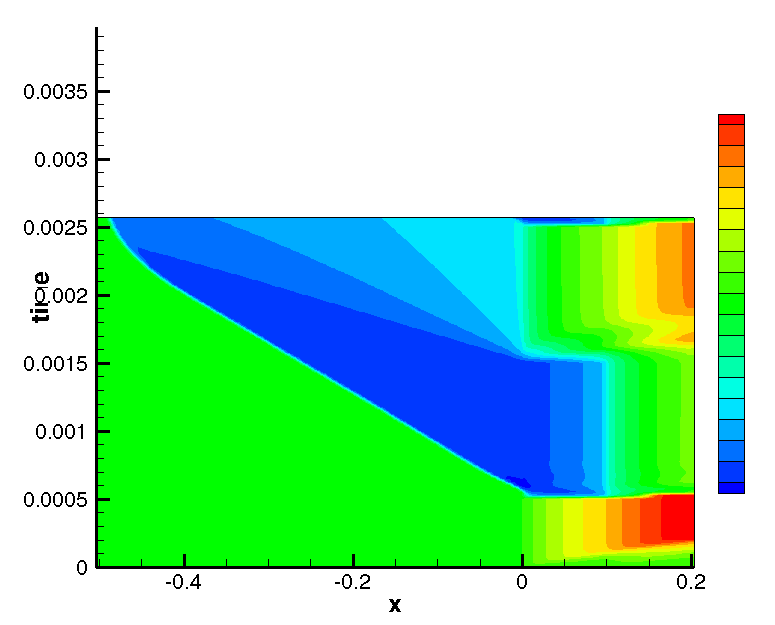} 
	\label{fig:high}}
	\subfigure[Compounding Shock]{\includegraphics[width=.48\textwidth]{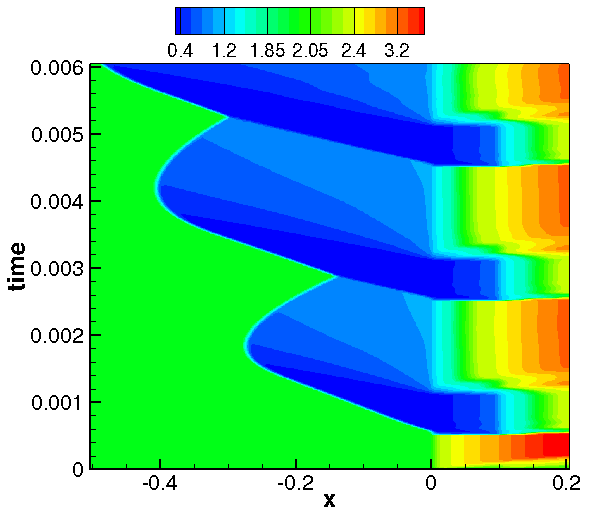} 
	\label{fig:stall}}
	\caption{Effect of heat release parameters on deterministic solutions of Mach number. (a)-(c):Inflow 
	Mach number = 2.0, $b = 1$ms.  Progressing from (a) to (c), more head is added per unit time, driving 
	the shock further into the isolator.  In (d), where the shock reaches the left boundary, the engine 
	stalls; this is called unstart. (d) Here, $T/b = 3.33$; the bursts are not spaced   
	sufficiently far apart, and the location of the shock compounds.}
\end{figure}

\subsection{Heat Release Model}

With this simplified model, there are two regimes of operation depending on $\phi$: for sufficiently small 
$\phi$ (Figure \ref{fig:low}) the shock does not leave the combustor, however insufficient thrust is 
produced; for $\phi$ greater than a threshold value $\phi^* \approx 0.25$, a shock forms in the isolator and 
will propagate until the engine unstarts.  This suggests that the simplest heat-release program that could 
result in sustained operation of the engine is to inject heat periodically. In this paper $f(t)$ is defined 
as follows:
\begin{equation*}
	f(t) = 
	\begin{cases} 
		1, & t \in [n\tau, n\tau+b),\\
		0, & t \in [n\tau+b, (n+1)\tau).
	\end{cases}
\end{equation*}
where $\tau$ is the length of the fuel cycle and $b$ is the length of the fuel burst (the amount of time 
fuel is being injected into the engine). With this heat release model, a potential cause of unstart is not 
spacing the fuel burst sufficiently far apart ($\tau/b$ too small) which causes the shock to build upon 
itself and eventually leads to unstart, see Figure \ref{fig:stall}.

The instantaneous thrust produced by the engine is given by
\begin{equation*}
	\mathrm{thrust}(t) = \dot{m}_e u_e - \dot{m}_i u_i + (P_e - P_i) A_e 
	= (A \rho u^2)_e - (A \rho u^2)_i + (P_e - P_i) A_e,
\end{equation*}
where the index $e$ is for quantities at the exit of the engine and $i$ if for quantities at the front of 
the inlet.  The mass flow, $\dot{m}$, is given by the state variable $\rho u$.  When the engine is stalled 
no additional thrust is produced.  The thrust produce is proportional to the amount of heat injected.

Depending on the extensive numerical experiments in \cite{West2011}, two fueling profiles are used 
in this paper: $\phi_S=0.78$, $\tau_S=0.5$ms, $b_S=0.1$ms (the short fuel cycle) and $\phi_L=0.78$, 
$\tau_L=2$ms, $b_L=0.4$ms (the long fuel cycle). Note that because $\phi_S=\phi_L$ and 
$\tau_S/b_S=\tau_L/b_L=5$, these two fueling profiles release the same amount of the heat so they generate 
roughly the same amount of thrust.

\subsection{Inflow Uncertainty and Its Effect on Unstart}

There are many sources of uncertainty that contribute to the total uncertainty in the operability of the 
scramjet engine. Those specific to the engine are the geometry (due to manufacturing errors or 
imperfections), the ratio of specific heats $\gamma$, the inflow conditions (the Mach number, density and 
pressure of air entering the engine) and the combustion processes. In this paper we only address the 
uncertainty of the inflow Mach number $M_{in}(t)$ that is modeled as: 
\begin{equation}
	\label{eq:inflow Mach number}
	M_{in}(t) = M_{in}(0) + \eps \sigma_M W_t
\end{equation}
with the initial condition $M_{in}(0)=2$, where $\epsilon$ and $\sigma$ are positive constants and $W_t$ is 
the standard Brownian motion. In addition, we assume that the inflow density $\rho_{in}(t)$ and the inflow 
pressure $P_{in}(t)$ are constant in time. Then the perturbation of $M_{in}(t)$ completely comes from the 
perturbation of the inflow speed $u_{in}(t)$

The inflow uncertainty can greatly affect the stability of the screamjet even though it operas normally 
under steady inflows. As shown in Figure \ref{fig:example of unstart due to stochastic inflow}, if the 
aforementioned fueling profiles are used (the short and long fuel cycles) and $M_{in}(t)=M_{in}(0)$ 
($\epsilon\sigma=0$), the scramjet operates normally. However, if stochastic inflows are used 
($\epsilon\sigma\neq0$), then the inflow may affect the locations of the shocks and therefore the 
probability of the unstart is nonzero.

\begin{figure}
	\centering
	\includegraphics[width=0.32\textwidth]{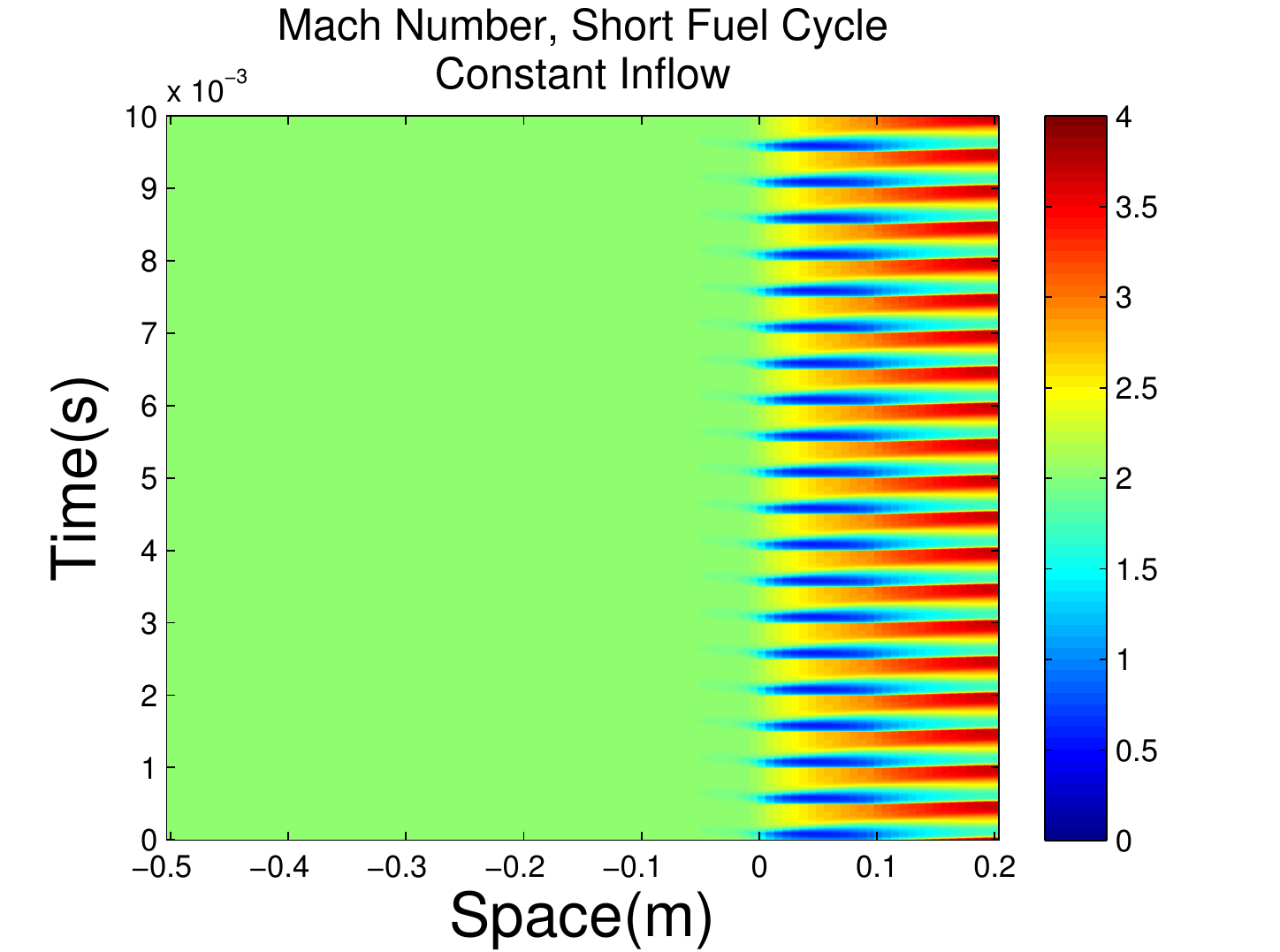}
	\includegraphics[width=0.32\textwidth]{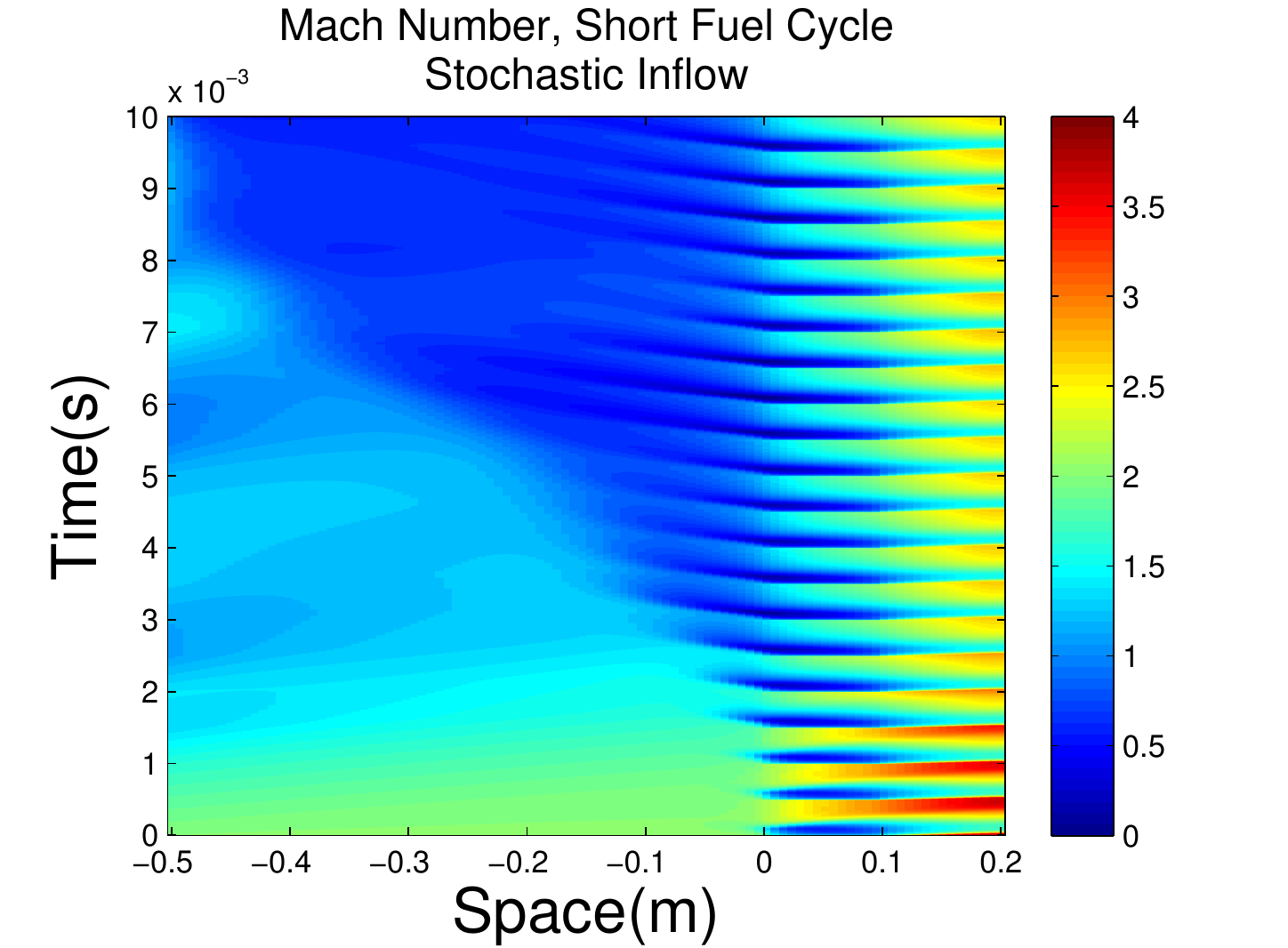}
	\includegraphics[width=0.32\textwidth]{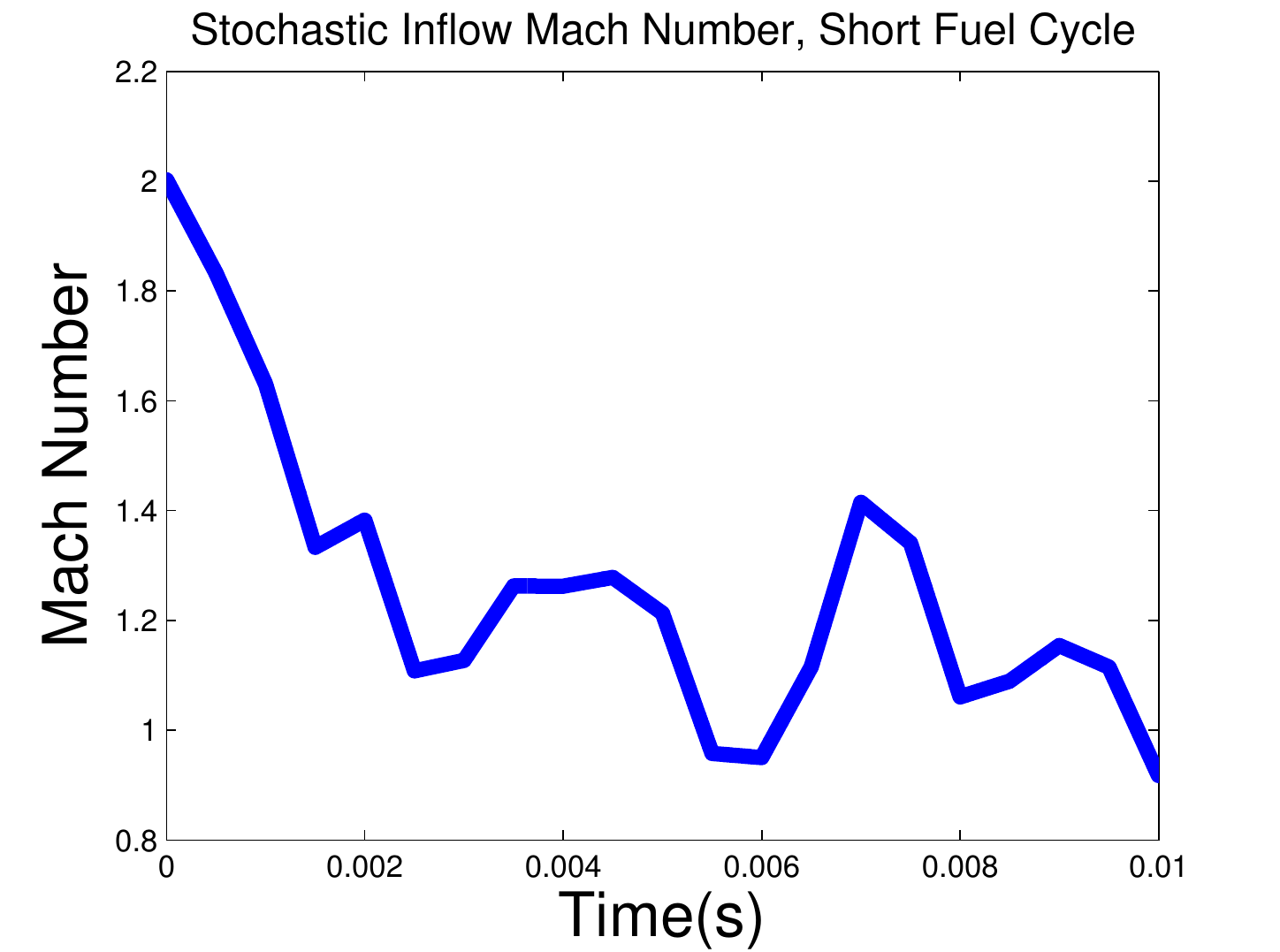}

	\includegraphics[width=0.32\textwidth]{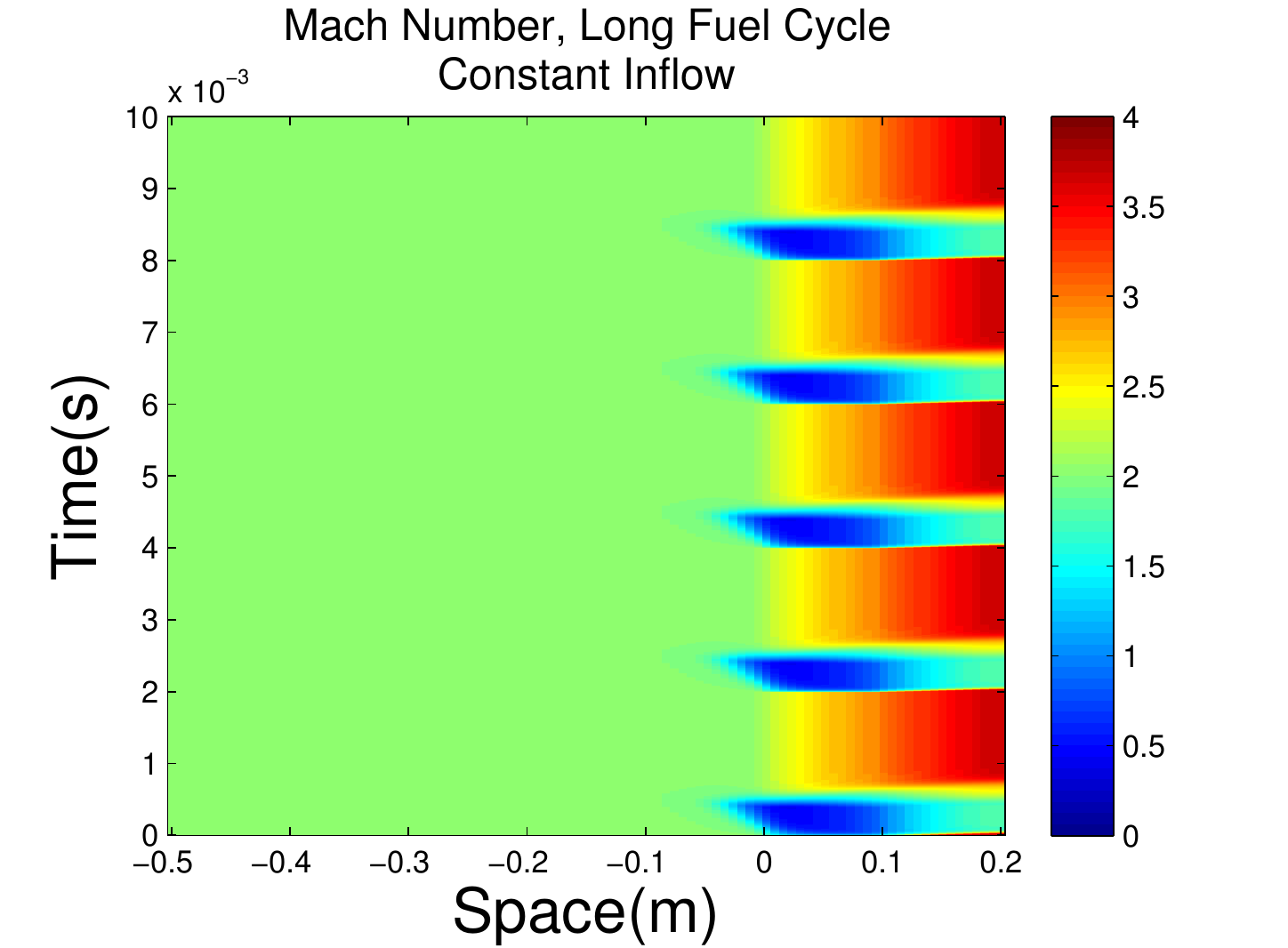}
	\includegraphics[width=0.32\textwidth]{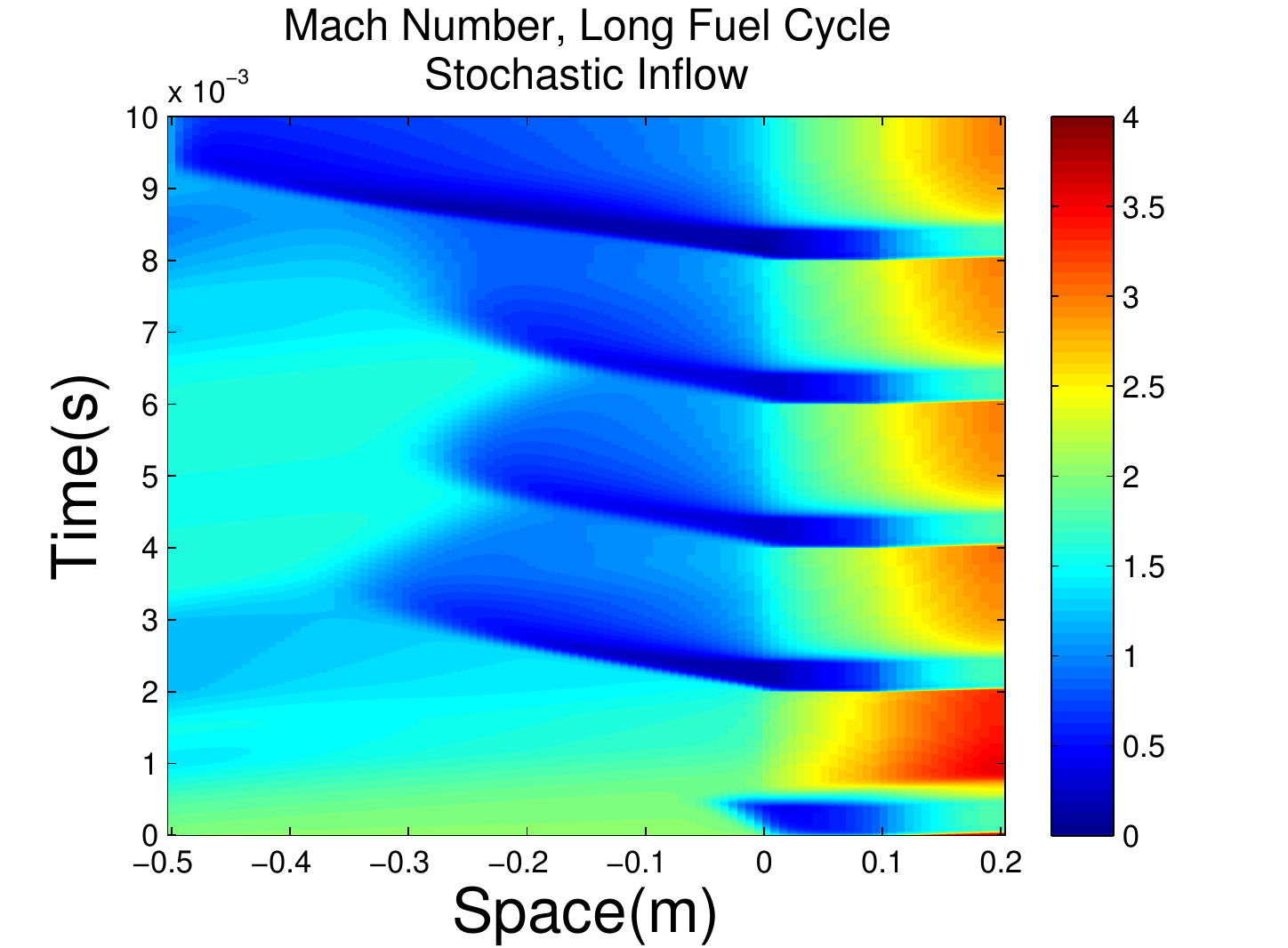}
	\includegraphics[width=0.32\textwidth]{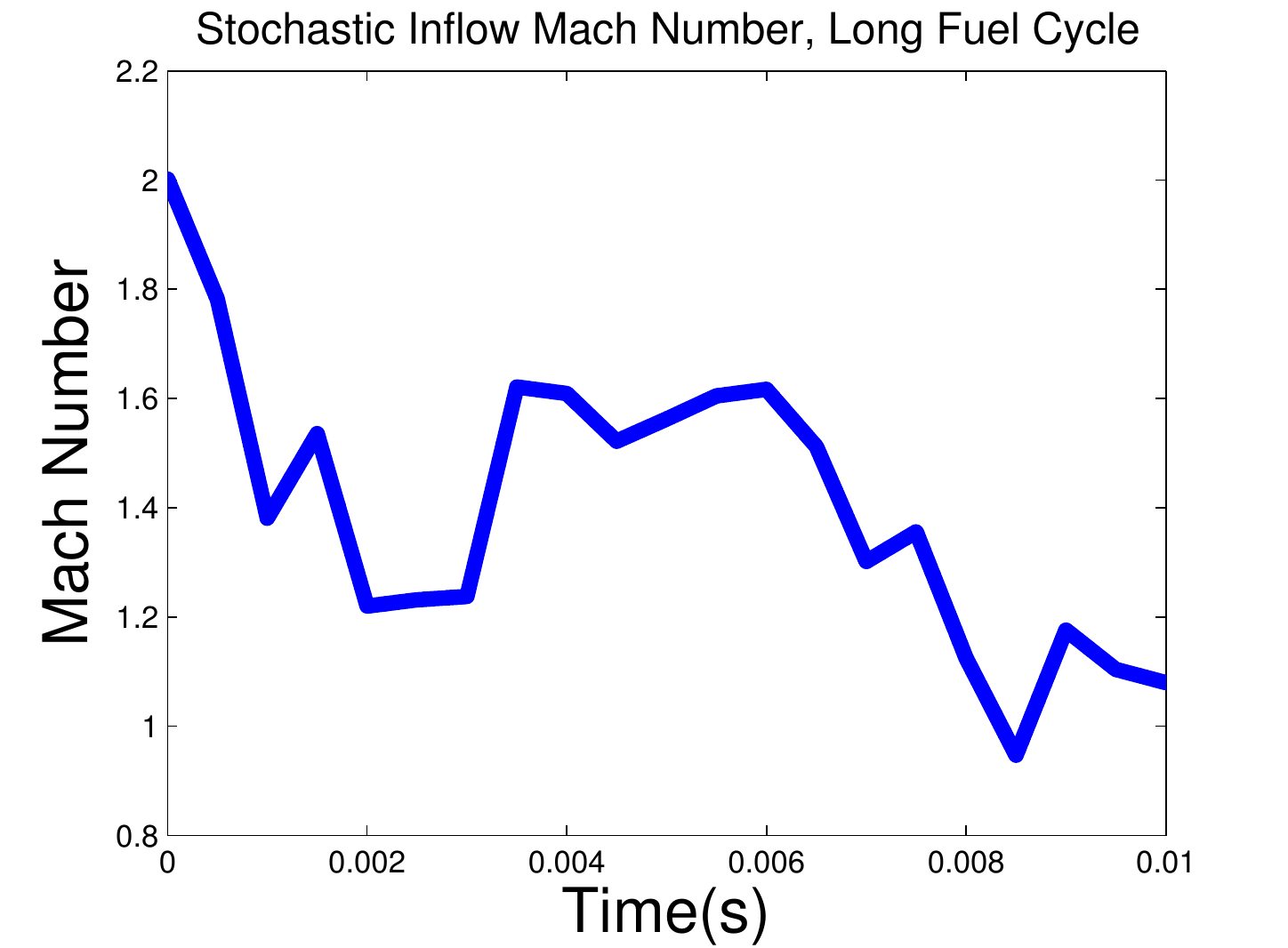}
	\caption{Effects of the inflow uncertainty on the scramjet stability. The left figures are the Mach 
	numbers in the engine with constant inflow (Mach $2$). The engine operates normally under both fueling 
	profiles (the short and long cycles) because the subsonic flows do not reach the left boundary. After 
	replacing the steady inflows by stochastic ones (shown in the right figures), however, the shock 	
	locations are affected by the inflow and reach the left boundary so the unstart happens (shown in the 
	middle figures).}
	\label{fig:example of unstart due to stochastic inflow}
\end{figure}

In this paper, the primary goal is to consider the probability of the unstart due to the stochastic inflow 
modeled as (\ref{eq:inflow Mach number}). We are especially interested in the case that the parameter 
$\epsilon$ in (\ref{eq:inflow Mach number}) is a small positive value, which lead to the large deviation 
analysis in the next section.

%
\section{Large Deviation Principle}
\label{sec:general LDP}

In this section, we first briefly review the classical Freidlin-Wentzell theory, the large deviation 
principle (LDP) for the stochastic differential equation. Then the analogous LDP for the Euler scheme of the 
discrete problem is introduced. The discrete problem provides a good understanding of the LDP for our 
unstart problem in the next section.

\subsection{The Freidlin-Wentzell Theory}
\label{sec:FW theory}

We consider the following stochastic differential equation (SDE):
\begin{equation}
	\label{eq:general SDE, continuous time}
	dX_t^\eps = \alpha(X_t^\eps)dt + \eps \sigma dW_t,
\end{equation}
where $X_0^\eps=x_0$ is deterministic, $W_t$ is the standard Brownian motion, and $\eps$ and $\sigma$ are 
positive constants. The standard theorem says that there exists a unique strong solution on any finite time 
interval $\left[0,T\right]$ if $\alpha:\RR\to\RR$ is uniformly Lipschitz continuous (see, for example, 
\cite{Oksendal2003}). We assume that $\bar{X}_t$ on $[0,T]$ with $\bar{X}_0=x_0$ is the unique solution of 
the following differential equation:
\[
	\frac{d}{dt}\bar{X}_t=\alpha(\bar{X}_t).
\]
It is well-known that as $\eps \to 0$, $X_t^\eps$ converges to $\bar{X}_t$ in probability, that is, 
for any $\delta>0$,
\[
	\lim_{\eps \to 0}\PP \left(\max_{t\in[0,T]}|X_t^\eps-\bar{X}_t|>\delta\right)=0.
\]

The Freidlin-Wentzell theory says that as $\eps \to 0$, the asymptotic probability that $X_t^{\eps}$ 
deviates from $\bar{X}_t$ can be computed as follows:
\begin{align*}
	-\inf_{x_t\in \mathring{\calA}} \calI(x_t) 
	& \leq\liminf_{\eps \to 0} \frac{1}{\eps^2} \log \PP(X_t^\eps\in \calA)\\
	& \leq\limsup_{\eps \to 0} \frac{1}{\eps^2} \log \PP(X_t^\eps\in \calA) 
	\leq -\inf_{x_t\in\bar{\calA}}\calI(x_t),
\end{align*}
where the \textit{rare event} $\calA$ is any subset of $C_{x_0}([0,T])$, the space of continuous paths on 
$[0,T]$ with the starting point $x_0$ endowed with the standard topology. $\mathring{\calA}$ and 
$\bar{\calA}$ stand for the interior and closure of $\calA$, respectively. The \textit{rate function} 
$\calI$ measures how much $X_t$ deviates from its typical behavior in the $L^2$ sense:
\[
	\calI(x_t)=
	\begin{cases}
		\frac{1}{2\sigma^2}\int_0^T [\frac{d}{dt}x_t-\alpha(x_t)]^2 dt, 
		&x_t=x_0+\int_0^ty(s)ds,~ y\in L^2([0,T]),\\
		+\infty, &\text{otherwise.}
	\end{cases}
\]

\textbf{Remark.} For the proof of the Freidlin-Wentzell theory, readers may refer to 
\cite[Section 5.6]{Dembo2010}. In fact, the complete version of the Freidlin-Wentzell theory works for any 
finite dimensional SDEs and $\sigma$ can be a function of $X_t$. However, because we simply model the inflow 
Mach number as a Brownian motion, we use the current version to keep the expression of the rate function 
simple.

An immediate observation is that if $\calA$ is \textit{regular}, that is,  
\[
	\inf_{x_t\in \mathring{\calA}}\calI(x_t) 
	= \inf_{x_t\in\calA}\calI(x_t) 
	= \inf_{x_t\in\bar{\calA}}\calI(x_t),
\]
then we can represent the asymptotic probability in the following way:
\[
	\PP(X_t^\eps\in \calA)\sim\exp\left(-\frac{1}{\eps^2}\inf_{x_t\in \calA}\calI(x_t)\right)
\]
for small $\eps$. In addition, if $\bar{X}_t\in \calA$, then $\PP(X_t^\eps\in \calA)\sim 1$ as $\eps\to 0$. 
Indeed, because $X_t^\eps\to\bar{X}_t$ as $\eps\to 0$, any regular event containing $\bar{X}_t$ should be of 
probability one as $\eps\to 0$.

\subsection{Large Deviations for the Euler Scheme}
\label{sec:LDP for Euler}

For the computational purpose, one can use the Freidlin-Wentzell theory to obtain the analytical rate 
function, then discretizes the rate function and the functional space, and finally obtains the probability 
by solving the numerical optimization. However, it is more informative to consider the discrete problem by 
the Euler method and derive the exact LDP for the discrete problem. Then this LDP can be solved directly by 
the numerical optimization.

We consider the following difference equation by the Euler scheme:
\begin{equation}
	\label{eq:general SDE, discrete time}
	X_{n+1}^\eps = X_n^\eps + \alpha(X_n^\eps)\Delta t + \eps \sigma \Delta W_{n+1},
\end{equation}
where $X_0^\eps=x_0$ is deterministic, $\{\Delta W_{n+1}\}_{n=0}^{N-1}$ are independent Gaussian random 
variables with mean zero and variance $\Delta t$, and $\eps>0$. We assume that $T=N\Delta t$. The joint 
density $p$ of $(X_1^\eps,\ldots,X_N^\eps)$ can be derived by the Markov property:
\begin{align*}
	&p(X_N^\eps=x_N, \ldots, X_1^\eps=x_1) 
	= \prod_{n=0}^{N-1} h(X_{n+1}^\eps=x_{n+1}|X_n^\eps=x_n)\\
	&\quad = \prod_{n=0}^{N-1} \frac{1}{\sqrt{2\pi\eps^2\sigma^2\Delta t}}
	\exp\left(-\frac{1}{2\eps^2\sigma^2\Delta t} (x_{n+1}-x_n-\alpha(x_n)\Delta t)^2\right)\\
	&\quad = (2\pi\eps^2\sigma^2\Delta t)^{-N/2} \exp\left(-\frac{\Delta t}{2\eps^2\sigma^2}
	\sum_{n=0}^{N-1}\left(\frac{x_{n+1}-x_n}{\Delta t}-\alpha(x_n)\right)^2\right).
\end{align*}
Let $x=(x_1,\ldots,x_N)$ and 
\[
	I(x)=\frac{\Delta t}{2\sigma^2}\sum_{n=0}^{N-1}\left(\frac{x_{n+1}-x_n}{\Delta t}-\alpha(x_n)\right)^2.
\]
Then given a set $\bfA\subset\RR^N$ the probability that $(X_1^\eps,\ldots,X_N^\eps)\in \bfA$ is
\begin{equation}
	\label{eq:exact probability by integrating I}
	\PP((X_1^\eps,\ldots,X_N^\eps)\in \bfA)
	= \int_\bfA (2\pi\eps^2\sigma^2\Delta t)^{-N/2} \exp\left(-\frac{1}{\eps^2}I(x)\right) dx.
\end{equation}

We use Laplace's method to compute the asymptotic probability as $\eps\to 0$. The basic idea is that as 
$\eps\to 0$, the mass of the integrand in (\ref{eq:exact probability by integrating I}) will concentrate at 
its maximizer, which is the minimizer of $I$. We can therefore use the minimum to compute the probability.
\begin{theorem}
	\label{thm:Laplace's method}
	(Laplace's method) Assume that $x^*=\arg\min_{x\in \bfA}I(x)$ with $x^*\in\mathring{\bfA}$, and $I(x)$ 
	grows at least quadratically. Then 
	\[
		\lim_{\eps\to 0}\eps^2\log\PP((X_1^\eps,\ldots,X_N^\eps)\in \bfA) = -I(x^*).
	\]
\end{theorem}
\begin{proof}
	Because $I(x)$ grows at least quadratically, for any $\delta>0$, we can find a sufficiently large ball 
	$B_d=\{x\in\RR^N:\|x\|\leq d\}$ such that 
	\[
		\int_{B_d^C}(2\pi\eps^2 \Delta t)^{-N/2} \exp\left(-\frac{1}{\eps^2}I(x)\right)dx < \delta
	\]
	for all sufficiently small $\eps$. To compute the upper bound, it hence suffices to prove the case that 
	$\bfA$ is bounded. We rewrite the integral as 
	\[
		\exp\left(-\frac{1}{\eps^2}I(x^*)\right)
		\int_\bfA (2\pi\eps^2 \Delta t)^{-N/2} \exp\left(-\frac{1}{\eps^2}[I(x)-I(x^*)]\right)dx.
	\]
	The upper bound is 
	\begin{align*}
		&\eps^2 \log\PP((X_1^\eps,\ldots,X_N^\eps)\in \bfA)\\
		&\quad = -I(x^*) + \eps^2 \log(2\pi\eps^2\Delta t)^{-N/2} 
		+ \eps^2\log\int_\bfA\exp\left(-\frac{1}{\eps^2}[I(x)-I(x^*)]\right)dx\\
		&\quad \leq -I(x^*) + \eps^2\log(2\pi\eps^2\Delta t)^{-N/2} + \eps^2\log|\bfA|.
	\end{align*}
	Then we have the upper bound of the limit 
	\[
		\lim_{\eps\to 0}\eps^2\log\PP((X_1^\eps,\ldots,X_N^\eps)\in \bfA) \leq -I(x^*).
	\]
	
	To show the lower bound, we use the assumption that $x^*\in \mathring{\bfA}$. By the continuity of $I$, 
	for any $\delta>0$, there exists a neighborhood $N_\delta\subset \mathring{\bfA}$ of $x^*$ such that 
	$0\leq I(x)-I(x^*)\leq\delta$ for $x\in N_\delta$. Then we have the lower bound: 
	\begin{align*}
		&\eps^2\log\PP((X_1^\eps,\ldots,X_N^\eps)\in \bfA)
		\geq \eps^2 \log\PP((X_1^\eps,\ldots,X_N^\eps)\in N_\delta)\\
		&\quad = -I(x^*) + \eps^2\log(2\pi\eps^2\Delta t)^{-N/2} 
		+ \eps^2\log\int_{N_\delta}\exp\left(-\frac{1}{\eps^2} [I(x)-I(x^*)]\right)dx\\
		&\quad \geq -I(x^*) + \eps^2 \log(2\pi\eps^2\Delta t)^{-N/2} 
		+ \eps^2 \log(-|N_\delta|\frac{\delta}{\eps^2}).
	\end{align*}
	Therefore the lower bound is also obtained:
	\[
		\lim_{\eps\to 0}\eps^2 \log\PP((X_1^\eps,\ldots,X_N^\eps)\in \bfA) \geq -I(x^*).
	\]
\end{proof}

Readers can find that $I(x)\to \calI(x_t)$ as $\Delta t\to 0$. In fact, the following lemma says that 
$X^\eps_n$ is an exponentially good approximation of $X^\eps_t$.

\begin{lemma}
	\label{lma:exponential equivalence}
	(Exponential equivalence \cite[Lemma 5.6.9]{Dembo2010}) If $W_t$ in 
	(\ref{eq:general SDE, continuous time}) and $\Delta W_{n+1}$ in 
	(\ref{eq:general SDE, discrete time}) satisfy
	\[
		\Delta W_{n+1} = \int_{n\Delta t}^{(n+1)\Delta t} dW_s = W_{(n+1)\Delta t} - W_{n\Delta t}
	\]
	almost surely, then for any $\delta>0$,
	\[
		\lim_{\Delta t\to 0} \limsup_{\eps\to 0} 
		\log\PP\left(\max_n\sup_{t\in[n\Delta t,(n+1)\Delta t]} |X^\eps_t-X^\eps_n|>\delta\right) = -\infty.
	\]
\end{lemma}

Lemma \ref{lma:exponential equivalence} also tells us that we can effectively simulate a rare event by the 
Euler method. The standard theory shows that the Euler method has the order of accuracy $\sqrt{\Delta t}$; 
however, by the large deviations, the probability of a rare event is of order $\exp(-1/\eps^2)$ for small 
$\eps$. Therefore we might need to decrease $\Delta t$ exponentially in $\eps$ to obtain the correct 
simulation, and such exponential discretization will make the simulation computationally impossible. Thanks 
to Lemma \ref{lma:exponential equivalence}, we only need to choose a uniform $\Delta t$ and the Euler method 
can accurately simulate a rare event for all small $\eps$.
\section{Large Deviations for the Unstart}
\label{sec:LDP for unstart}

In this section, we derive the large deviation principle for the unstart of the scramjet. First we formulate 
the analytical problem. Readers can immediately identify that the analytical problem is an application of 
the Freidlin-Wentzell theory. However, as this problem can only be solved numerically, we also need the LDP 
for the discretized problem. In other words, we derive the LDP for the numerical PDE of the flow equation. 
The readers can also find that the theory is almost ready because of the derivation in the last section.

\subsection{Analytical Model}

Recall that the flow equation in the scramjet engine is 
\begin{equation}
	\label{eq:flow equation in the engine}
	\u_t + \f_x 
	= \frac{A'(x)}{A(x)} \left( \begin{pmatrix} 0 \\ P \\ 0 \\ \end{pmatrix} - \f\right) 
	+ \begin{pmatrix} 0 \\ 0 \\ f(x,t) \\ \end{pmatrix},
\end{equation}
for $x\in[-L_I,L_C+L_E]$ and $t\in[0,T]$. To solve this PDE, one also needs to specify the initial condition
and the inflow boundary condition. We assume that a suitable, deterministic initial condition is given. For 
simplicity, suppose that $P(t,-L_I)\equiv P_0$ and $\rho(t,-L_I)\equiv\rho_0$ for all $t\in[0,T]$; then the 
perturbation of the inflow Mach number $M_{in}(t)$ is entirely from that of the inflow speed 
$u_{in}(t):=u(t,-L_I)$ as $M=u/\sqrt{\gamma P/\rho}$. The inflow speed is modeled as a Brownian motion:
\[
	u_{in}(t) = u_{in}(0) + \eps \sigma_u W_t,
\]
with $u_{in}(0)=u_0=1300(m/s)$ that corresponds to Mach $2$ in our setting. Then from Section 
\ref{sec:FW theory}, $u_{in}(t)$ satisfies the large deviation principle with the rate function
\[
	\calI(u_{in})=
	\begin{cases}
		\frac{1}{2\sigma_u^2}\int_0^T \left(\frac{d}{dt}u_{in}(t)\right)^2 dt, 
		&u_{in}(t)=u_{in}(0)+\int_0^t y(s)ds,~ y\in L^2([0,T]),\\
		+\infty, &\text{otherwise.}
	\end{cases}
\]
by the Freidlin-Wentzell theory. The rare event $\calA$ is the set of all possible $u_{in}(t)$ that causes 
the unstart in the time horizon $[0,T]$, the event that the subsonic flow reaches the entrance of the 
isolator during $[0,T]$:
\[
	\calA = \{u_{in}(t):~ u_{in}(0)=u_{0},~ \exists t\in[0,T],~ x_{shock}(t)=-L_I\}.
\]
Then the probability of the unstart is obtained by the large deviation principle:
\[
	\PP(u_{in}\in\calA) \sim \exp\left(-\frac{1}{\eps^2}\inf_{u_{in}\in\calA}\calI(u_{in})\right).
\]
for small $\eps$. We note that in the optimization problem $\inf_{u_{in}\in\calA}\calI(u_{in})$, although 
the objective function $\calI(u_{in})$ is convex, it is very difficult to verify the convexity of the 
constraint set $\calA$, because it involves the analysis of the nonlinear hyperbolic system 
(\ref{eq:flow equation in the engine}).

\subsection{Numerical Model}

\subsubsection{Numerical PDE}
\label{sec:numerical PDE}

We uniformly discretize the space and time: $-L_I=x_0<\cdots<x_K=L_C+L_E$ and $0=t_0<\cdots<t_N=T$. As 
the rate function needs to be defined \textit{a priori}, we choose a uniform $\Delta t$ satisfying the CFL 
condition so that the numerical PDE method is stable when we solve the large deviation problem. By 
convention, $X^n_k$ denotes the average of the quantity X over the cell $(x_k,x_{k+1})$ at time $t_n$.
The local-Lax-Friedrichs (LLF) scheme is used to solve numerically the governing equation (\ref{eq:qce}). 
See also Appendix \ref{sec:component-wise LLF} for the details of the numerical PDE method.

%

\subsubsection{The Rate Function}

In the discrete case, similarly, we model the inflow speed $u_{in}(n)$ as a Gaussian random walk:
\begin{equation}
	\label{eq:full order u_in}
	u_{in}(n+1) = u_{in}(n) + \eps \sigma_u \Delta W_{n+1},\quad n=0,\ldots,N-1
\end{equation}
with $u_{in}(0)=u_0=1300(m/s)$. $\{\Delta W_{n+1}\}_{n=0}^{N-1}$ are independent Gaussian random variables 
with mean zero and variance $\Delta t$. From Section \ref{sec:LDP for Euler}, $u_{in}(n)$ satisfies the 
large deviation principle with the rate function:
\[
	I(u_{in})
	=\frac{\Delta t}{2\sigma_u^2}\sum_{n=0}^{N-1}\left(\frac{u_{in}(n+1)-u_{in}(n)}{\Delta t}\right)^2.
\]
We note that $I(u_{in})$ is a convex function in $u_{in}$.

\subsubsection{The Rare Event}

For the discrete problem, we define the unstart as the event that the subsonic flow is produced at 
$(x_1,x_2)$, the location next to the entrance of the isolator $(x_0,x_1)$. Hence the rare event $\bfA$ is 
the set of $u_{in}=(u_0^0,\ldots,u_0^N)$ causing the unstart on $[0,T]$:
\begin{equation}
	\label{eq:the set of unstart}
	\bfA = \{u_{in}=(u_0^0,\ldots,u_0^N):~ u_0^0=u_0,~ \min_{1\leq n\leq N} M_1^n \leq 1\},
\end{equation}
where $M_1^n$ is the Mach number on $(x_1,x_2)$ at time $t_n$ and is computed by the numerical PDE method.
Similar to the analytical case, although the rate function is convex, it is difficult to verify the 
convexity of $\bfA$.

\subsection{Numerical Optimization}

By the large deviation principle, the probability of the unstart is therefore 
\[
	\PP(u_{in}\in\bfA) \sim \exp\left(-\frac{1}{\eps^2}\inf_{u_{in}\in\bfA}I(u_{in})\right)
\]
for small $\eps$. We compute this probability by solving the nonlinear constrained optimization problem 
$\inf_{u_{in}\in\bfA}I(u_{in})$. We use the interior-point method (see \cite[Chapter 19]{Nocedal2006}) as 
the optimization algorithm.

Here we address an important issue about this optimization. Any gradient-based optimization algorithm, 
including the interior-point method, only obtains a local minimum, which might not be the global minimum. 
The global minimum is guaranteed only if the given problem is convex: both the objective function and the 
constraint set are convex. In this problem, although the objective function $I(u_{in})$ is convex, it is not 
clear if the constraint set $\bfA$ is. Thus in theory, we do not know if the answer we obtain is truly 
global. However, we solved this problem by multiple random initial guesses and obtained essentially the same 
result, therefore we have a high confidence that our answer is the global minimum.

\subsubsection{Dimension Reduction of the Optimization}

Another computational challenge of the optimization problem is its extremely large scale. The typical 
setting of this problem is as follows. The space domain $[-L_I, L_C+L_E]$ is discretized into $100$ points.
As the engine operates at the supersonic speed around Mach $2$, the CFL condition requires that 
$\Delta t \approx 10^{-6}$ seconds. If we simulate the model up to $0.01$ seconds, then we need $10^4$ 
points to discretize the time domain. Every time we evaluate the constraint, a complete run of the numerical 
PDE is performed, and therefore every iteration of the optimization algorithm is very expensive. 
Consequently, to optimize the rate function with $10^4$ variables is almost computationally impossible.

To speed up the optimization, we can reduce the degree of freedom of the inflow speed $u_{in}$. Indeed, 
although the CFL condition asks for the very fine grids in time, the resolution of the large deviation 
solution can be far lower than that. More precisely, we let $\tilde{N}\in\NN$ such that $m\tilde{N}=N$ for 
some $m\in\NN$, and $\tilde{u}_{in}(0), \tilde{u}_{in}(m), \ldots, \tilde{u}_{in}(N)$ are the Gaussian 
random walks:
\begin{equation}
	\label{eq:reduced order u_in}
	\tilde{u}_{in}((n+1)m) = \tilde{u}_{in}(nm) + \eps \sigma_u \Delta \tilde{W}_{n+1},\quad 
	n=0,\ldots,\tilde{N}-1,
\end{equation}
where $\{\Delta\tilde{W}_{n+1}\}_{n=0}^{\tilde{N}-1}$ are independent Gaussian random variables with mean 
zero and variance $m\Delta t$. The corresponding rate function is 
\begin{equation}
	\label{eq:reduced order rate function}
	I(\tilde{u}_{in})
	=\frac{m\Delta t}{2\sigma_u^2}\sum_{n=0}^{\tilde{N}-1}
	\left(\frac{u_{in}((n+1)m)-u_{in}(nm)}{m\Delta t}\right)^2.
\end{equation}
When we perform the numerical PDE, we linearly interpolate the other variables to obtain the inflow 
condition:
\begin{equation}
	\label{eq:linear interpolation of inflow speed}
	\tilde{u}_{in}(nm + k) 
	= \left(1-\frac{k}{m}\right)\tilde{u}_{in}(nm) 
	+ \frac{k}{m}\tilde{u}_{in}((n+1)m),\quad 1\leq k\leq m-1.
\end{equation}

The optimization cost is generally at least proportional to the number of variables. In our case, we choose 
$\tilde{N}=20$ and linearly interpolate the others, and therefore the algorithm only has to optimize over 
$20$ variables. The reduced problem is at least $500$ time faster than the full problem. A potential issue 
of the reduced problem is that if the result of the reduced problem is accurate enough. Our numerical 
experiment shows that $\tilde{N}=20$ is sufficient for this problem. In fact, in the later section, we find 
that even thought we double the resolution ($\tilde{N}=40$), the relative improvement is less than $1\%$.

\textbf{Remark.} By using the dimension reduction technique in this subsection, it is also possible to use 
the typical adaptive-time-discretization to solve the numerical PDE (\ref{eq:flow equation in the engine}) 
by considering the following inflow condition:
\[
	u_{in}(t) 
	= \left(1-\frac{t}{m}\right)\tilde{u}_{in}(nm) 
	+ \frac{t}{m}\tilde{u}_{in}((n+1)m),\quad t\in(nm\Delta t, (n+1)m\Delta t).
\]
However, our numerical experiments show that the constraint set $\bfA$ in this way is less smooth that of 
the uniform discretization and the resulting numerical optimization is also less robust. For this reason, we 
still use the uniform discretization and choose an appropriate $\Delta t$.

\subsubsection{Summary of the Numerical Optimization}

We finish this section by summarizing the numerical optimization.
\begin{algorithm}
	\label{alg:numerical LDP for unstart}
	\begin{enumerate}
		\item The goal of this section is to compute numerically the optimization problem 
		$\inf_{u_{in}\in\bfA}I(u_{in})$.
		
		\item The full problem is computationally expensive so we instead solve the reduced problem 
		$\inf_{\tilde{u}_{in}\in\bfA}I(\tilde{u}_{in})$.
		
		\item The interior-point method is the optimization algorithm. The objective function is given in 
		(\ref{eq:reduced order rate function}) and the numerical PDE of 
		(\ref{eq:flow equation in the engine}) is used to determine if $\tilde{u}_{in}\in \bfA$.
		
		\item Although our numerical experiments show that the optimization algorithm finds the same result even 
		with different random initial guesses, a good initial guess can greatly speed up the optimization 
		process. We let the initial guess $\{\tilde{u}_{in}(n)\}_{n=0}^N$ linear in $n$ and choose 	
		$\tilde{u}_{in}(N)$ to be the largest value so that $\{\tilde{u}_{in}(n)\}_{n=0}^N$ triggers the 
		unstart.
	\end{enumerate}
\end{algorithm}

\section{Numerical Results of the Large Deviations}
\label{sec:numerical result of LD}

In this section, we show the numerical results of the large deviation problem 
$\inf_{\tilde{u}_{in}\in\bfA}I(\tilde{u}_{in})$. The values of the parameters are listed in Appendix 
\ref{sec:parameters}.

%
%
%

\subsection{The Event of Subsonic Inflows}
\label{sec:subsonic inflows}

Before we go to the details of the large deviation result. We address an important subset of $\calA$ that 
causes the unstart: the situation that the scramjet receives a subsonic inflow. We define 
\begin{equation}
	\label{eq:the set of subsonic inflow}
	\calB = \{u_{in}(t):~ u_{in}(0)=u_0,~ 
	\min_{t\in[0,T]}M_{in}(t) = \min_{t\in[0,T]} u_{in}(t)/\sqrt{\gamma P_0/\rho_0} \leq 1\}.
\end{equation}
We note that $\calB\subset\calA$ and thus
$\inf_{u_{in}\in\calA}\calI(u_{in})\leq\inf_{u_{in}\in\calB}\calI(u_{in})$. Because $u_{in}(t)$ is modeled 
as a Brownian motion, $\inf_{u_{in}\in\calB}\calI(u_{in})$ can be solved explicitly:
\begin{equation}
	\label{eq:optimal path for a Brownian motion}
	u_{in}^* = \arg\inf_{u_{in}\in\calB}\calI(u_{in}) 
	= \left(1-\frac{t}{T}\right)u_0 + \frac{t}{T}\sqrt{\gamma P_0/\rho_0}
	= \left(1-\frac{t}{T}\right)u_0 + \frac{t}{2T}u_0.
\end{equation}
In other words, the minimizer $u_{in}^*$ is a straight line starting from $u_0$ and ending at $0.5 u_0$, and 
equivalently, the minimizer $M_{in}^*$ is a straight line starting from $2$ and ending at $1$. Further, we 
obtain an upper bound for $\inf_{u_{in}\in\calA}\calI(u_{in})$:
\begin{equation}
	\label{eq:upper bound for the rate function}
	\inf_{u_{in}\in\calA}\calI(u_{in}) 
	\leq \inf_{u_{in}\in\calB}\calI(u_{in})
	= \calI(u_{in}^*)
	= \frac{1}{2\sigma_u^2}\int_0^T \left(\frac{d}{dt}u_{in}^*(t)\right)^2 dt 
	= \frac{u_0^2}{8\sigma_u^2 T}.
\end{equation}
Thus we can conclude that when $\inf_{\tilde{u}_{in}\in\bfA}I(\tilde{u}_{in})$ is very close to this upper 
bound, the scramjet operates in a very stable region, because the probability of the unstart is as low as 
the theoretical limit.

\textbf{Remark.} One may analogously want to define $\bfB$ in the discrete sense. However, in general $\bfB$ 
is not necessarily a subset of $\bfA$ because in the definition of $\bfA$ (see 
(\ref{eq:the set of unstart})), what we check is if the Mach number in the second cell $M_1^n$ is less than 
$1$, but the inflow Mach number $M_{in}(n)=M_0^n \leq 1$ does not imply $M_1^n \leq 1$. In fact one often 
needs $M_{in}(n)=M_0^n \leq 1-\eta$ with a small positive $\eta$ to have $M_1^n \leq 1$. Therefore in the 
later numerical results, readers will find that the computed $\inf_{\tilde{u}_{in}\in\bfA}I(\tilde{u}_{in})$ 
in some cases will slightly higher than the continuum upper bound $u_0^2/(8\sigma_u^2 T)$. Obviously this 
discrepancy will be reduced as we refine the discretization, and we still use this upper bound 
(\ref{eq:upper bound for the rate function}) as a reference value in our numerical results.

\subsection{Impact of Fuel Cycles}
\label{sec:impact of fueling}

The function of the fueling control is 
\[
	f(t,x) = c\cdot\phi\cdot f(t)\cdot f(x),
\]
where $c$ is a constant and $f(x)$ is a fixed spatial function. $\phi$ is the mixing ratio and $f(t)$ is the 
control of the fuel cycle:
\[
	f(t)
	=\begin{cases}
		1, & t\in[i\tau,i\tau+b],\quad i\in\NN,\\
		0, & \text{otherwise.}
	 \end{cases}
\]
In every fuel cycle of length $\tau$ seconds, the fuel is injected in the first $b$ seconds.

Two representative fueling profiles are used: $\phi_S=0.78$, $\tau_S=0.5$ms, $b_S=0.1$ms and $\phi_L=0.78$,
$\tau_L=2$ms, $b_L=0.4$ms. These two profiles are extensively studied in \cite{West2011}. Because 
$\phi_S=\phi_L$ and $b_S/\tau_S=b_L/\tau_L$, the two profiles generate roughly the same amount of thrust. 
However, the numerical experiments in \cite{West2011} show that the first profile is more robust to prevent 
the engine from stalling. We observe the same qualitative behavior in the large deviation case.

\begin{figure}
	\centering
	\includegraphics[width=0.49\textwidth]{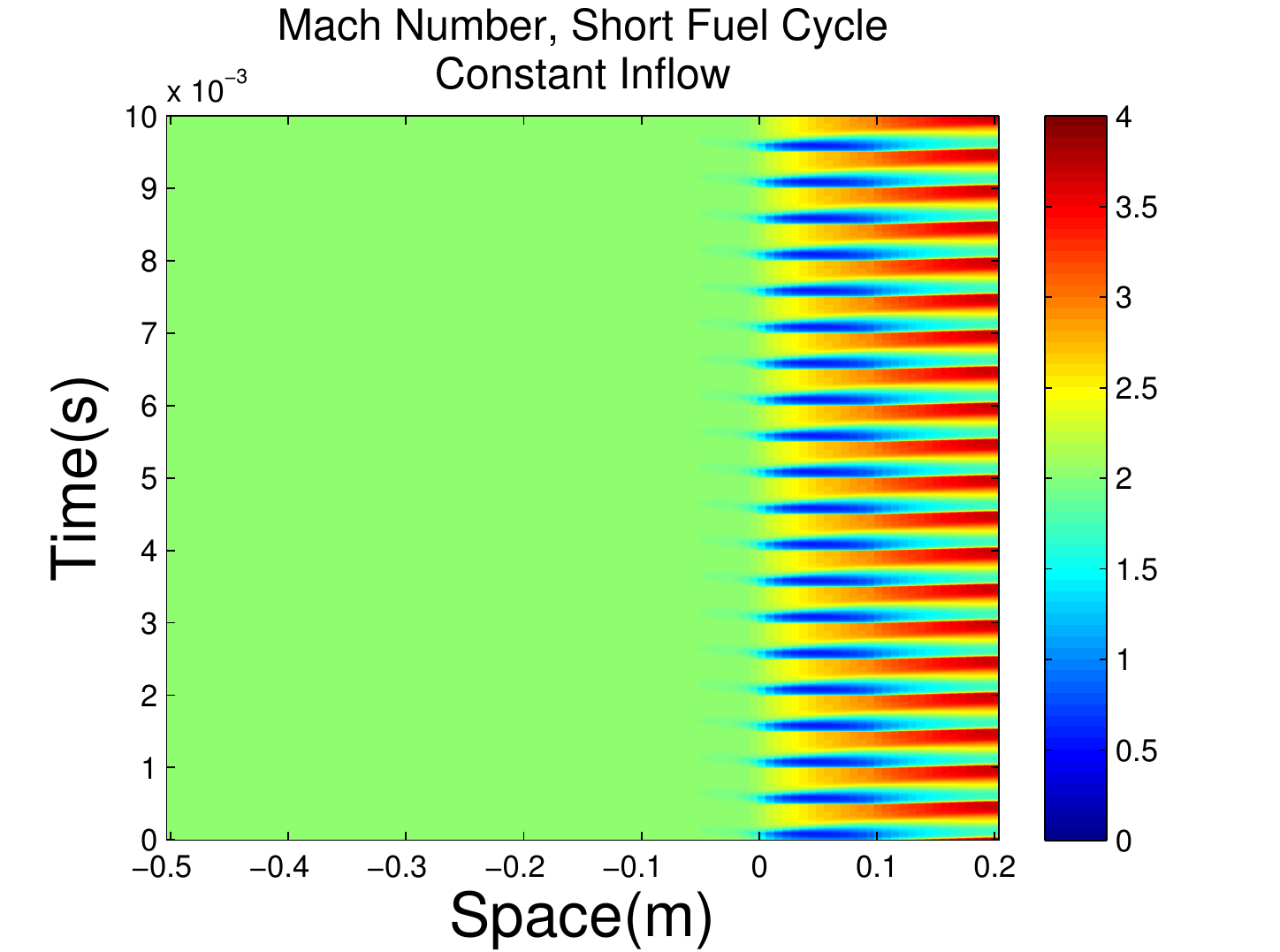}
	\includegraphics[width=0.49\textwidth]{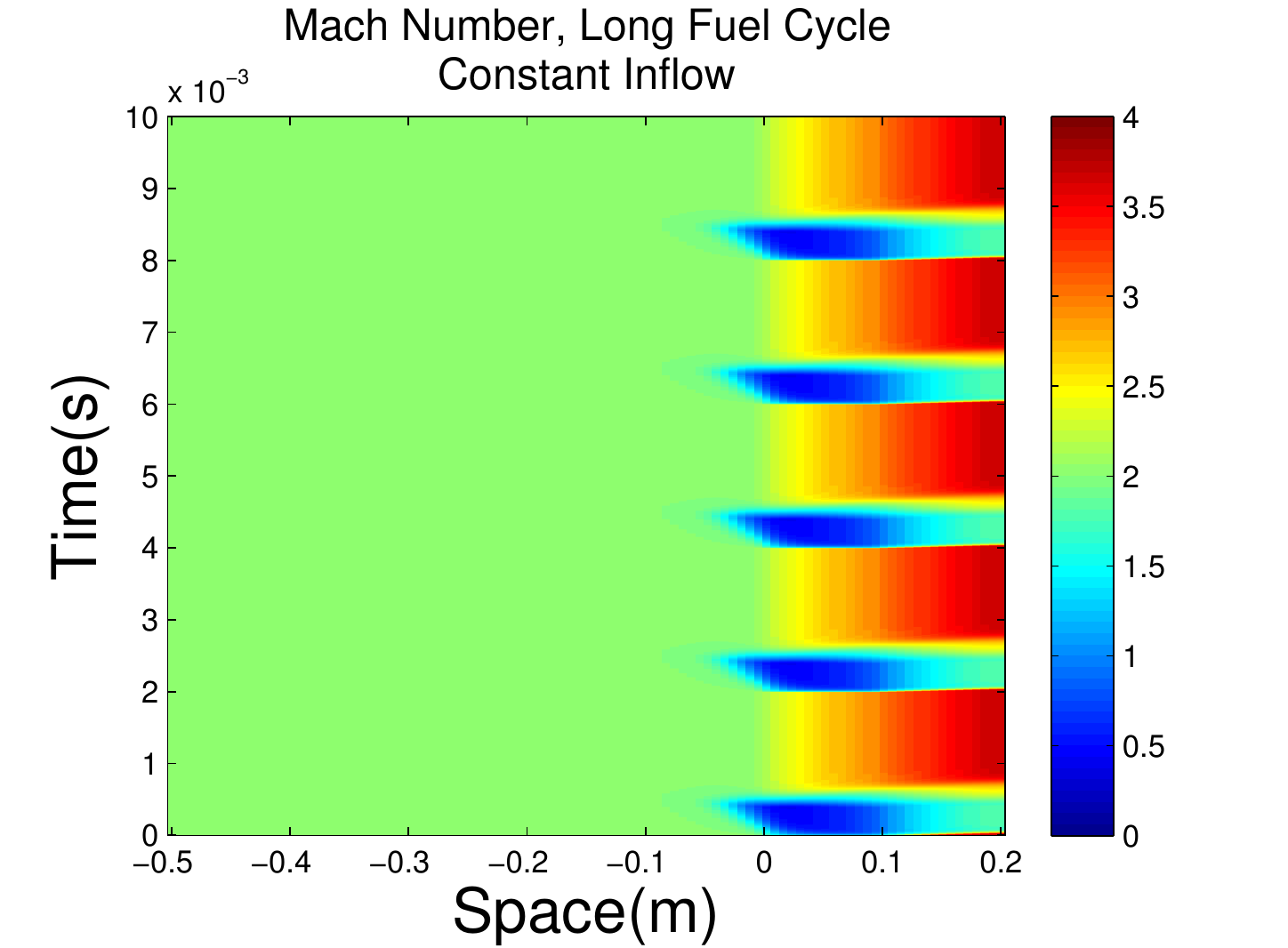}
	\caption{
	\label{fig:reference figure with constant inflow}
	The Mach numbers in the scramjet engine when a constant inflow (Mach $2$) is used. The engine operates 
	normally for both short and long fuel cycles.}
\end{figure}

\begin{figure}
	\centering
	\includegraphics[width=0.49\textwidth]{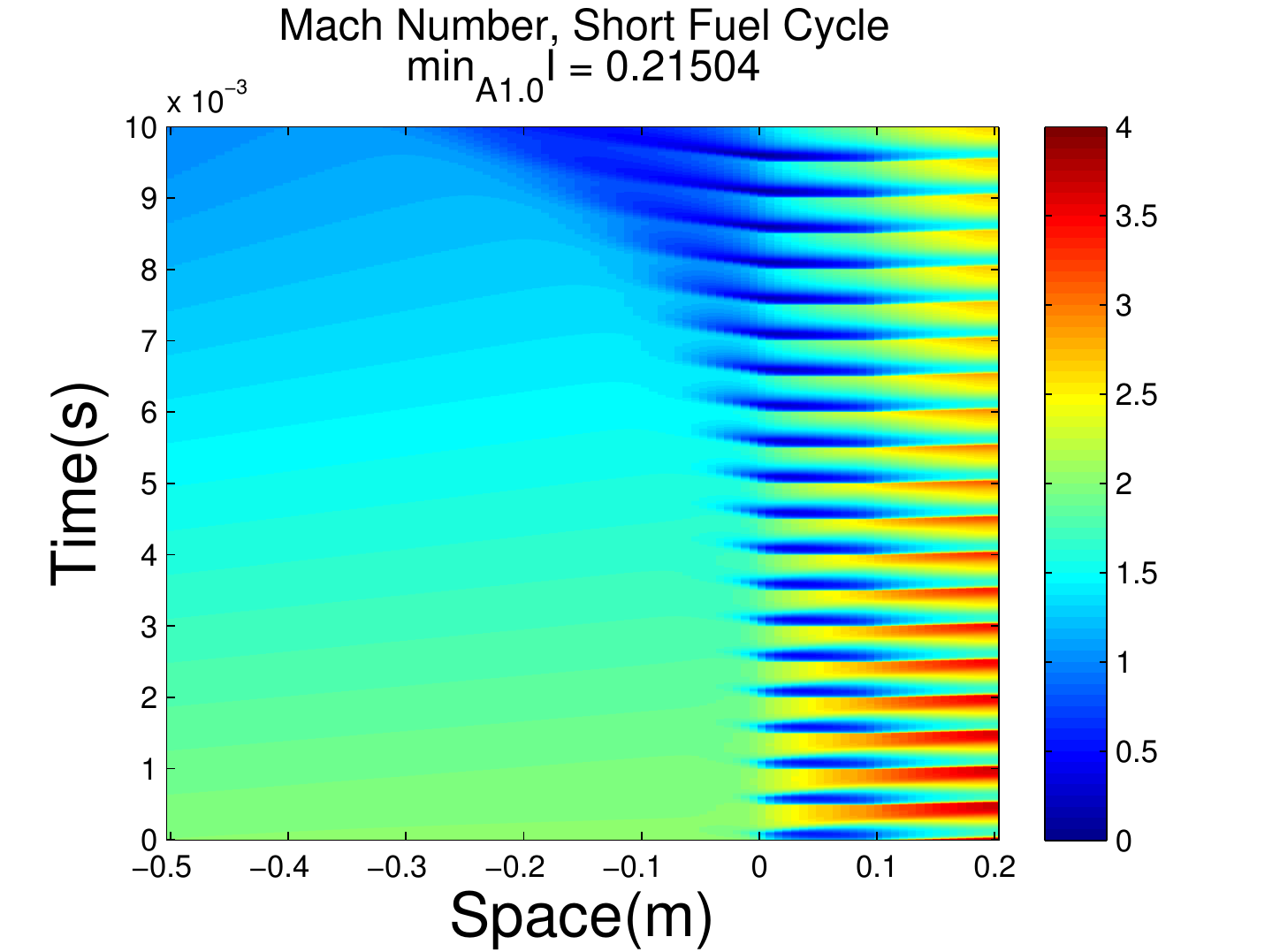}
	\includegraphics[width=0.49\textwidth]{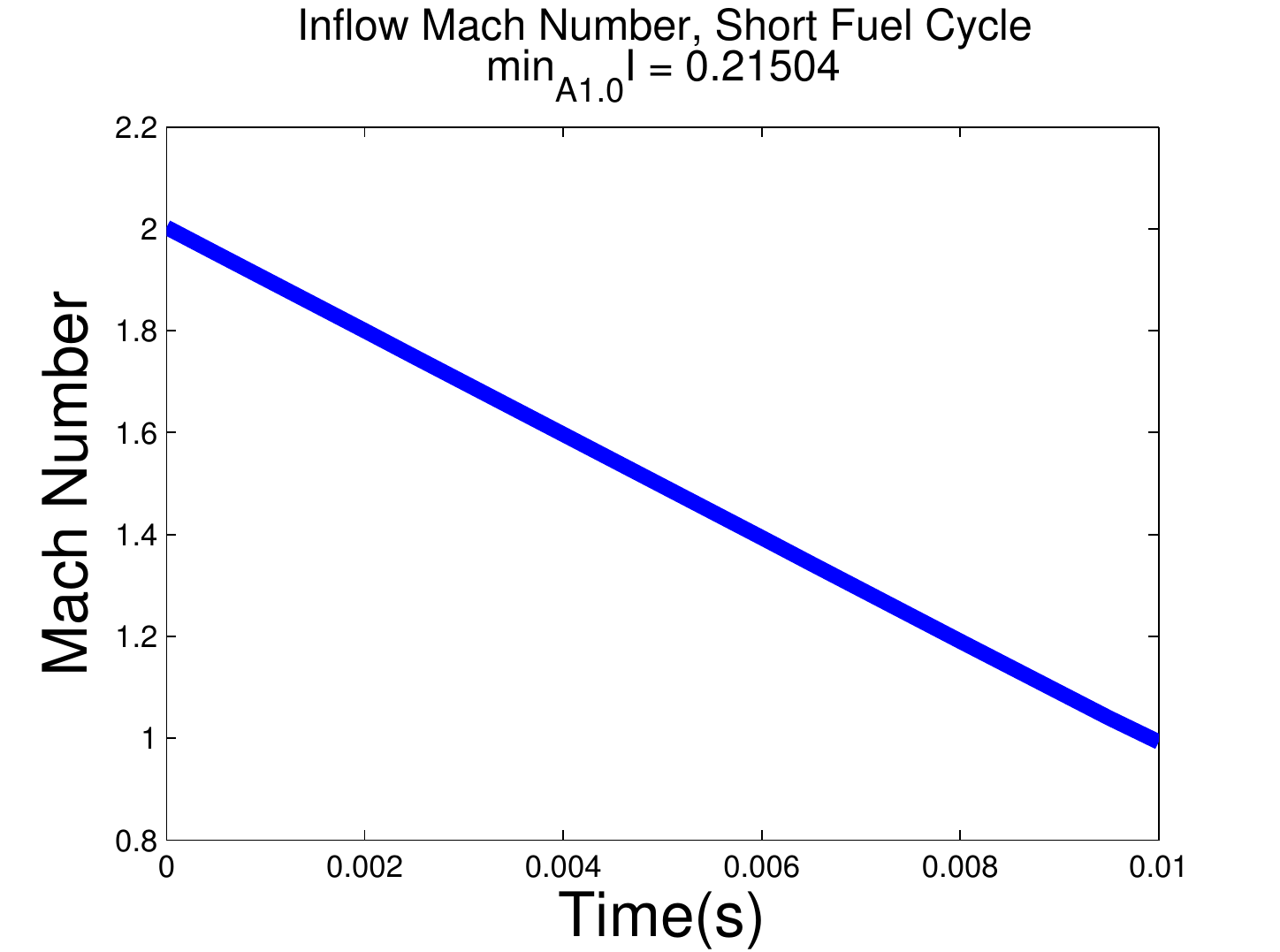}

	\includegraphics[width=0.49\textwidth]{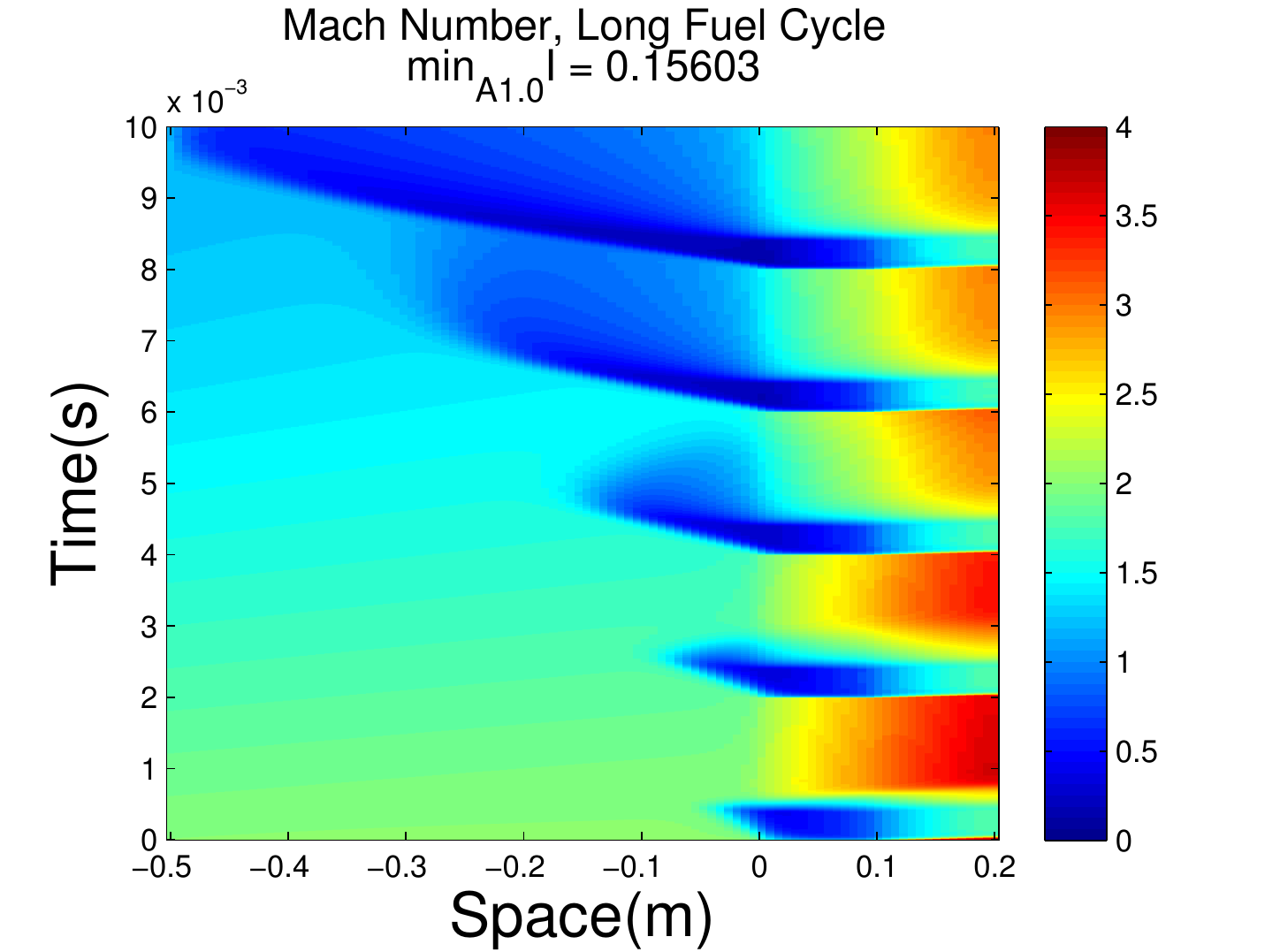}
	\includegraphics[width=0.49\textwidth]{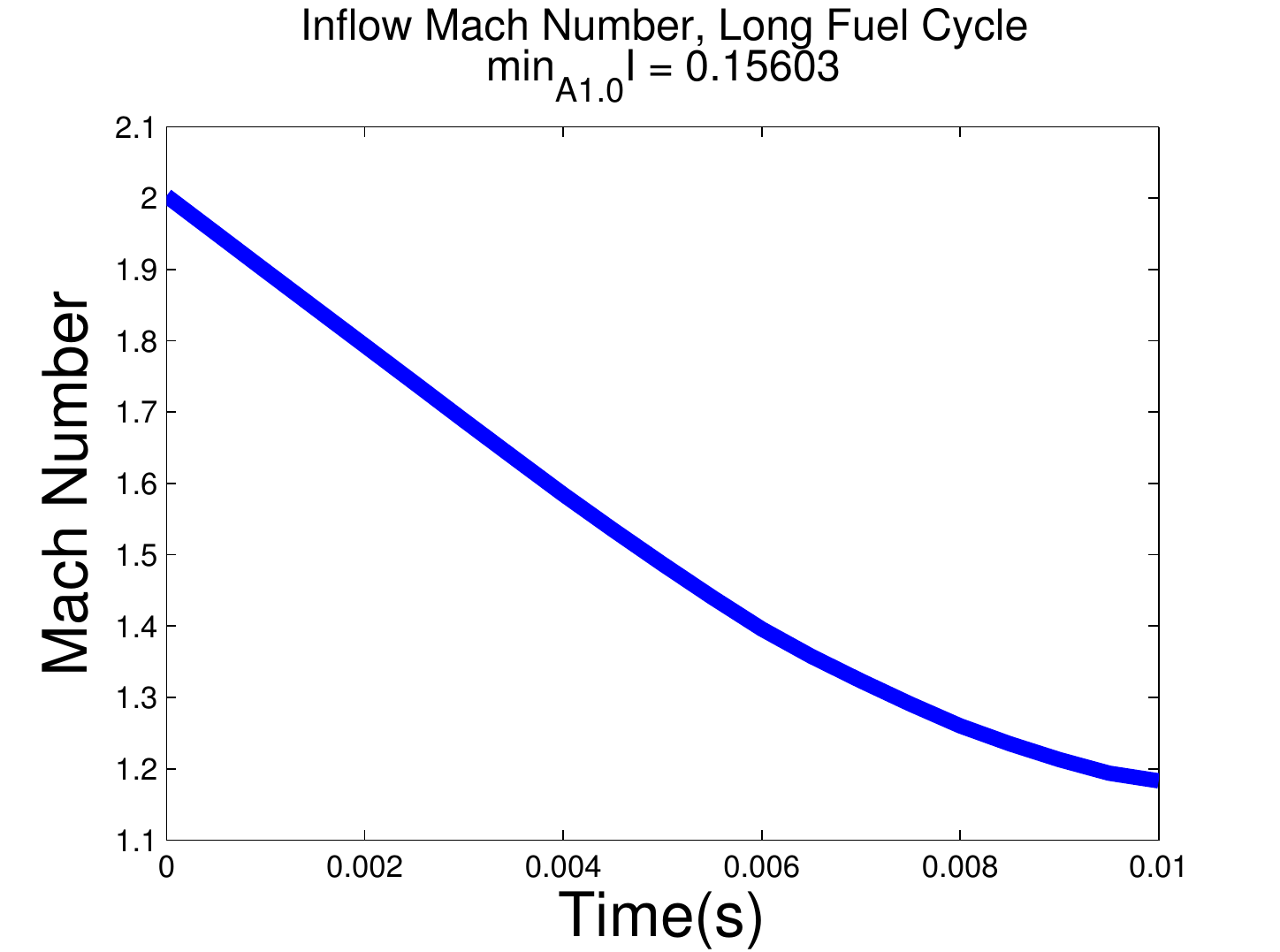}
	\caption{
	\label{fig:impact of fueling}
	The large deviation solutions of two different fuel cycles. The optimal $I$ of the short fuel 
	cycle ($\phi_S$, $\tau_S$, $b_S$) is higher than that of the long fuel cycle ($\phi_L$, $\tau_L$, 
	$b_L$). It means the short fuel cycle is a more stable strategy.}
\end{figure}

\begin{table}
	\centering
	\begin{tabular}{|c|c|c|c|}
		\hline 
		& Short Fuel Cycle & Long Fuel Cycle & $\calI(u_{in}^*)$ \tabularnewline
		\hline
		$\inf_{\tilde{u}_{in}\in\bfA}I(\tilde{u}_{in})$ 
		& $0.21504$ & $0.15603$ & $0.21125$ \tabularnewline
		\hline 
	\end{tabular}
	\caption{
	\label{tab:the optimal rate functions for different fuel cycles}
	The optimal values of the rate function for the short and long fuel cycles.}
\end{table}

We can see in Figure \ref{fig:reference figure with constant inflow} that with the constant inflow (Mach 
$2$), the engine operates operates normally for both fuel cycles. No subsonic flows reach the entrance
of the isolator. In this section, Figure \ref{fig:reference figure with constant inflow} will serve as the 
reference case representing the normal situation.

When the inflow perturbations are considered, we solve the large deviation problem 
$\inf_{\tilde{u}_{in}\in\bfA}I(\tilde{u}_{in})$. The solutions are plotted in Figure 
\ref{fig:impact of fueling} (left). The optimal $I$ of the short fuel cycle ($0.21504$) is about $1/3$ 
higher than that of the long fuel cycle ($0.15603$). It means that when $\epsilon$ is small, the 
probability of the unstart with the short fuel cycle case is lower than that with the long fueling case. 
This observation is qualitatively consistent with the study in \cite{West2011} by the Monte Carlo 
simulations.

Figure \ref{fig:impact of fueling} also tells us an interesting information: with the short fuel cycle, 
the minimizer for $\inf_{\tilde{u}_{in}\in\bfA}I(\tilde{u}_{in})$ is very close to a straight line starting 
from $2$ and ending at $1$, which is the minimizer $u_{in}^*(t)$ for the large deviation problem 
$\inf_{u_{in}\in\calB}\calI(u_{in})$ discussed in Section \ref{sec:subsonic inflows}. In addition, its 
optimal value of the rate function ($0.21504$) is also close to the continuum upper bound ($0.21125$), which 
says that with the short fuel cycle, the scramjet has very strong resistance to the unstart.

\subsection{Sensitivity to Constraints}
\label{sec:margin}

We can slightly modify the constraint set $\bfA$ to see the sensitivity of the large deviation problem to 
the change of the constraint. More precisely, we define 
\begin{align*}
	\bfA_{0.8} &= \{u_{in}=(u_0^0,\ldots,u_0^N):~ u_0^0=u_0,~ \min_{1\leq n\leq N} M_1^n \leq 0.8\},\\
	\bfA_{1.2} &= \{u_{in}=(u_0^0,\ldots,u_0^N):~ u_0^0=u_0,~ \min_{1\leq n\leq N} M_1^n \leq 1.2\},
\end{align*}
and let $\bfA_{1.0}:=\bfA$ in (\ref{eq:the set of unstart}). It is reasonable to argue that 
$\bfA_{0.8}\subset\bfA_{1.0}\subset\bfA_{1.2}$ (by assuming that the Mach number is continuous in time.) 
Similarly, we can also define $\calB_{0.8}$, $\calB_{1.0}$ and $\calB_{1.2}$ according to 
(\ref{eq:the set of subsonic inflow}), and derive the continuum upper bounds by the formulas similar to 
(\ref{eq:optimal path for a Brownian motion}) and (\ref{eq:upper bound for the rate function}) (see also 
Table \ref{tab:the optimal rate functions over different A}).

\begin{figure}
	\centering
	\includegraphics[width=0.49\textwidth]{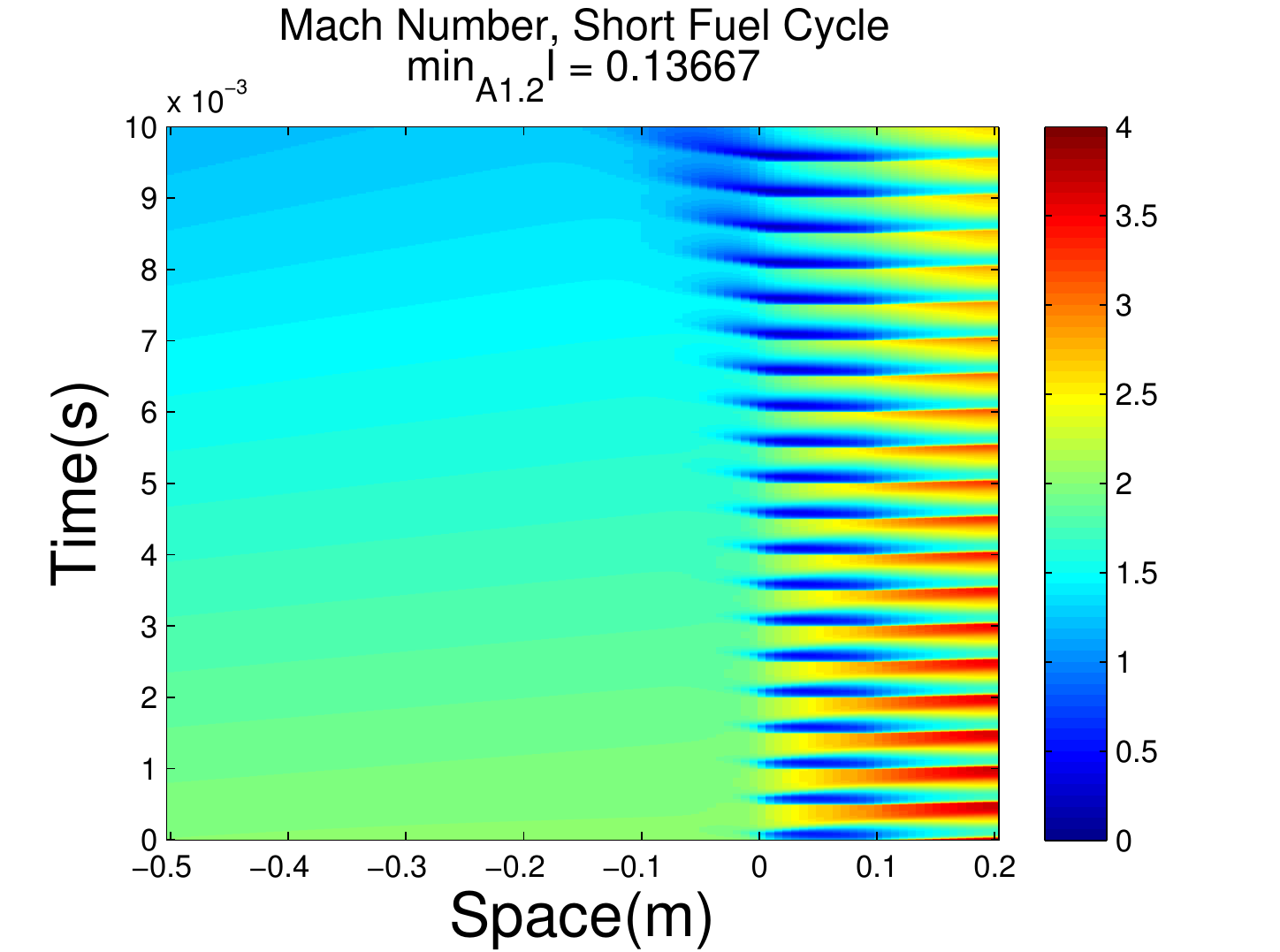}
	\includegraphics[width=0.49\textwidth]{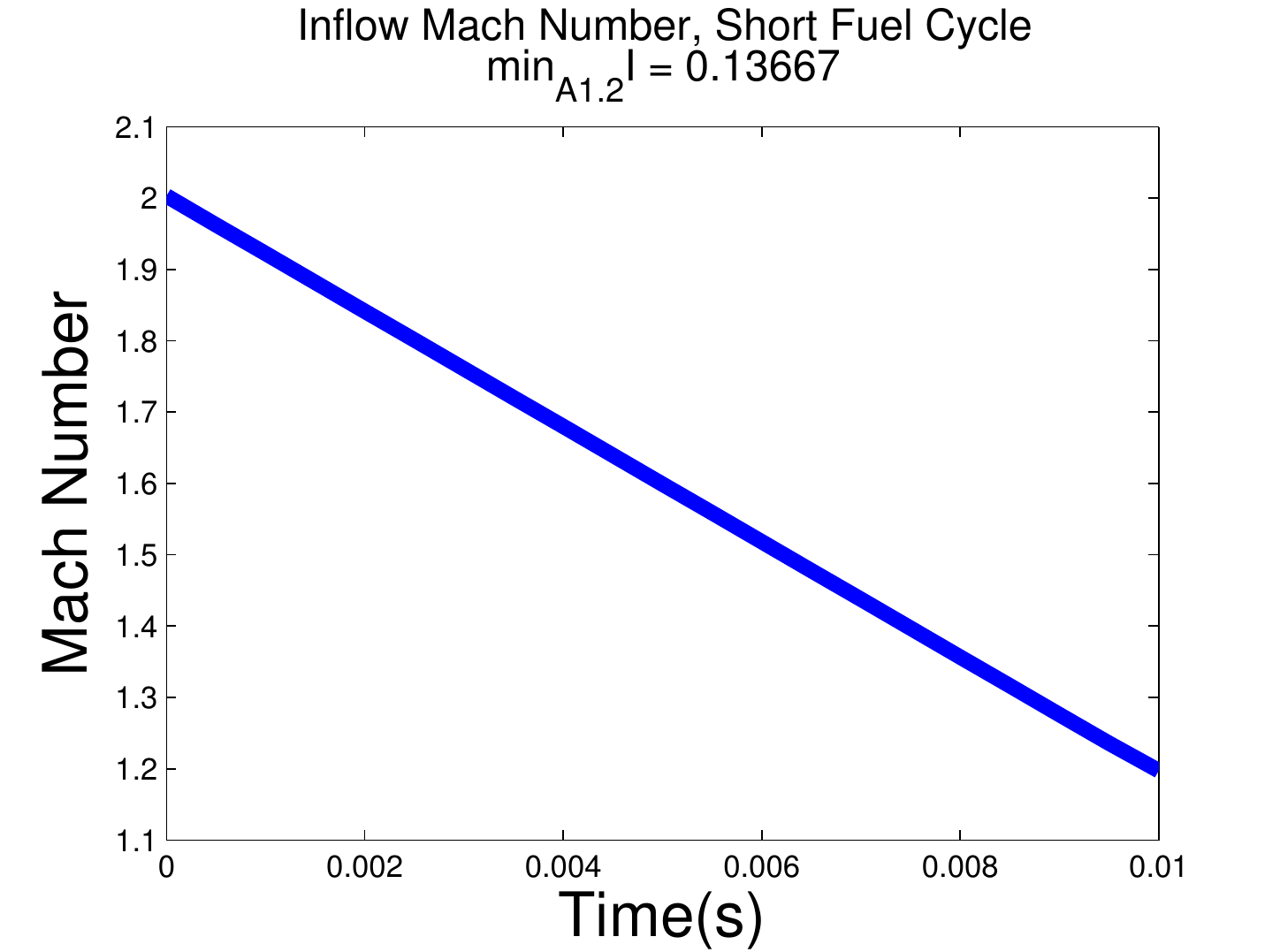}
	
	\includegraphics[width=0.49\textwidth]{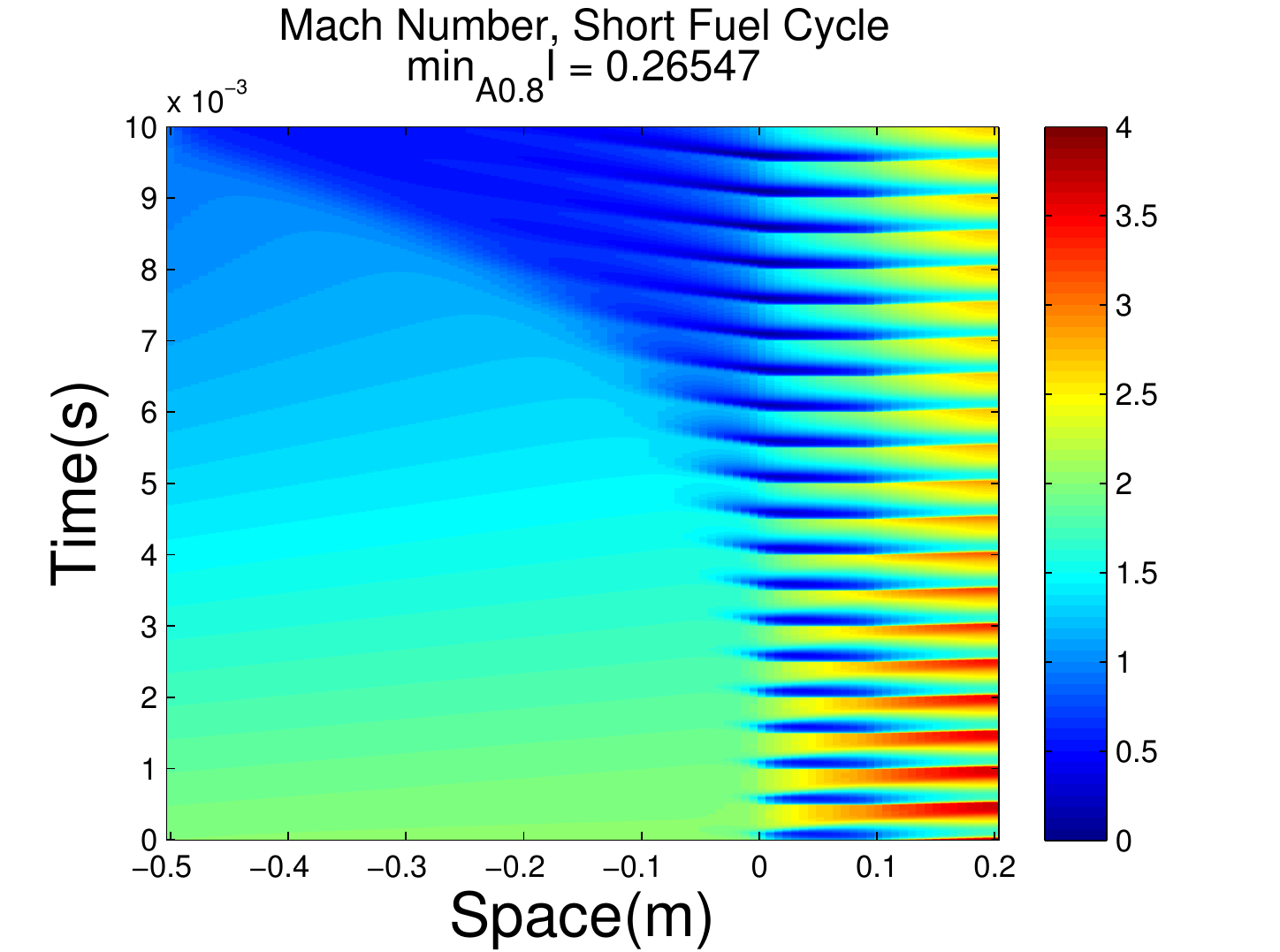}
	\includegraphics[width=0.49\textwidth]{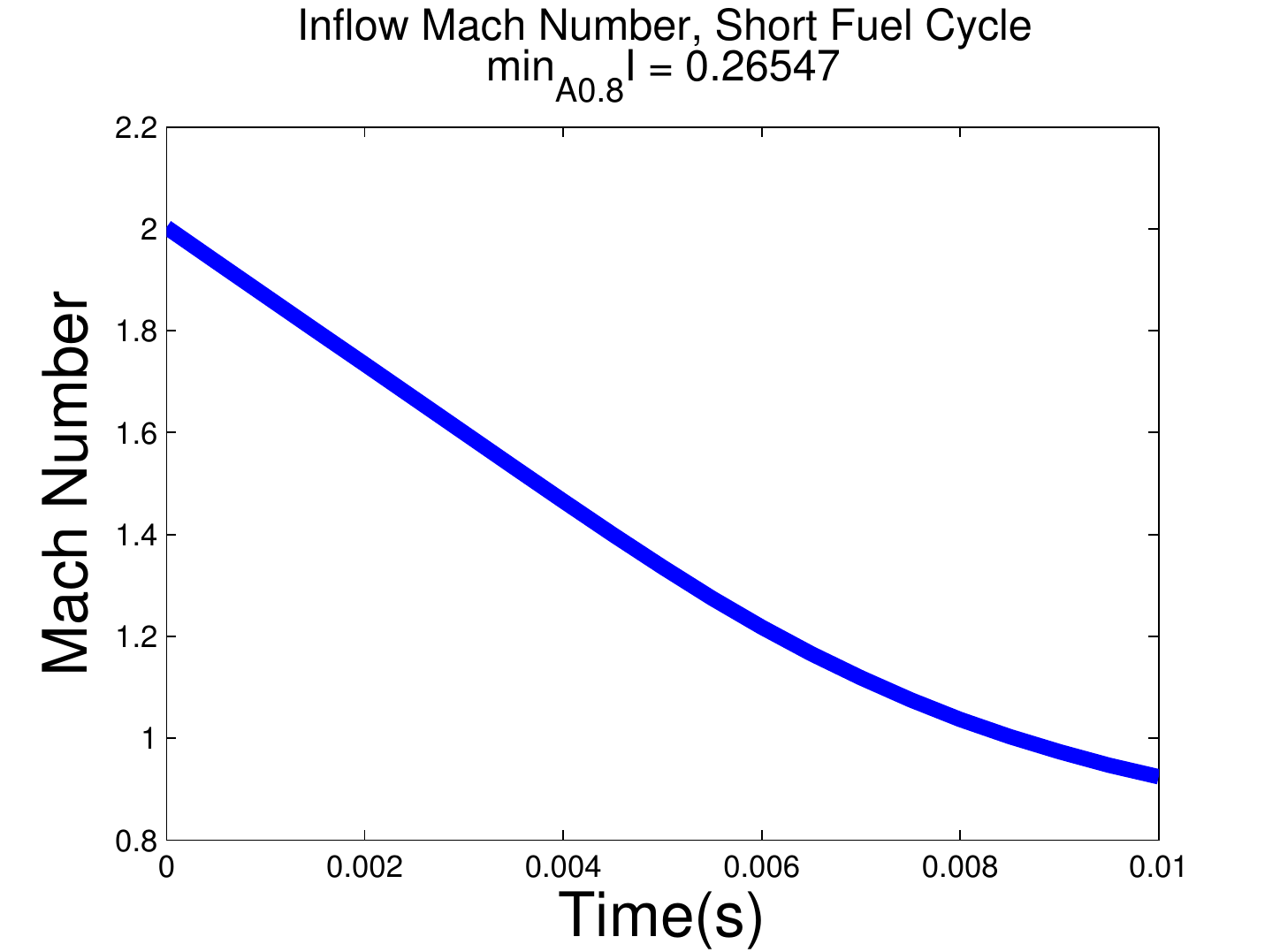}
	
	\caption{
	\label{fig:change of the constraint with the short fuel cycle}
	The results of the large deviation problem by changing the constraint set $\bfA_{1.0}$ to $\bfA_{0.8}$ 
	and $\bfA_{1.2}$ with the short fuel cycle.}
\end{figure}

\begin{figure}
	\centering
	\includegraphics[width=0.49\textwidth]{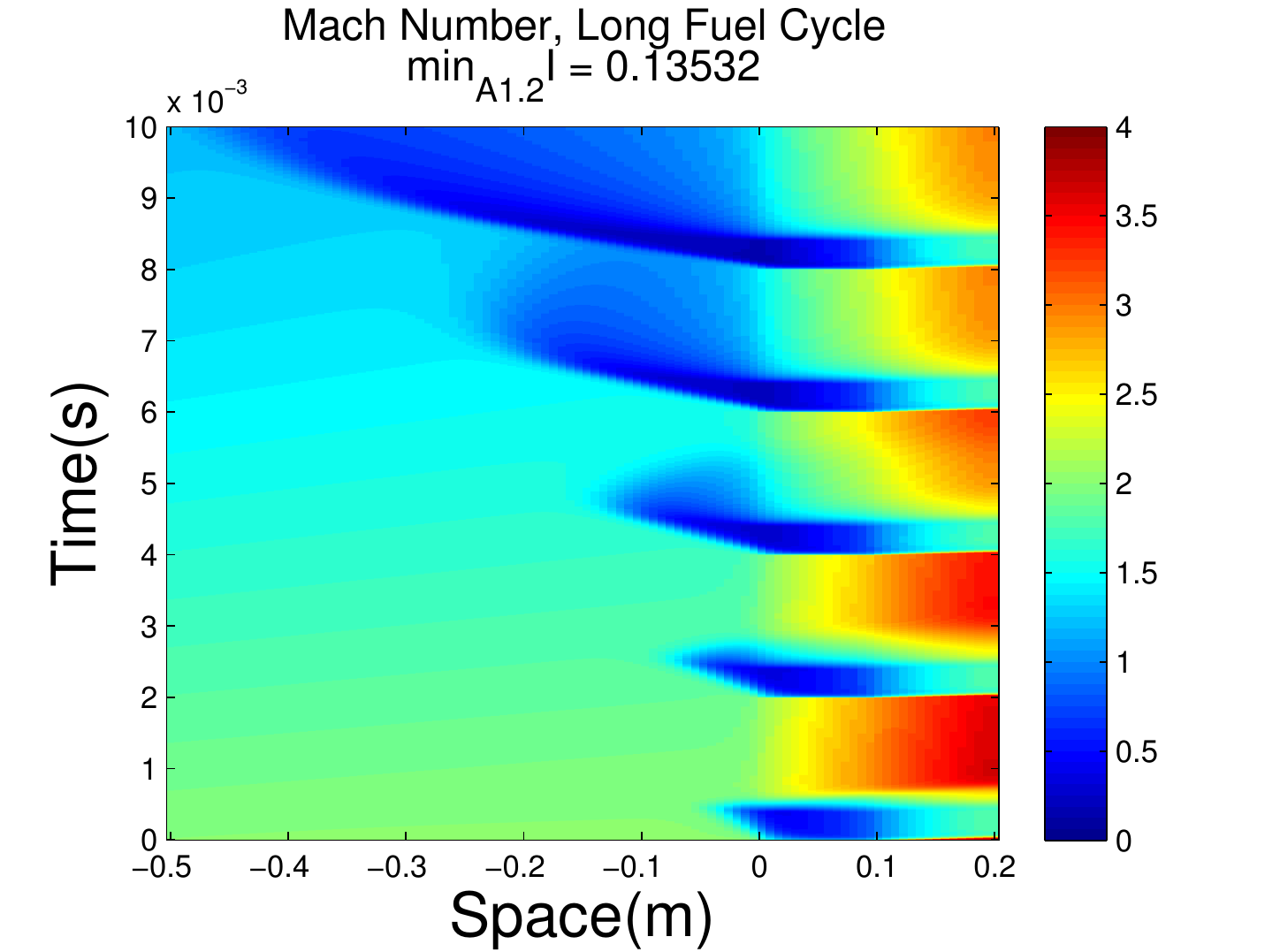}
	\includegraphics[width=0.49\textwidth]{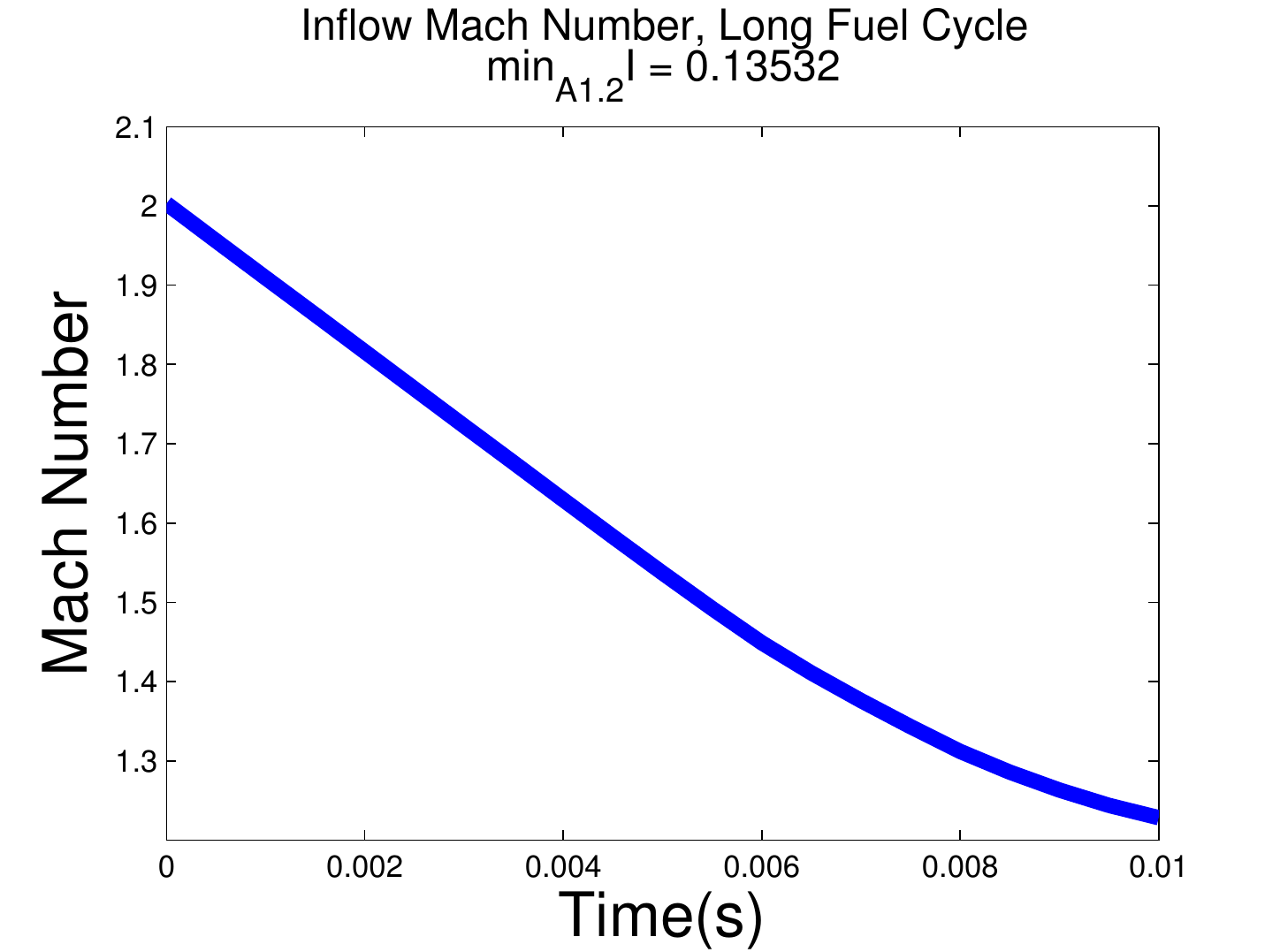}
	
	\includegraphics[width=0.49\textwidth]{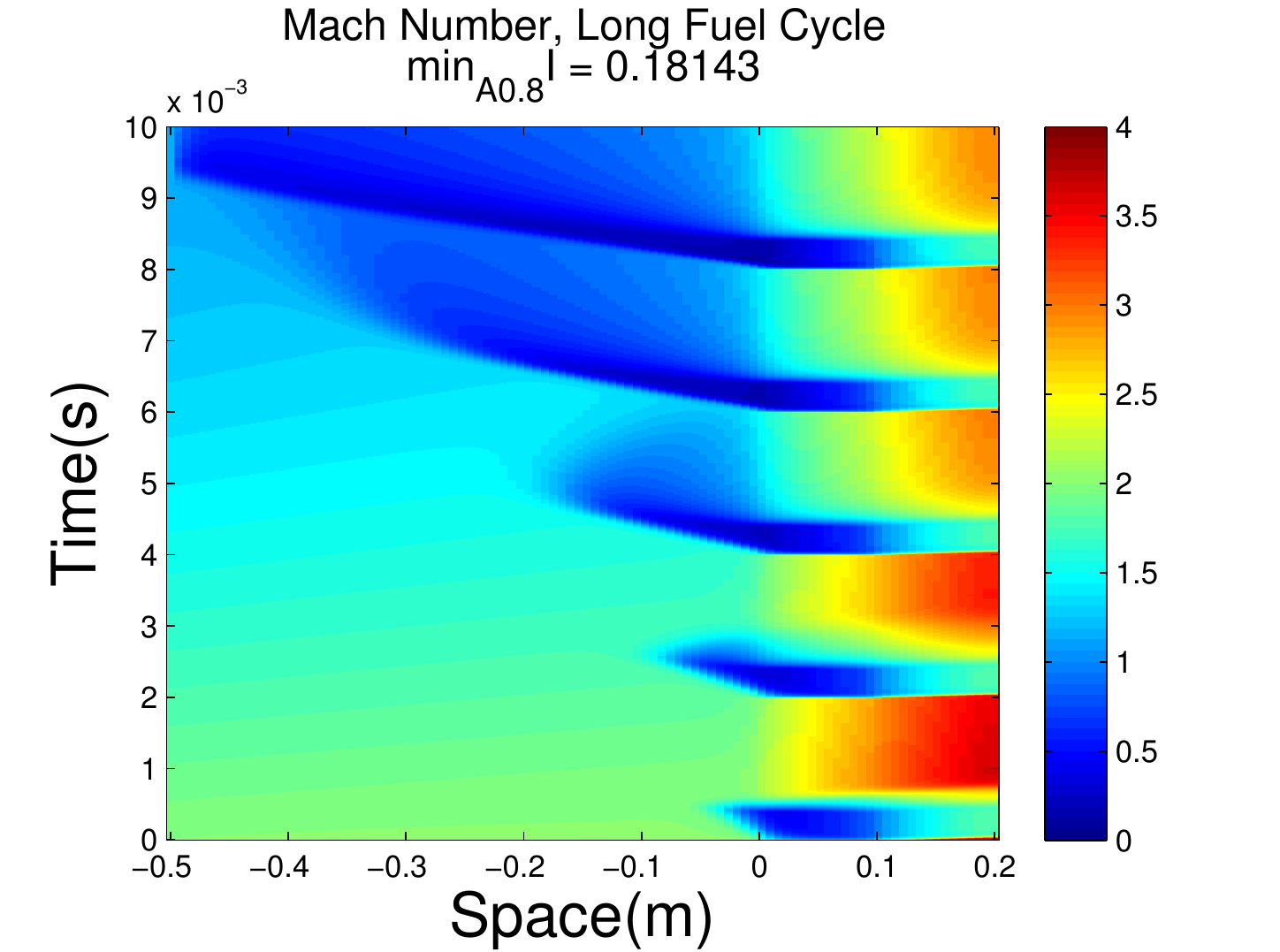}
	\includegraphics[width=0.49\textwidth]{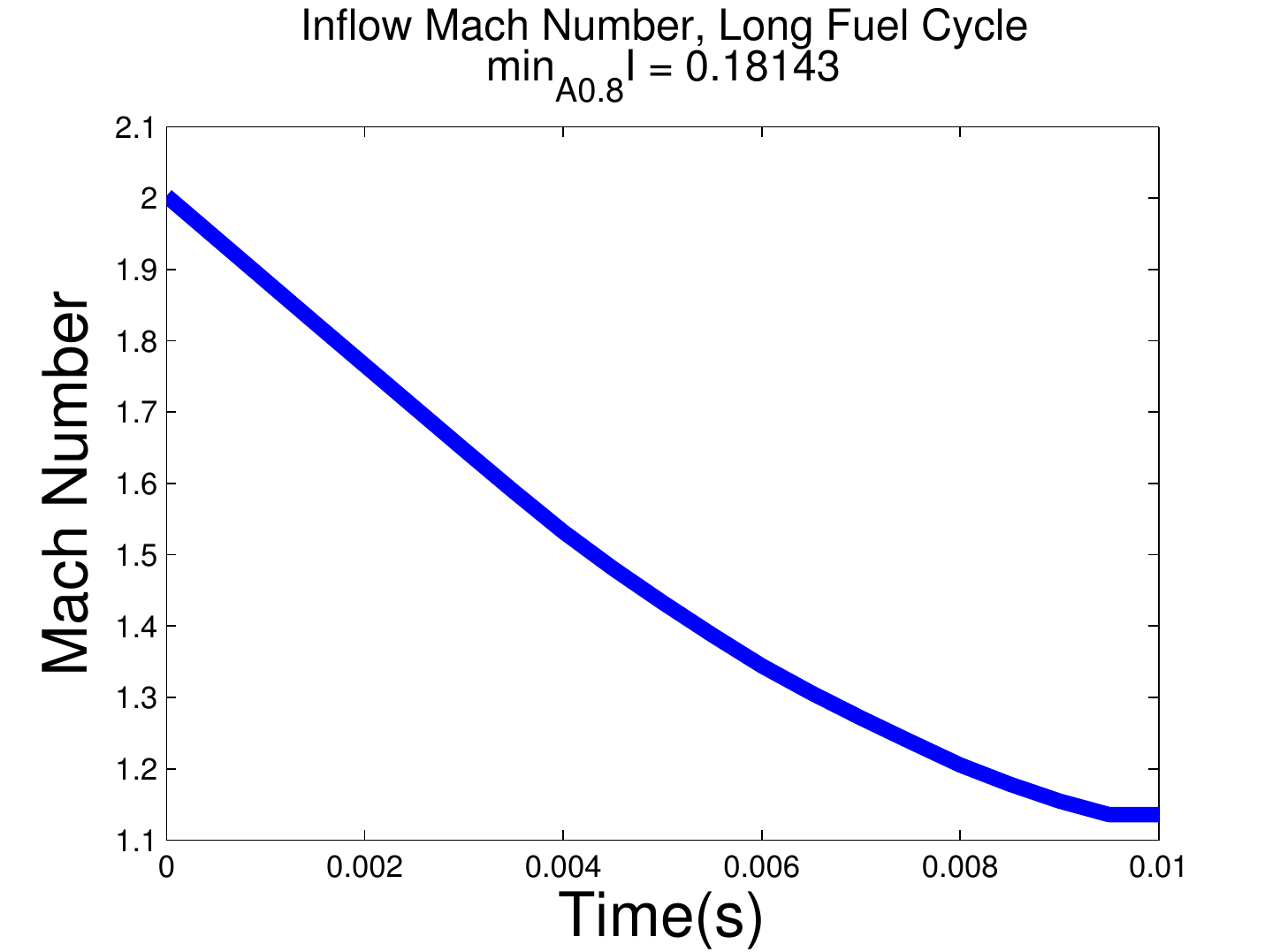}

	\caption{
	\label{fig:change of the constraint with the long fuel cycle}
	The results of the large deviation problem by changing the constraint set $\bfA_{1.0}$ to $\bfA_{0.8}$ 
	and $\bfA_{1.2}$ with the long fuel cycle.}
\end{figure}

\begin{table}
	\centering
	\begin{tabular}{|c|c|c|c|}
		\hline 
		& Short Fuel Cycle & Long Fuel Cycle & $\calI(u_{in}^*)$ \tabularnewline
		\hline
		$\inf_{\tilde{u}_{in}\in\bfA_{0.8}}I(\tilde{u}_{in})$ 
		& $0.26547$ & $0.18143$ & $0.3042$ \tabularnewline
		\hline
		$\inf_{\tilde{u}_{in}\in\bfA_{1.0}}I(\tilde{u}_{in})$ 
		& $0.21504$ & $0.15603$ & $0.21125$ \tabularnewline
		\hline
		$\inf_{\tilde{u}_{in}\in\bfA_{1.2}}I(\tilde{u}_{in})$ 
		& $0.13667$ & $0.13532$ & $0.1352$ \tabularnewline
		\hline 
	\end{tabular}
	\caption{
	\label{tab:the optimal rate functions over different A}
	The optimal rate functions over $\bfA_{0.8}$, $\bfA_{1.0}$ and $\bfA_{1.2}$. The values decrease with 
	the looser constraints.}
\end{table}

From Figure \ref{fig:change of the constraint with the short fuel cycle}, with the short fuel cycle, 
when we consider the constraint set $\bfA_{1.2}$, the result is not too surprising: the most probable inflow 
Mach number is close to a straight line starting from Mach $2$ to Mach $1.2$, and the shock generated by the 
engine has no contribution. Indeed, the set $\bfA_{1.2}$ can be viewed as the case that the engine operates 
in a more normal condition, and in this case the effect of the engine is less than the effect in the case of 
$\bfA_{1.0}$ so the inflow Mach number is still the dominating effect. Nevertheless, when $\bfA_{0.8}$ is 
considered, the qualitative behavior changes. The most probable situation to have the Mach $0.8$ subsonic 
flow is because the shock generated by the engine reaches the entrance, and the most probable inflow Mach 
number is not a straight line. In fact, the end point of the most probable inflow Mach number is higher than 
$0.8$, which means that to in the sense of the large deviations, the scramjet does not need a inflow with 
Mach $0.8$ to trigger the event of $\bfA_{0.8}$.

Then we consider the case of the long fuel cycle. From Figure 
\ref{fig:change of the constraint with the short fuel cycle}, we find that for the case of $\bfA_{1.2}$, 
the most probable inflow Mach number (upper right) is close to that of the short fuel cycle case (upper 
right in Figure \ref{fig:change of the constraint with the short fuel cycle}), and their optimal values 
of the rate function are also close to the continuum upper bound (see Table 
\ref{tab:the optimal rate functions over different A}). That means the shock generated by the engine is not 
the dominating effect in this case. For the case of $\bfA_{0.8}$, we can see that to have a subsonic flow of 
Mach $0.8$ at the entrance one just needs a supersonic inflow.

From Table \ref{tab:the optimal rate functions over different A}, we can find that the optimal values of the 
rate function with the short fuel cycle are consistently higher than those with the long fuel cycle.
This tells us that the strategy of the short fuel cycle is a more robust profile for the safety of the
operation of the scramjet.

\subsection{Impact of Engine Geometry}
\label{sec:Engine Geometry}

As Iaccarino et al. indicates in \cite{Iaccarino2011}, the engine geometry greatly influences the safe 
operating region against the unstart. In theory, for a larger angle of combustor $\theta_C$, the engine can 
accommodate more heat and therefore the scramjet has a larger safe operation region. However, 
\cite{Iaccarino2011} also mentions that a excessively large $\theta_C$ will leads to the flow separation and 
cause the loss of thrust. The quasi-1D model cannot capture this phenomenon and therefore it can not be seen 
as well in our results. Readers may refer to \cite{Iaccarino2011} to see more details of the impact of 
$\theta_C$.

We also find the similar qualitative results in the large deviation sense. We change $\theta_C$ from 
$7.5^\circ$ (the typical setting in \cite{Iaccarino2011}) to two extreme values in \cite{Iaccarino2011}, 
$\theta_C=2.5^\circ$ and $\theta_C=12^\circ$. In \cite{Iaccarino2011}, the safe operation region of the 
scramjet is very small when $\theta_C=2.5^\circ$, and $\theta_C=12^\circ$ is considered as the largest angle 
of combustor without causing flow separation.

From Figure \ref{fig:the LD solution of low theta_C}, we see that when $\theta_C=2.5^\circ$, the scramjet is 
not able to contain as much heat as before. Then the shock is build up and reaches the entrance. The optimal 
values of the rate function are very low so the unstart can happen very easily when $\theta_C=2.5^\circ$. On 
the other hand from Figure \ref{fig:the LD solution of high theta_C}, when $\theta_C=12^\circ$, the optimal 
values of the rate function of both fueling profiles are close to the continuum upper bound and the 
minimizer are similar to $u_{in}^*$. This is because for a larger $\theta_C$, the engine has more space for 
the generated heat and the shock is not easy to move upstream.

\begin{figure}
	\centering
	\includegraphics[width=0.49\textwidth]{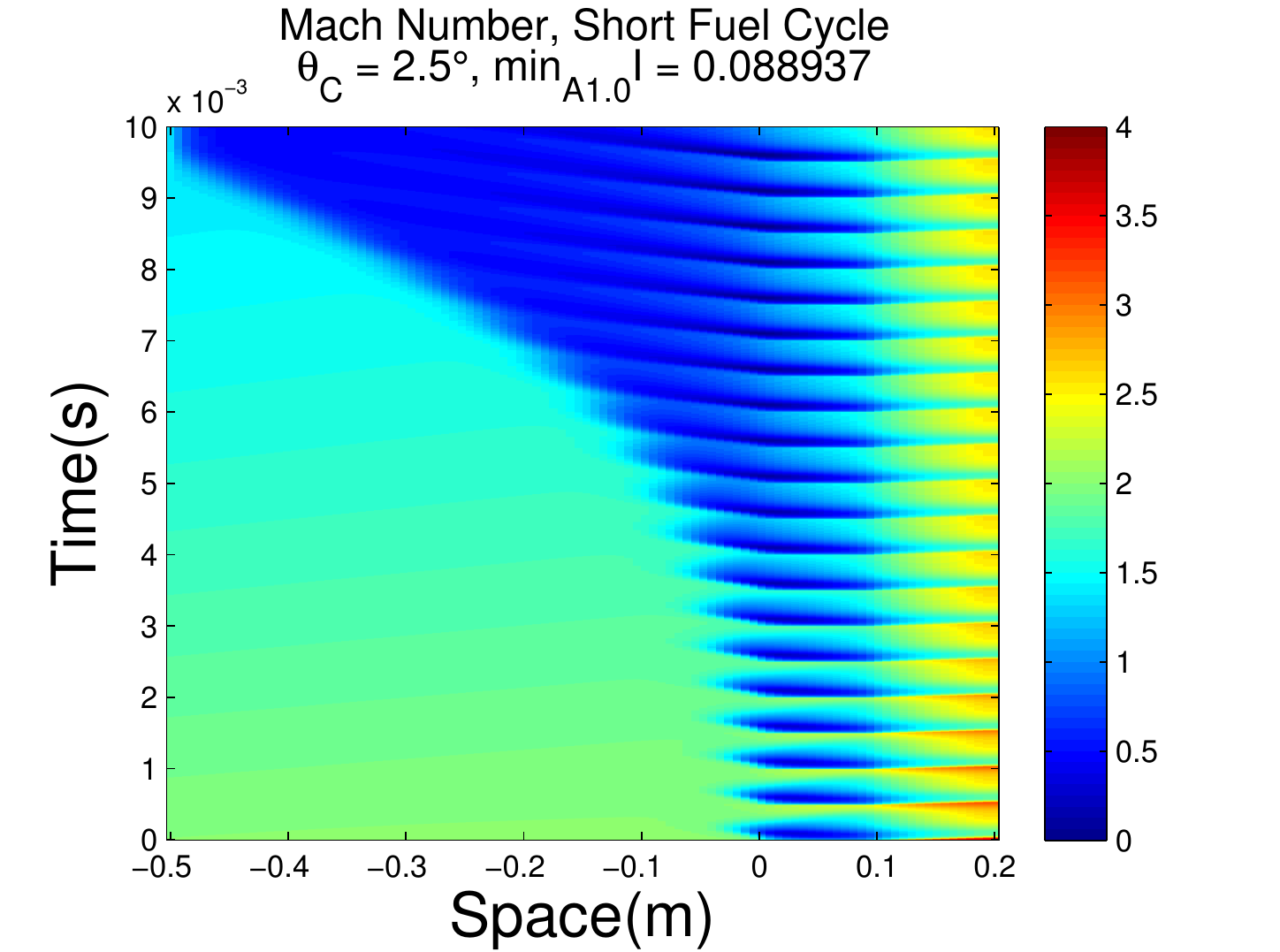}
	\includegraphics[width=0.49\textwidth]{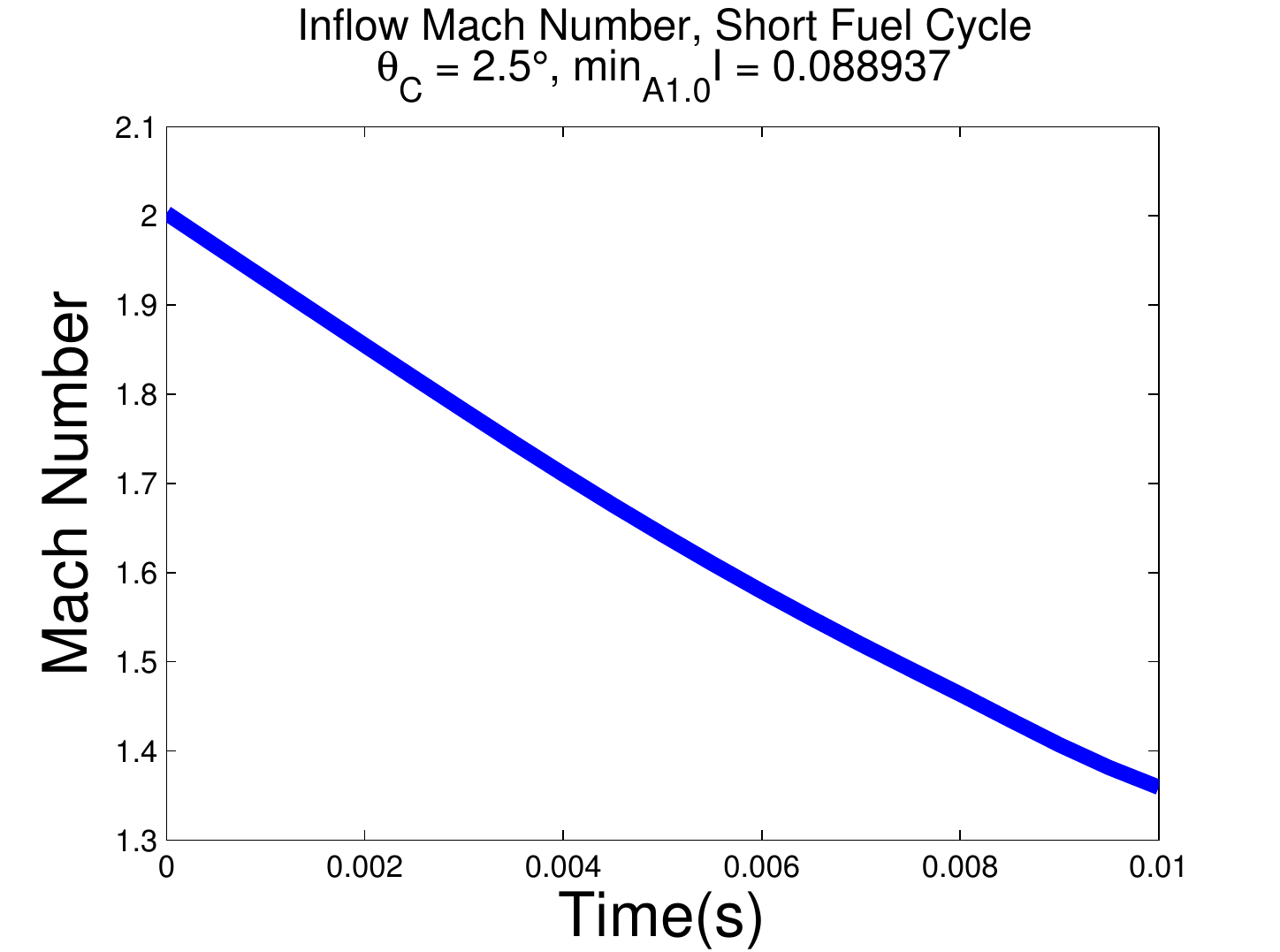}

	\includegraphics[width=0.49\textwidth]{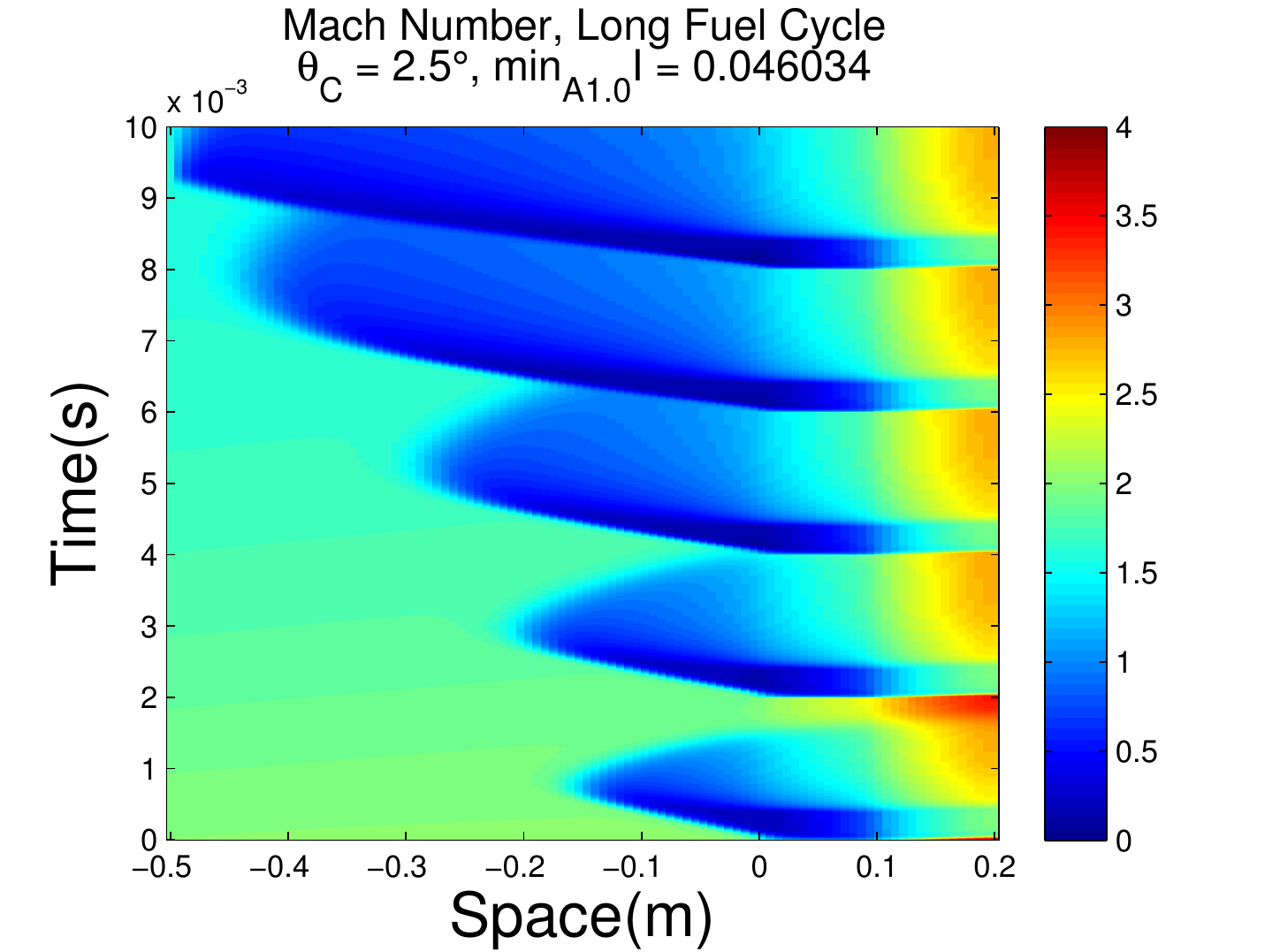}
	\includegraphics[width=0.49\textwidth]{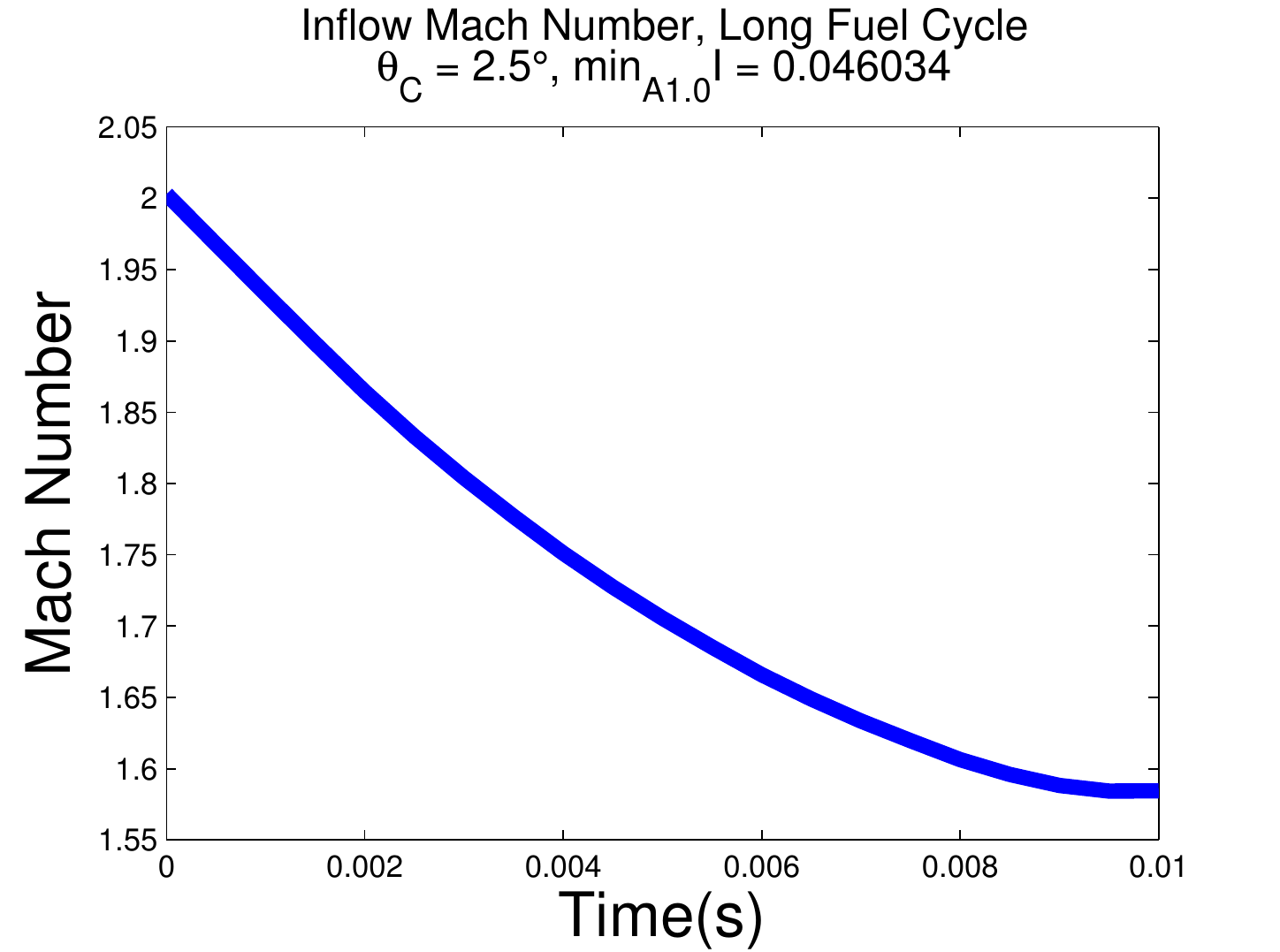}
	\caption{
	\label{fig:the LD solution of low theta_C}
	The large deviation solutions by changing $\theta_E$ from $7.5^\circ$ to $2.5^\circ$. The optimal values 
	of the rate function are significantly decreased.}
\end{figure}

\begin{figure}
	\centering
	\includegraphics[width=0.49\textwidth]{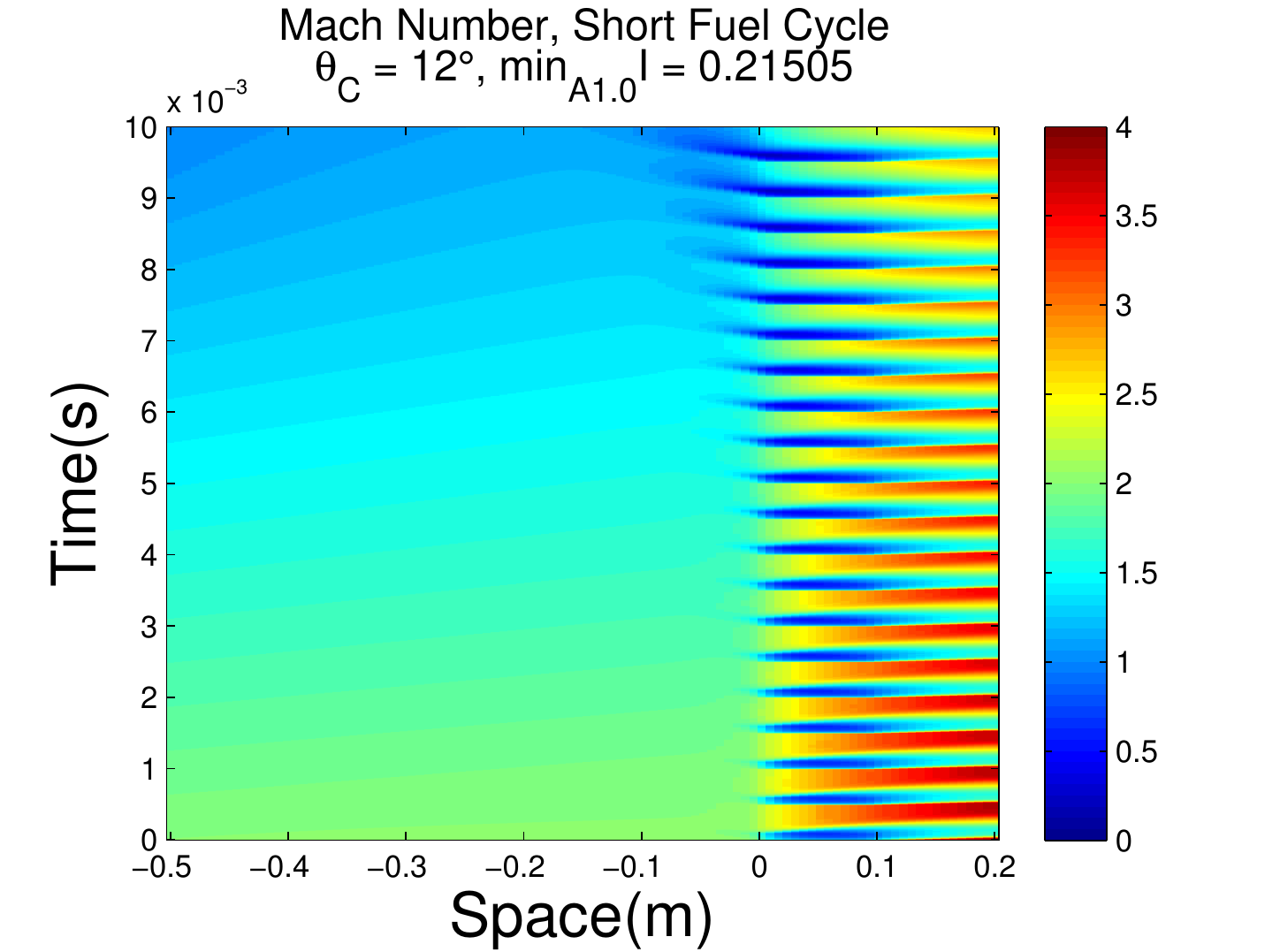}
	\includegraphics[width=0.49\textwidth]{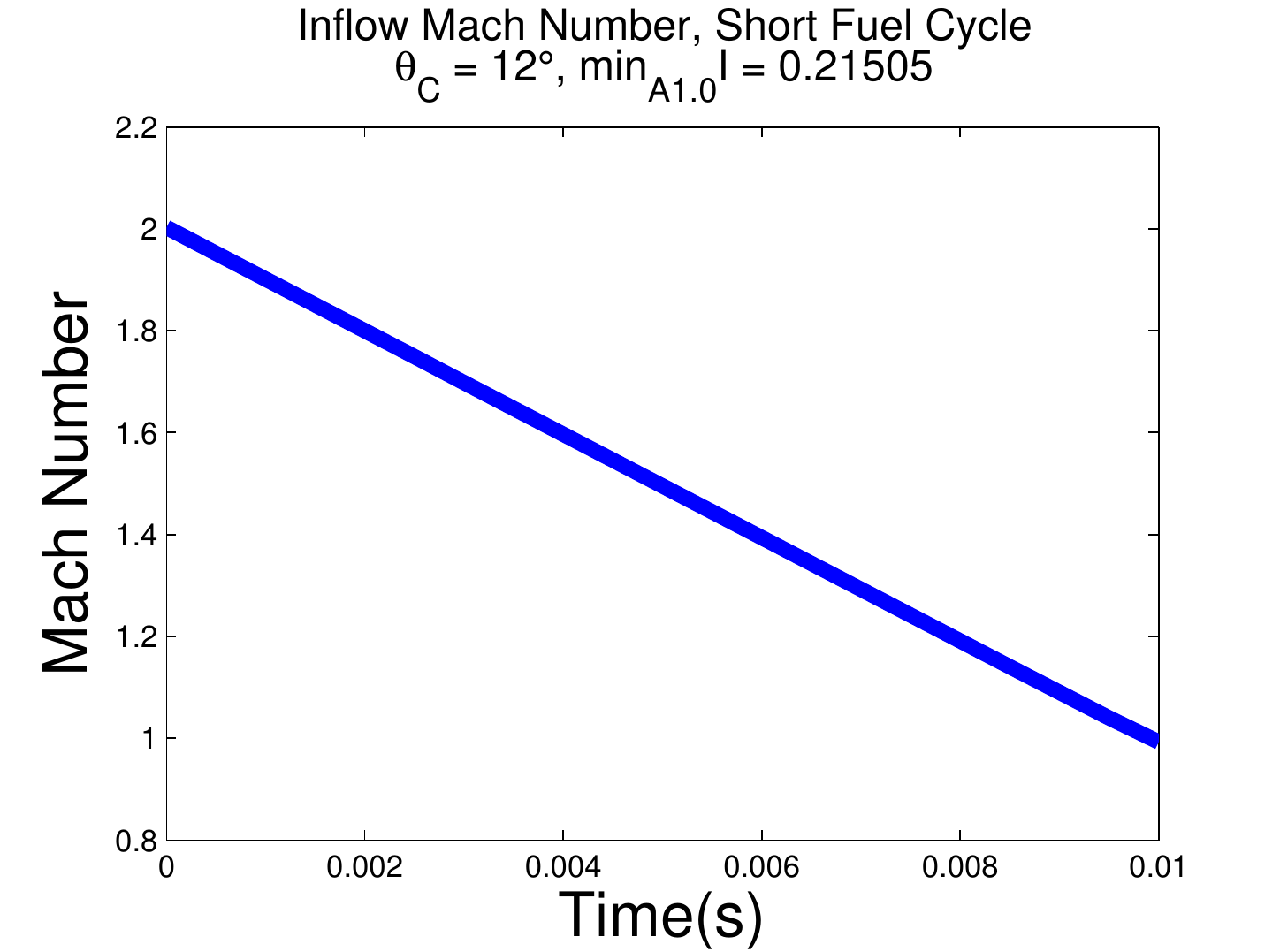}

	\includegraphics[width=0.49\textwidth]{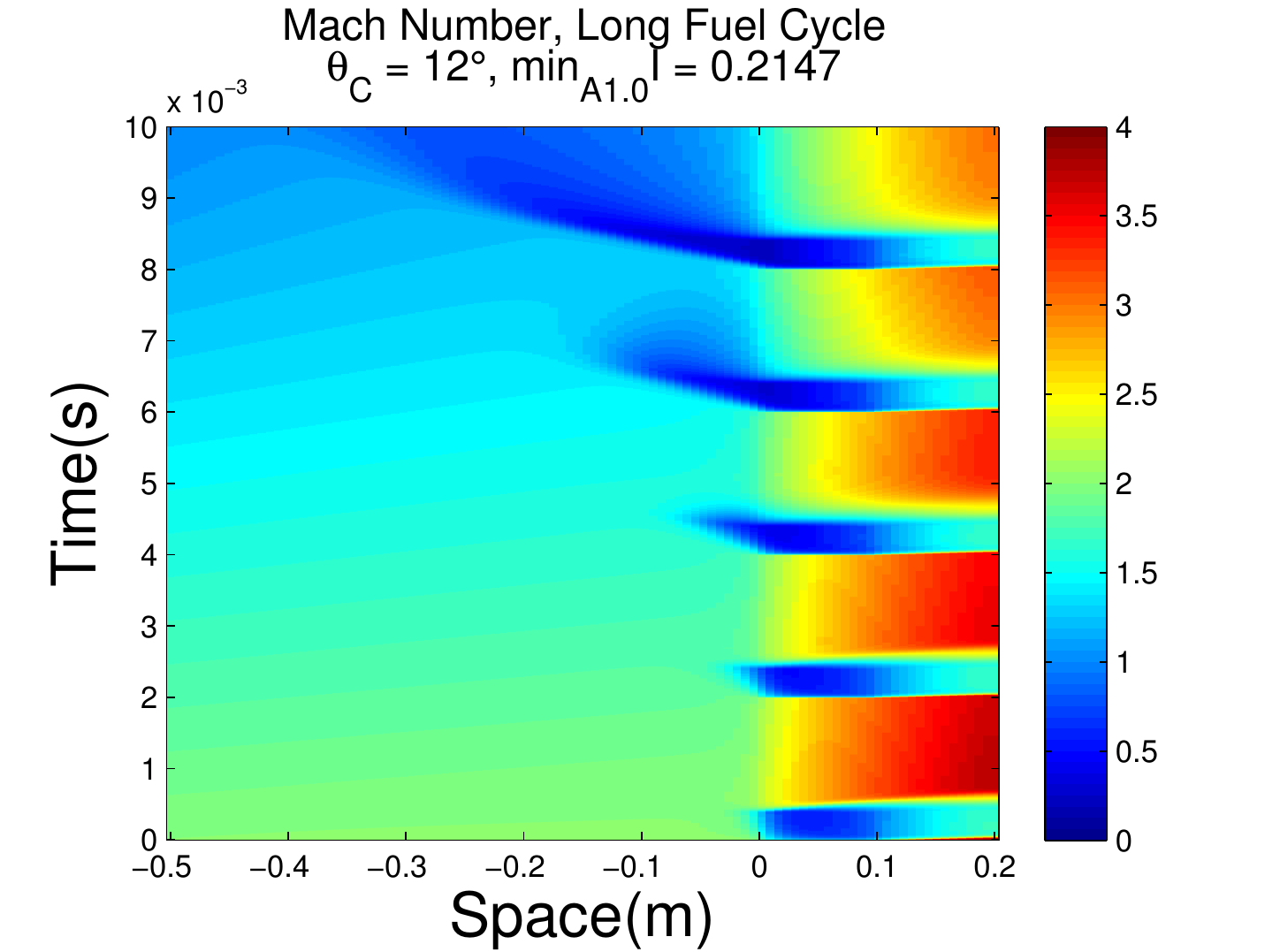}
	\includegraphics[width=0.49\textwidth]{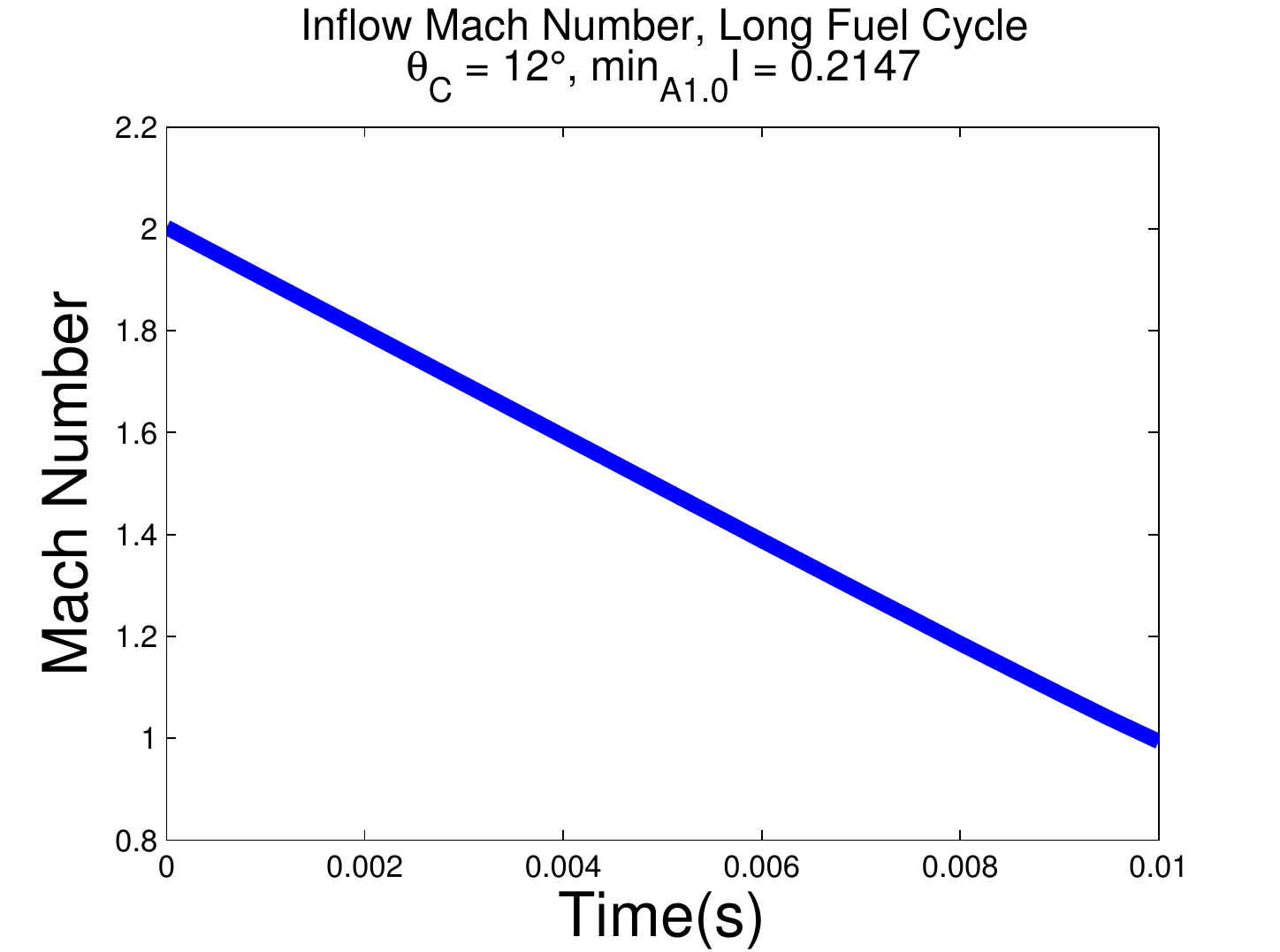}
	\caption{
	\label{fig:the LD solution of high theta_C}
	The large deviation solutions by changing $\theta_E$ from $7.5^\circ$ to $12^\circ$. The optimal value 
	of the high fuel cycle case is significantly increased. The optimal value of the short fuel cycle 
	case is almost the same because the value already reaches the upper limit.}
\end{figure}

\begin{table}
	\centering
	\begin{tabular}{|c|c|c|c|}
		\hline 
		$\inf_{\tilde{u}_{in}\in\bfA_{1.0}}I(\tilde{u}_{in})$ 
		& Short Fuel Cycle & Long Fuel Cycle & $\calI(u_{in}^*)$ \tabularnewline
		\hline 
		$\theta_C = 2.5^\circ$ & $0.088937$ & $0.046034$ & $0.21125$ \tabularnewline
		\hline 
		$\theta_C = 7.5^\circ$ & $0.21504$  & $0.15603$ & $0.21125$ \tabularnewline
		\hline 
		$\theta_C = 12^\circ$  & $0.21505$  & $0.2147$ & $0.21125$ \tabularnewline
		\hline 
	\end{tabular}
	\caption{
	\label{tab:the optimal rate function of diffetent theta_C}
	The optimal values of the rate function of different $\theta_C$. The higher $\theta_C$ have the higher 
	optimal values}
\end{table}

\subsection{Resolutions of Large Deviation Solutions}
\label{sec:resolution}

From the previous simulations, we are actually convinced that $\tilde{N}=20$ is the sufficient resolution of 
the large deviation solution due to the smoothness and the strong linearity of the solutions. Here we 
support our claim that $\tilde{N}=20$ is enough by doubling $\tilde{N}$ to $40$. If the results are very 
close, we can strongly believe our argument.

\begin{figure}
	\centering
	\includegraphics[width=0.49\textwidth]{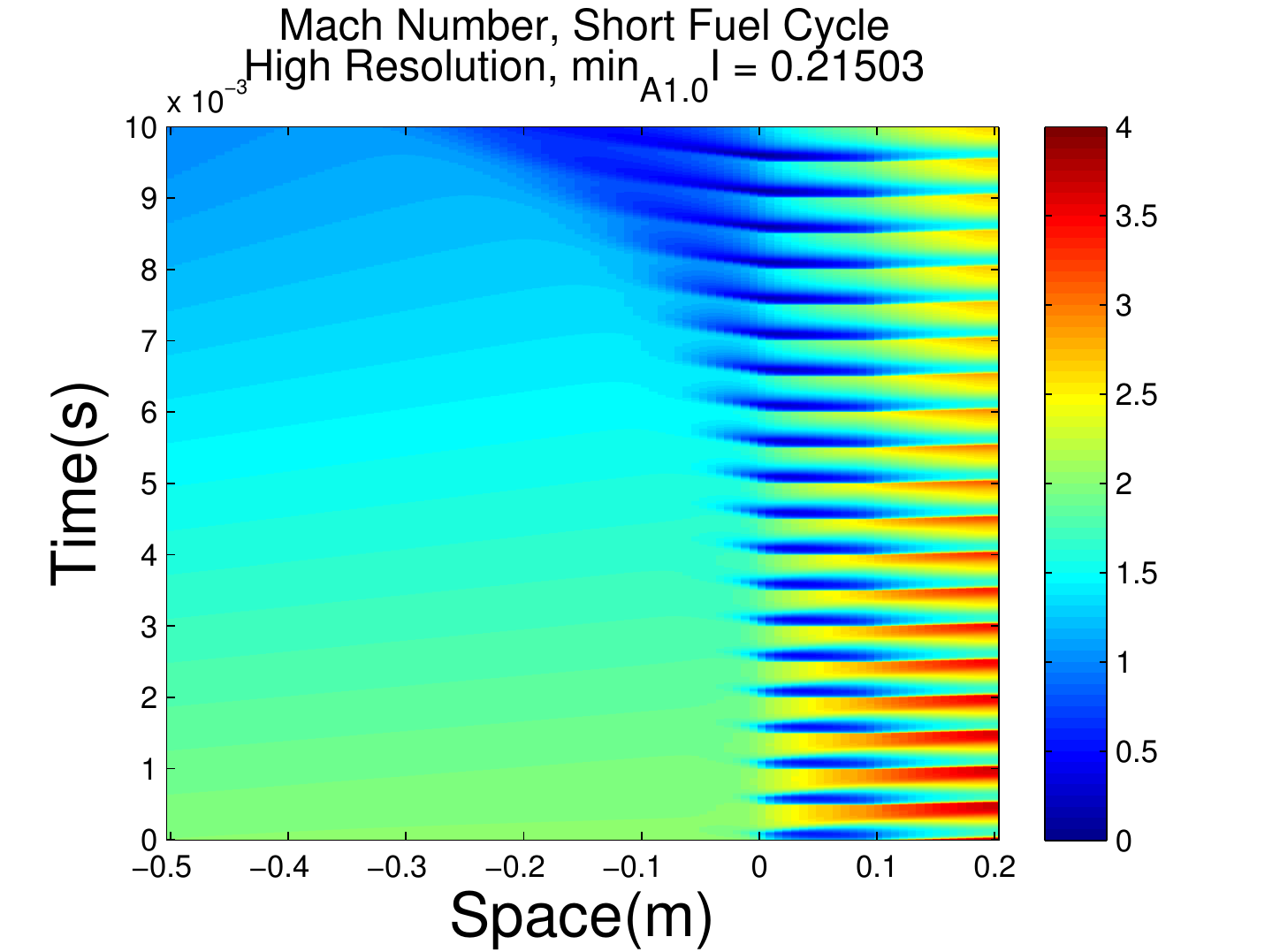}
	\includegraphics[width=0.49\textwidth]{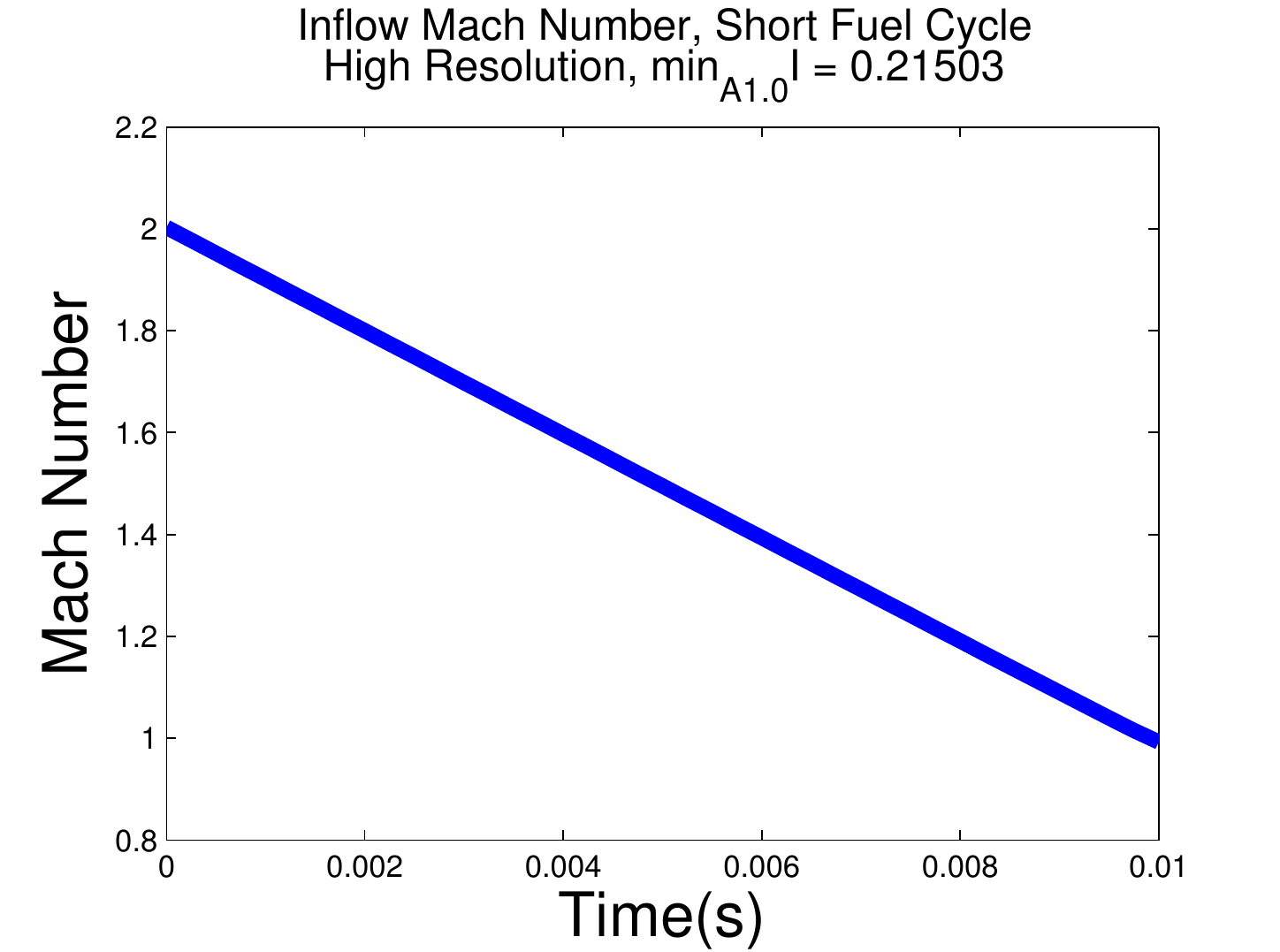}
	
	\includegraphics[width=0.49\textwidth]{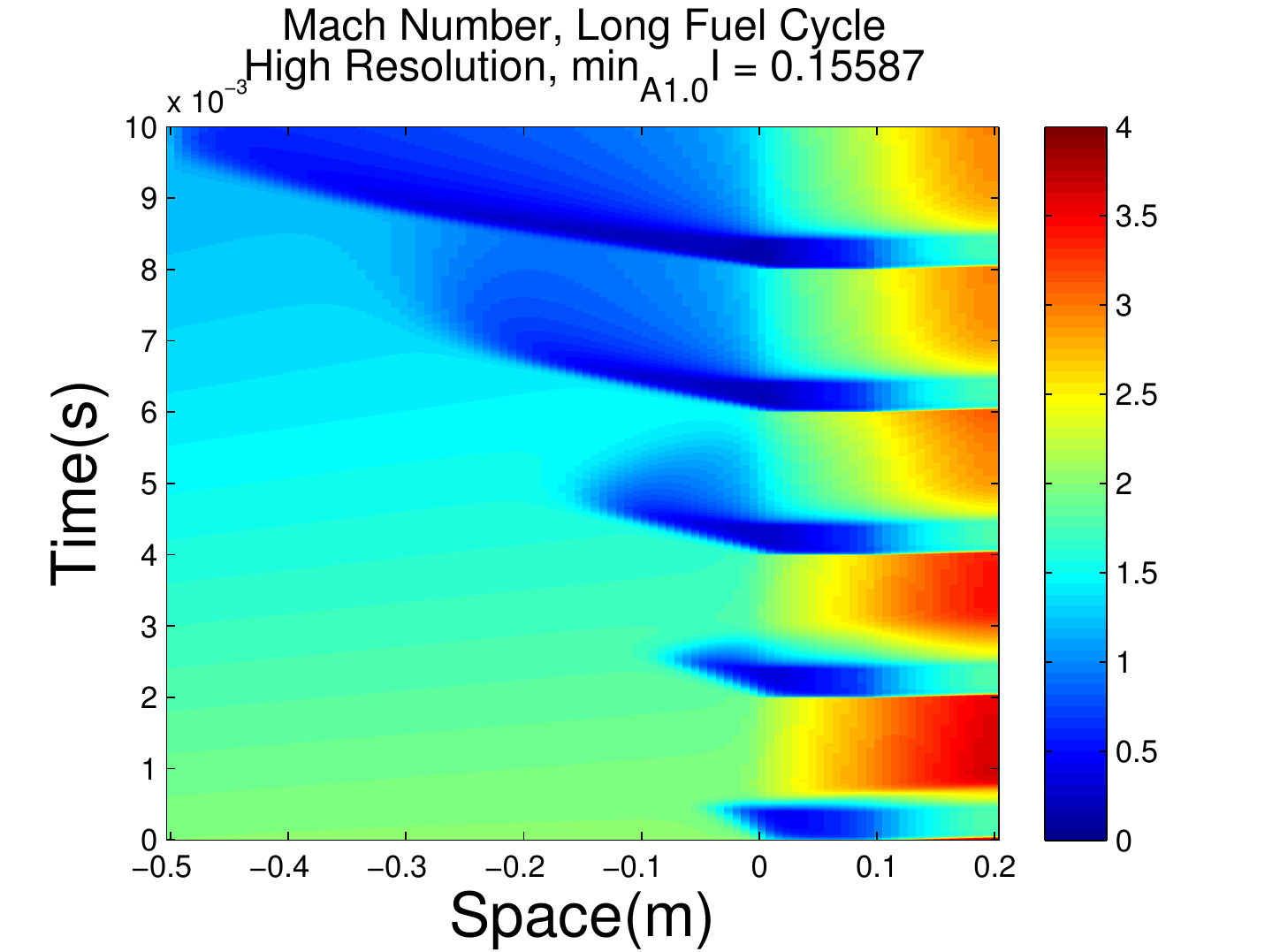}
	\includegraphics[width=0.49\textwidth]{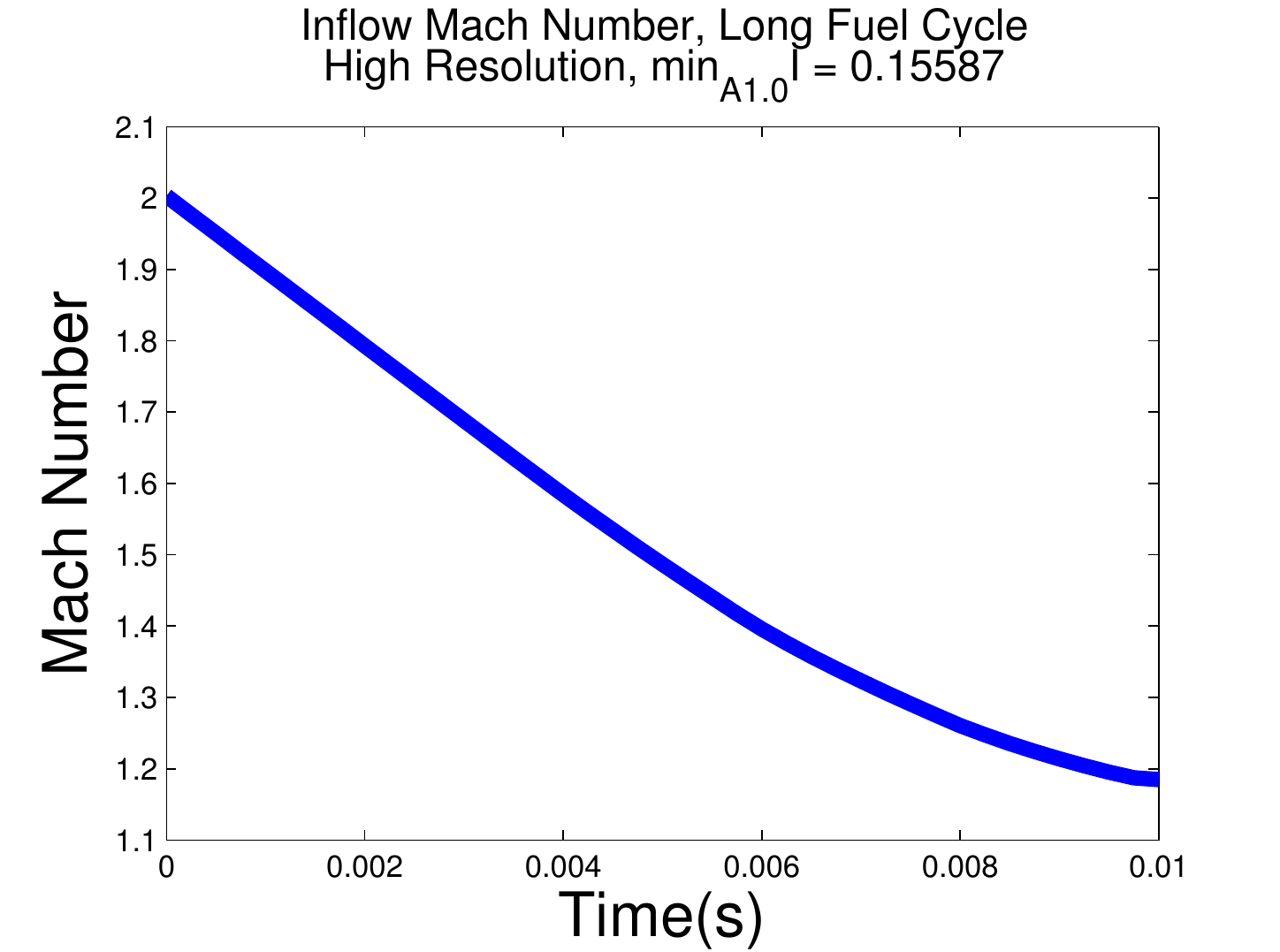}
	
	\caption{
	\label{fig:the LD solutions with the high resolution}
	The large deviation solutions with the high resolution. The settings are the same as those in Figure 
	\ref{fig:impact of fueling} but the solution resolution $\tilde{N}$ is doubled to $40$. The results are 
	nearly the same and the relative differences between the optimal values of the rate function are less 
	than $1\%$.}
\end{figure}

We use the same settings in Section \ref{sec:impact of fueling}, but let $\tilde{N}=40$ instead. From Figure 
\ref{fig:the LD solutions with the high resolution}, the optimal solutions are essentially the same, and the 
relative differences between the optimal values of the rate function are less than $1\%$.

\begin{table}
	\centering
	\begin{tabular}{|c|c|c|c|}
		\hline 
		$\inf_{\tilde{u}_{in}\in\bfA_{1.0}}I(\tilde{u}_{in})$ 
		& Short Fuel Cycle & Long Fuel Cycle & $\calI(u_{in}^*)$ \tabularnewline
		\hline 
		$\tilde{N}=20$ & $0.21504$ & $0.15603$ & $0.21125$ \tabularnewline
		\hline 
		$\tilde{N}=40$ & $0.21503$ & $0.15587$ & $0.21125$ \tabularnewline
		\hline 
	\end{tabular}
	\caption{
	\label{tab:the rate functions of different resolutions}
	The optimal values of the rate function for different resolutions. The relative differences are less 	
	than $1\%$.}
\end{table}
\section{Monte Carlo Simulation with Importance Sampling}
\label{sec:importance sampling}

The large deviation results obtained in Section \ref{sec:numerical result of LD} is the exponential rate of 
decay of the probability, but are not the actual one. In this section we compute the probability of the 
unstart by using the Monte Carlo method.

\subsection{Introduction to Importance Sampling}

The most basic way to compute the probability of the unstart $\PP(\bfA)$ is to generate many independent 
sample paths $\{\tilde{u}^j_{in}\}_{j=1}^J$, where for each $j$, $\{\tilde{u}^j_{in}(n)\}_{n=0}^N$ satisfy 
(\ref{eq:reduced order u_in}) and (\ref{eq:linear interpolation of inflow speed}). Then use the numerical 
PDE in Section \ref{sec:numerical PDE} to determine if $\tilde{u}_{in}^j\in\bfA$. The basic Monte Carlo 
estimator is
\begin{equation}
	\label{eq:basic Monte Carlo estimator}
	\hat{P}^{MC} = \frac{1}{J}\sum_{j=1}^J 1_\bfA(\tilde{u}_{in}^j).
\end{equation}
Because $\EE[\hat{P}^{MC}]=\PP(\bfA)$, $\hat{P}^{MC}$ is an unbiased estimator, which means that 
$\hat{P}^{MC}\to\PP(\bfA)$ almost surely as $J\to\infty$ by the law of large numbers. In addition, by the 
central limit theorem, the error bar of $\hat{P}^{MC}$ is proportional to its standard deviation
\[
	\Std(\hat{P}^{MC}) = \frac{1}{\sqrt{J}}[\PP(\bfA)-\PP^2(\bfA)]^{1/2}.
\]
In order to have a meaningful estimate, the order of the error bar should not exceed the order of the 
estimated probability. Namely,
\[
	\frac{\Std(\hat{P}^{MC})}{\PP(\bfA)} 
	= \frac{1}{\sqrt{J}}\left(\frac{1}{\PP(\bfA)}-1\right)^{1/2}
	= \mathcal{O}(1).
\]
The large deviation analysis tells us that as $\eps\ll 1$, $\PP(\bfA)$ decreases exponentially in $\eps$ so 
at the same time $J$ has to increase exponentially. The exponential growth of $J$ will eventually cause the 
basic Monte Carlo method computationally impossible.

To see what causes this computational difficulty, we note that the basic Monte Carlo estimator is simply the 
empirical frequency of $\tilde{u}_{in}\in\bfA$. When $\eps$ is small, only a very small fraction of the 
samples is meaningful (in $\bfA$), and most of the samples has no contribution to the estimation so 
resulting the inaccuracy of the estimator.

The well-established method to solve this issue is to use the importance sampling technique. The idea is 
that since most of the samples under the original measure $\PP$ have no contribution, we use a different 
measure $\QQ$ to sample $\tilde{u}_{in}^j$ so that there is a significant fraction of the samples 
contributing to the estimate. Since we bias the measure $\PP$, a correction is needed to obtain an unbiased 
estimate. More precisely, we have
\[
	\PP(\tilde{u}_{in}\in\bfA) 
	= \EE_\PP[1_\bfA(\tilde{u}_{in})] 
	= \EE_\QQ[1_\bfA(\tilde{u}_{in})\frac{d\PP}{d\QQ}(\tilde{u}_{in})],
\]
where $d\PP/d\QQ$ is the change of the measure. The importance sampling estimator is 
\[
	\hat{P}^{IS} = \frac{1}{J}\sum_{j=1}^J 1_\bfA(\tilde{u}_{in}^j)\frac{d\PP}{d\QQ}(\tilde{u}_{in}^j),
\]
where $\tilde{u}_{in}^j$ is sampled under $\QQ$. Note that $\hat{P}^{IS}$ is unbiased as 
$\EE[\hat{P}^{IS}]=\PP(\tilde{u}_{in}\in\bfA)$ and its standard deviation is 
\[
	\Std(\hat{P}^{MC}) = \frac{1}{\sqrt{J}}\Std(1_\bfA(\tilde{u}_{in})\frac{d\PP}{d\QQ}(\tilde{u}_{in})).
\]

\subsection{Large-Deviation-Based Importance Sampling}

The main challenge of the importance sampling is how to choose $\QQ$ (and therefore $d\PP/d\QQ$) to lower 
the standard deviation. A good choice may come from the solution of the large deviation problem 
$\inf_{\tilde{u}_{in}\in\bfA}I(\tilde{u}_{in})$. Assuming that the minimizer
\[
	\tilde{u}_{in}^{**}=\arg\inf_{\tilde{u}_{in}\in\bfA}I(\tilde{u}_{in}),
\]
is unique, then for any open neighborhood $\bfN(\tilde{u}_{in}^{**})$ of $\tilde{u}_{in}^{**}$, we have the 
following asymptotic conditional probability:
\begin{multline*}
	\PP(\tilde{u}_{in}\in\bfN(\tilde{u}_{in}^{**})|\tilde{u}_{in}\in\bfA)
	= 1- \PP(\tilde{u}_{in}\in\bfN^C(\tilde{u}_{in}^{**})|\tilde{u}_{in}\in\bfA)\\
	= 1 - {\PP(\bfN^C(\tilde{u}_{in}^{**})\cap\bfA)}/{\PP(\bfA)}
	\overset{\eps\ll 1}{\approx} 1 - 
	\frac{\exp(-\frac{1}{\eps^2}\inf_{\tilde{u}_{in}\in\bfN^C(\tilde{u}_{in}^{**})\cap\bfA}
	I(\tilde{u}_{in}))}
	{\exp(-\frac{1}{\eps^2}\inf_{\tilde{u}_{in}\in\bfA}I(\tilde{u}_{in}))} \overset{\eps\to 0}{\to} 1.
\end{multline*}
In other words, the mass of the conditional probability is concentrated exponentially fast around the most 
probable path $\tilde{u}_{in}^{**}$. This observation motivates us to choose the measure $\QQ$ so that the 
sampled $\tilde{u}_{in}^j$'s are centered around $\tilde{u}_{in}^{**}$.

Now we construct $\QQ$ and $d\PP/d\QQ$ based on $\tilde{u}_{in}^{**}$. Recall that under $\PP$, from 
(\ref{eq:reduced order u_in}) we have
\[
	\tilde{u}_{in}((n+1)m) = \tilde{u}_{in}(nm) + \eps \sigma_u \Delta \tilde{W}_{n+1}
	= u_0 + \eps \sigma_u \sum_{l=0}^n \Delta \tilde{W}_{l+1},\quad 
	n=0,\ldots,\tilde{N}-1,
\]
We let $\QQ$ such that under $\QQ$, $\tilde{u}_{in}$ is a Gaussian random walk centered at $ 
\tilde{u}_{in}^{**}$:
\begin{equation}
	\label{eq:reduced order u_in under Q}
	\tilde{u}_{in}((n+1)m) 
	= \tilde{u}_{in}^{**}((n+1)m) 
	+ \eps \sigma_u \sum_{l=0}^n \Delta \hat{W}_{l+1},\quad 
	n=0,\ldots,\tilde{N}-1,
\end{equation}
where $\{\Delta\hat{W}_{n+1}\}_{n=0}^{\tilde{N}-1}$ are independent Gaussian random variables with mean 
zero and variance $m\Delta t$ under $\QQ$. The intermediate variables are still determined by the linear 
interpolation (\ref{eq:linear interpolation of inflow speed}). Note that 
$\{\tilde{u}_{in}(nm)\}_{n=1}^{\tilde{N}}$ are jointly Gaussian under both $\PP$ and $\QQ$ so the change of 
measure can be obtained explicitly:
\begin{equation}
	\label{eq:change of measure}
	\frac{d\PP}{d\QQ}(\tilde{u}_{in};\tilde{u}_{in}^{**})
	= \frac
	{\exp\left(-\frac{1}{\eps^2\sigma_u^2}(\vecu-\mathbf{u}_0)^{\mathbf{T}}
	\mathbf{\Sigma}^{-1}(\vecu-\mathbf{u}_0)\right)}
	{\exp\left(-\frac{1}{\eps^2\sigma_u^2}(\vecu-\vecuSS)^{\mathbf{T}}
	\mathbf{\Sigma}^{-1}(\vecu-\vecuSS)\right)},
\end{equation}
where $\mathbf{u}_0$, $\vecuSS$ and $\vecu$ are $\tilde{N}$-dimensional column vectors:
\[
	\mathbf{u}_0 = (u_0,\ldots,u_0),\quad 
	\vecuSS = (\tilde{u}_{in}^{**}(m),\ldots,\tilde{u}_{in}^{**}(\tilde{N}m)),\quad
	\vecu = (\tilde{u}_{in}(m),\ldots,\tilde{u}_{in}(\tilde{N}m)),
\]
and $\mathbf{\Sigma}$ is the covariance matrix of $(\Delta\hat{W}_1,\ldots,\Delta\hat{W}_{\tilde{N}})$ under 
$\QQ$.

In summary, the large-deviation-based importance sampling is implemented as follows:
\begin{algorithm}
	\begin{enumerate}
		\item Compute $\tilde{u}_{in}^{**}=\arg\inf_{\tilde{u}_{in}\in\bfA}I(\tilde{u}_{in})$ numerically by 
		Algorithm \ref{alg:numerical LDP for unstart}.
		
		\item Sample $J$ independent $\tilde{u}_{in}^j$ under $\QQ$ by (\ref{eq:reduced order u_in under Q}).
		
		\item For each $j$, compute the change of measure 
		$\frac{d\PP}{d\QQ}(\tilde{u}_{in}^j;\tilde{u}_{in}^{**})$ in (\ref{eq:change of measure}).
		
		\item The large-deviation-based importance sampling estimator is 
		\[
			\hat{P}^{IS} = \frac{1}{J}\sum_{j=1}^J 
			1_\bfA(\tilde{u}_{in}^j)\frac{d\PP}{d\QQ}(\tilde{u}_{in}^j;\tilde{u}_{in}^{**}).
		\]
	\end{enumerate}
\end{algorithm}


\subsection{Simulation Results}

Here we test the two estimators of the probability of the unstart: the basic Monte Carlo estimator 
$\hat{P}^{MC}$ and the large-deviation-based importance sampling estimator $\hat{P}^{IS}$. The event we test 
is $\PP(\tilde{u}_{in}\in\bfA)=\PP(\tilde{u}_{in}\in\bfA_{1.0})$ (see (\ref{eq:the set of unstart}) for the 
definition of $\bfA$) with the short and long fuel cycles (see Section \ref{sec:impact of fueling}). As the 
inflow condition is random, instead of the uniform time increment $\Delta t$, we use the adaptive time 
increment to ensure that the CFL condition is satisfied (see Appendix \ref{sec:component-wise LLF} for more 
details.)

We let $J=10^4$ and test for $\eps=0.2,0.22,0.24,\ldots,0.4$, where $\PP(\bfA)=\mathcal{O}(10^{-1})$ for 
$\eps=0.4$ and $\PP(\bfA)=\mathcal{O}(10^{-3})$ for $\eps=0.2$. The (numerical) $99\%$ confidence interval 
and the (numerical) relative error are two quantities measuring the performance of the estimators. We define 
the numerical standard deviations as 
\begin{align*}
	&(\Std^{MC}_J)^2 = \frac{1}{J-1}\sum_{j=1}^J (1_\bfA(\tilde{u}_{in}^j)-\hat{P}^{MC})^2,\\
	&(\Std^{IS}_J)^2 = \frac{1}{J-1}\sum_{j=1}^J 	
	(1_\bfA(\tilde{u}_{in}^j)\frac{d\PP}{d\QQ}(\tilde{u}_{in}^j;\tilde{u}_{in}^{**}) -\hat{P}^{IS})^2.
\end{align*}
Then the numerical $99\%$ confidence intervals are 
\[
	[\hat{P}^{MC}-\frac{2.58}{\sqrt{J}}\Std^{MC}_J, \hat{P}^{MC}+\frac{2.58}{\sqrt{J}}\Std^{MC}_J],\quad 
	[\hat{P}^{IS}-\frac{2.58}{\sqrt{J}}\Std^{IS}_J, \hat{P}^{IS}+\frac{2.58}{\sqrt{J}}\Std^{IS}_J],
\]
and the numerical relative errors are $\Std^{MC}_J/\hat{P}^{MC}$ and $\Std^{IS}_J/\hat{P}^{IS}$.

\begin{figure}
	\centering
	\includegraphics[width=0.49\textwidth]{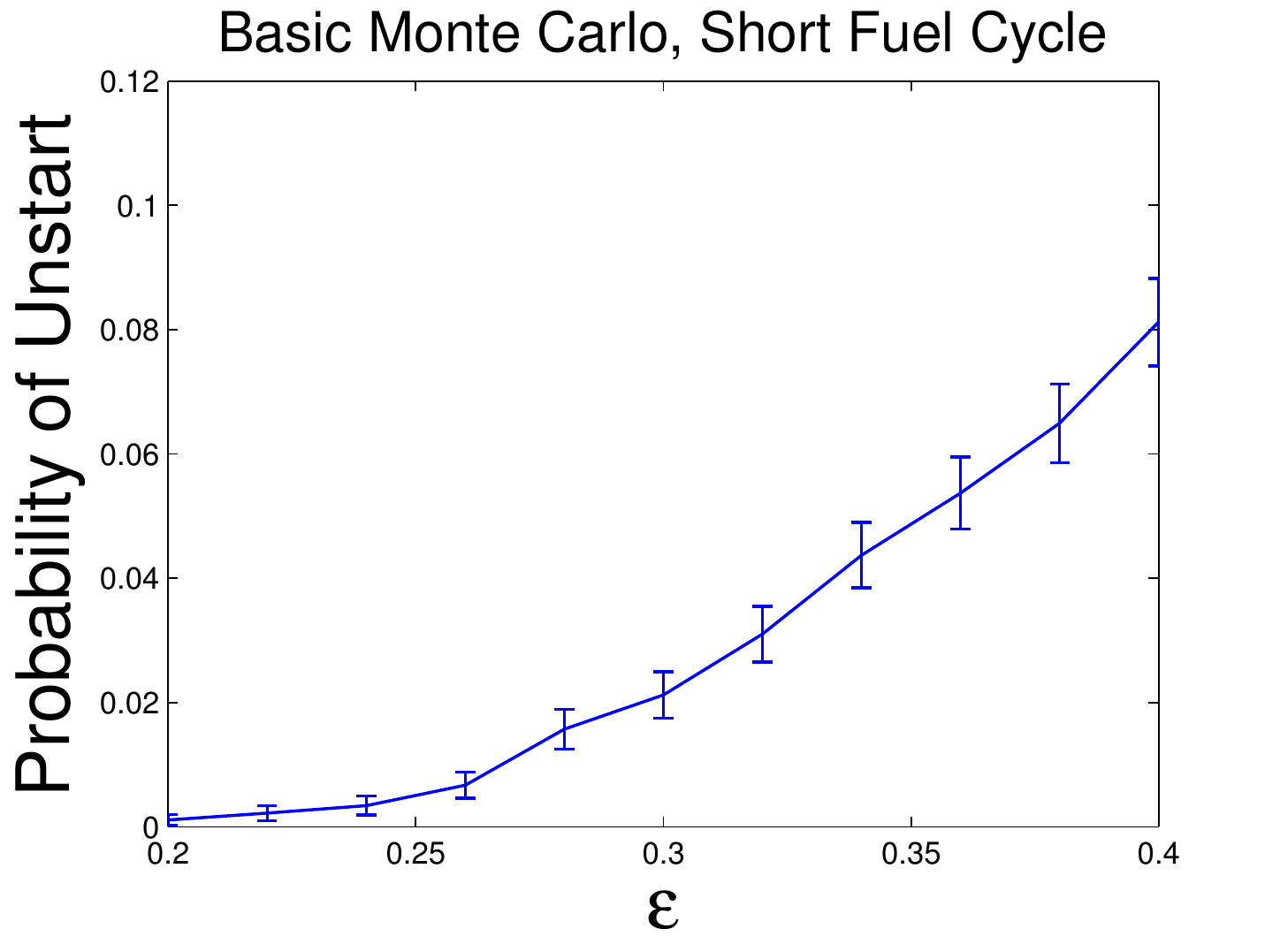}
	\includegraphics[width=0.49\textwidth]{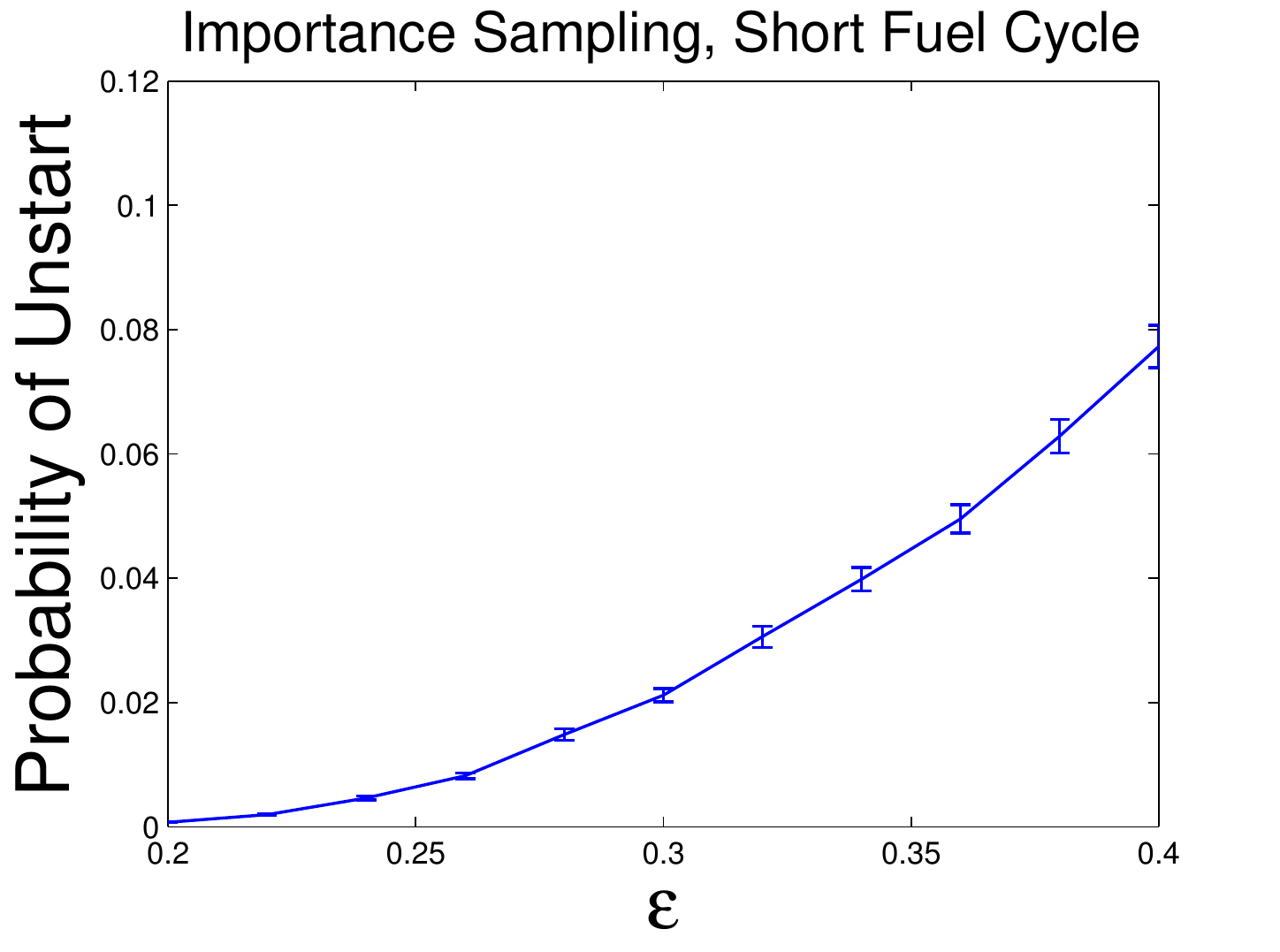}

	\includegraphics[width=0.49\textwidth]{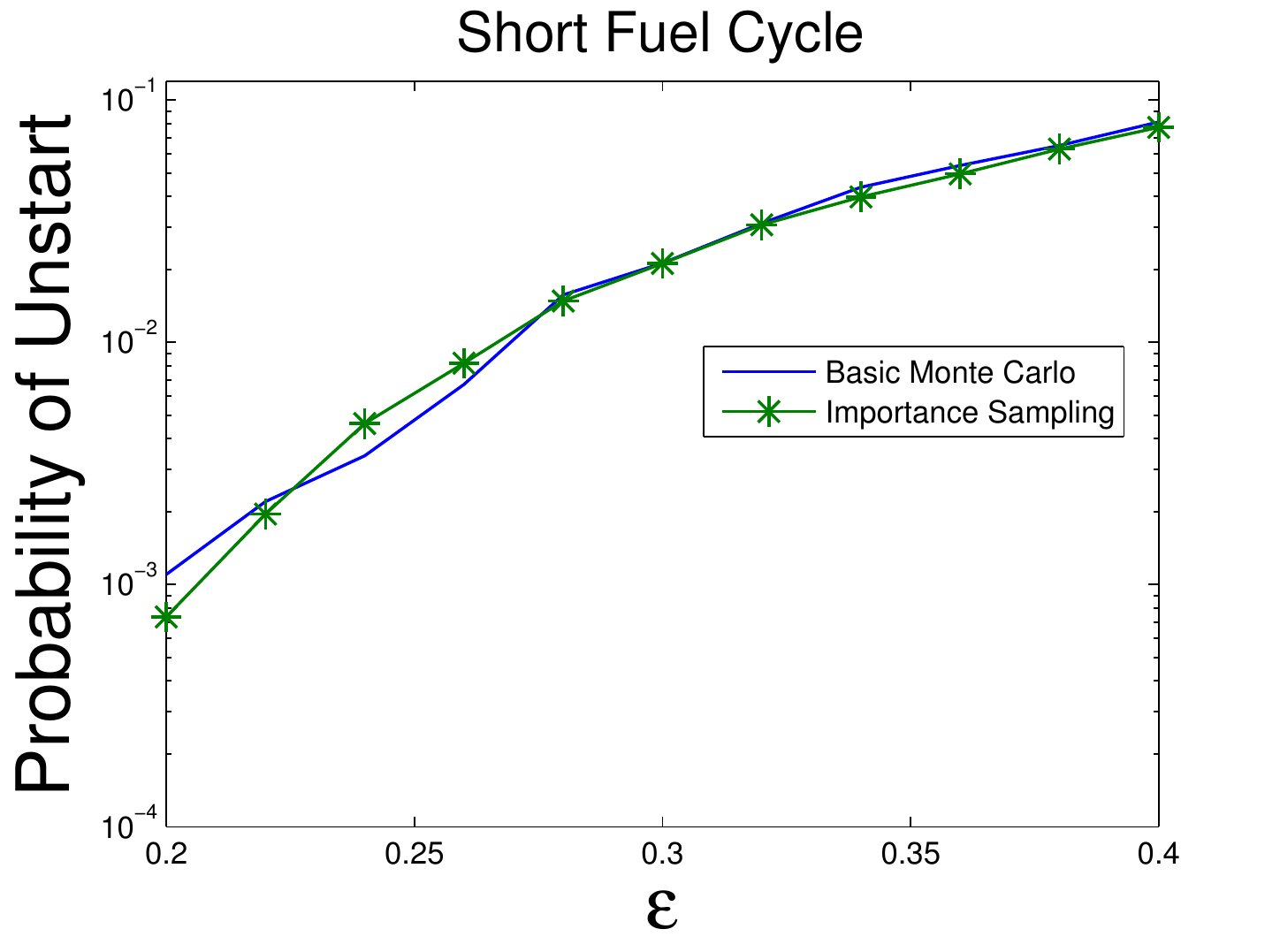}
	\includegraphics[width=0.49\textwidth]{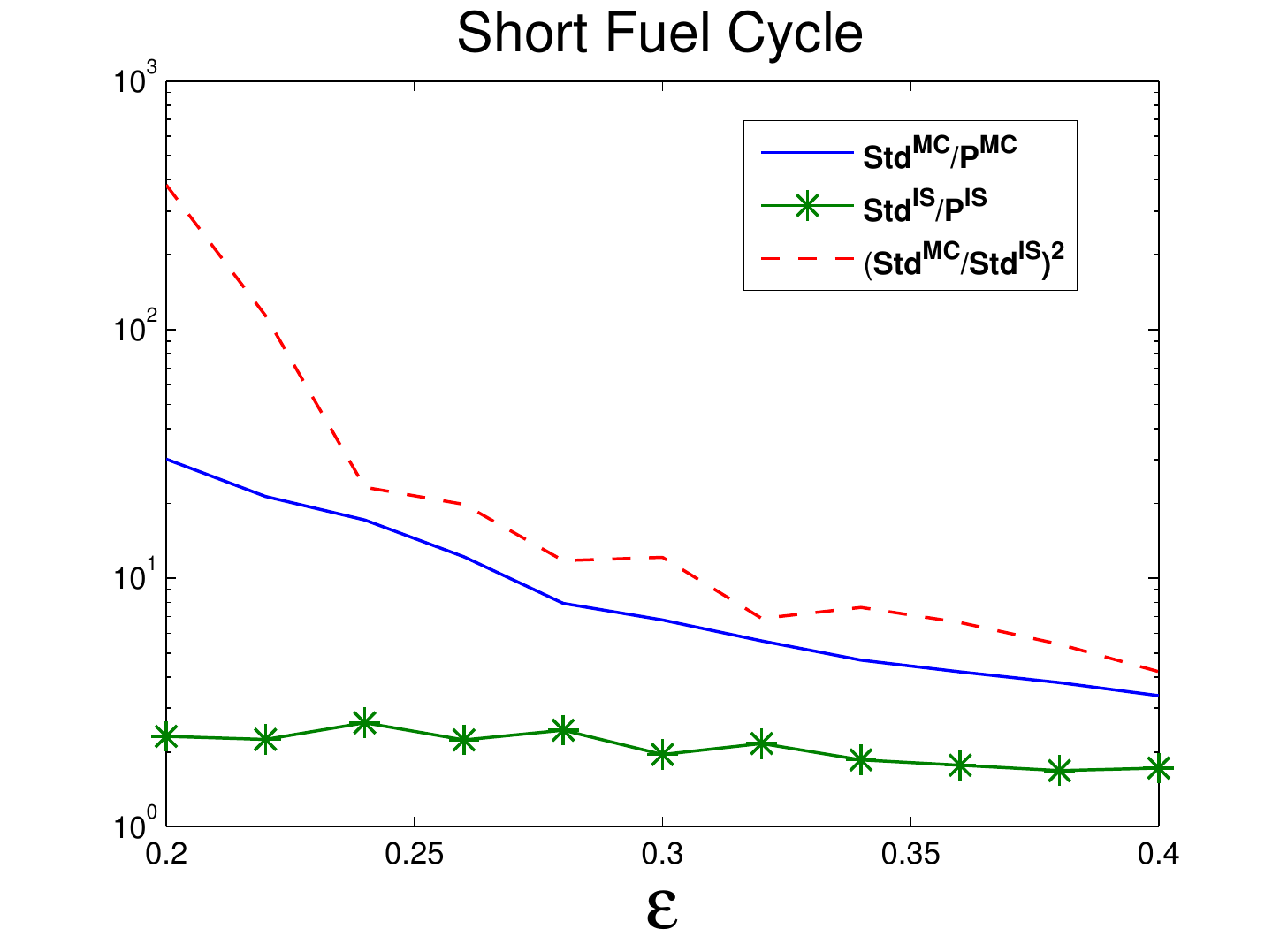}
	\caption{
	\label{fig:Monte Carlo for unstart, short fuel cycle}
	The basic Monte Carlo method and the importance sampling estimator for the probability of the unstart 
	when the short fuel cycle is used (the top figures). The error bar are the $99\%$ confidence intervals. 
	The bottom left figure is the plot of the estimated probabilities in the log scale. The bottom right  
	figure is the plot of the relative errors $\Std^{MC}_J/\hat{P}^{MC}$ and $\Std^{IS}_J/\hat{P}^{IS}$, and 
	the ratio of $\Std^{MC}_J$ to $\Std^{IS}_J$, which is the factor of the improvement of $\Std^{IS}_J$.}
\end{figure}

\begin{figure}
	\centering
	\includegraphics[width=0.49\textwidth]{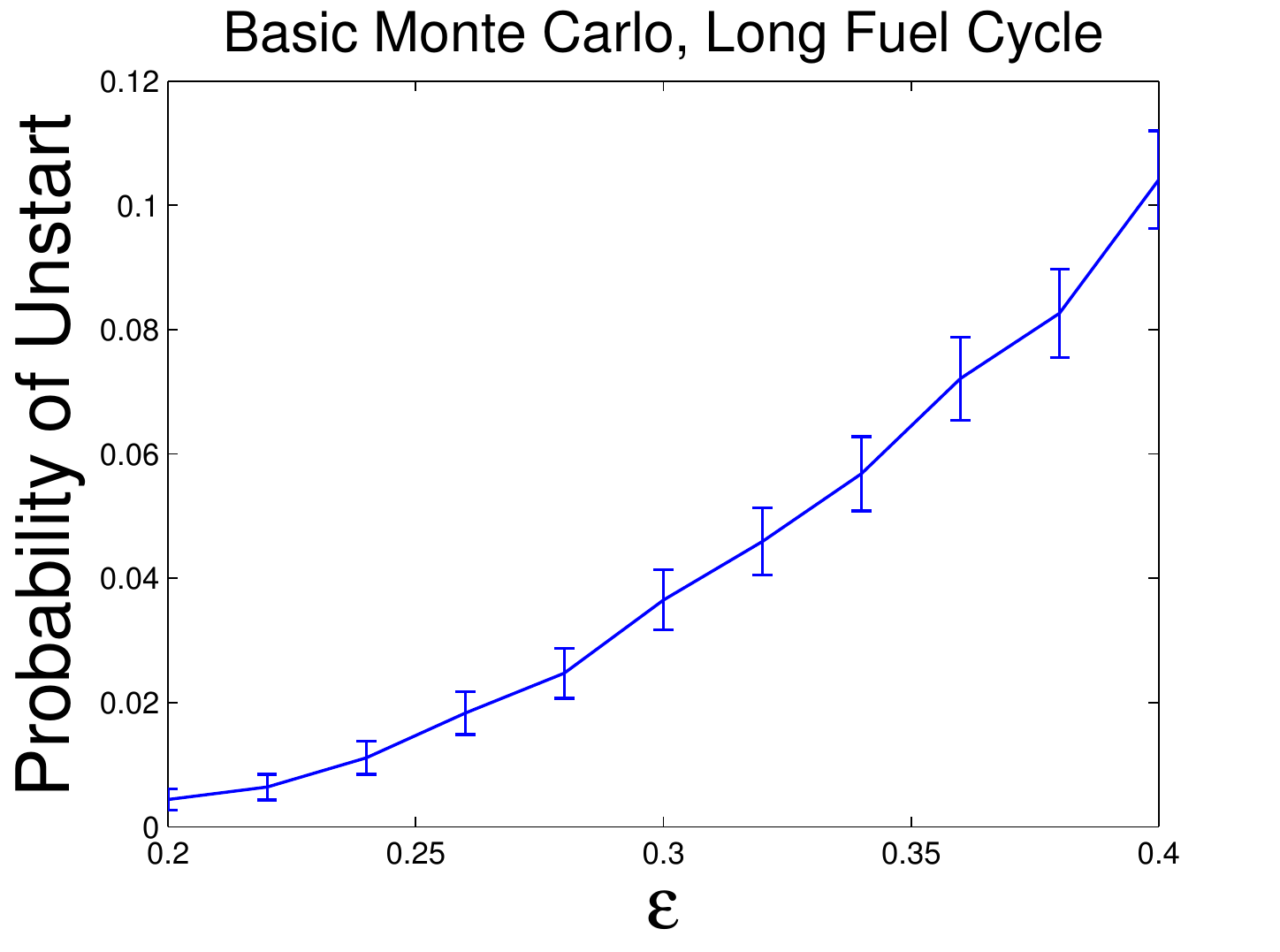}
	\includegraphics[width=0.49\textwidth]{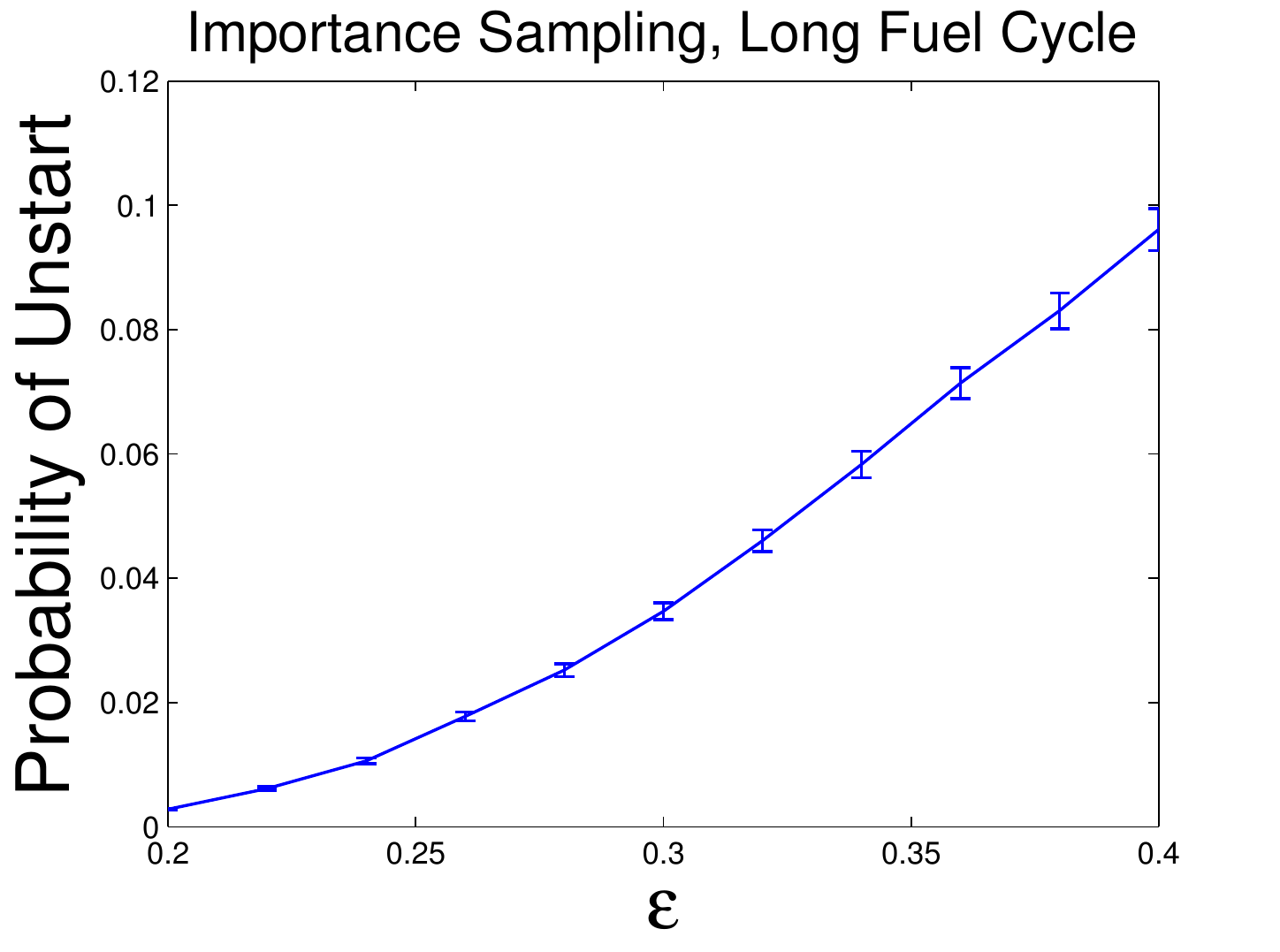}

	\includegraphics[width=0.49\textwidth]{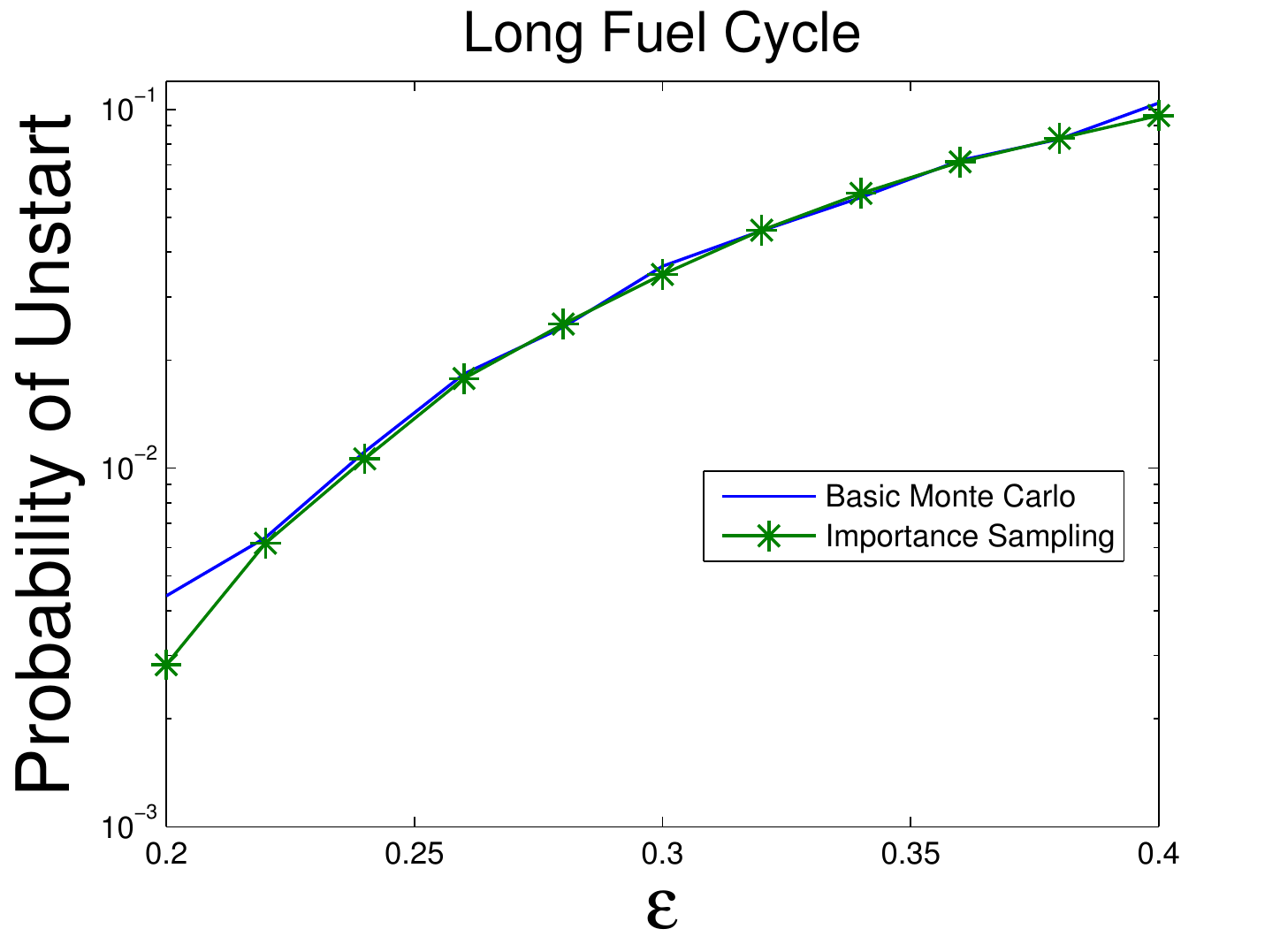}
	\includegraphics[width=0.49\textwidth]{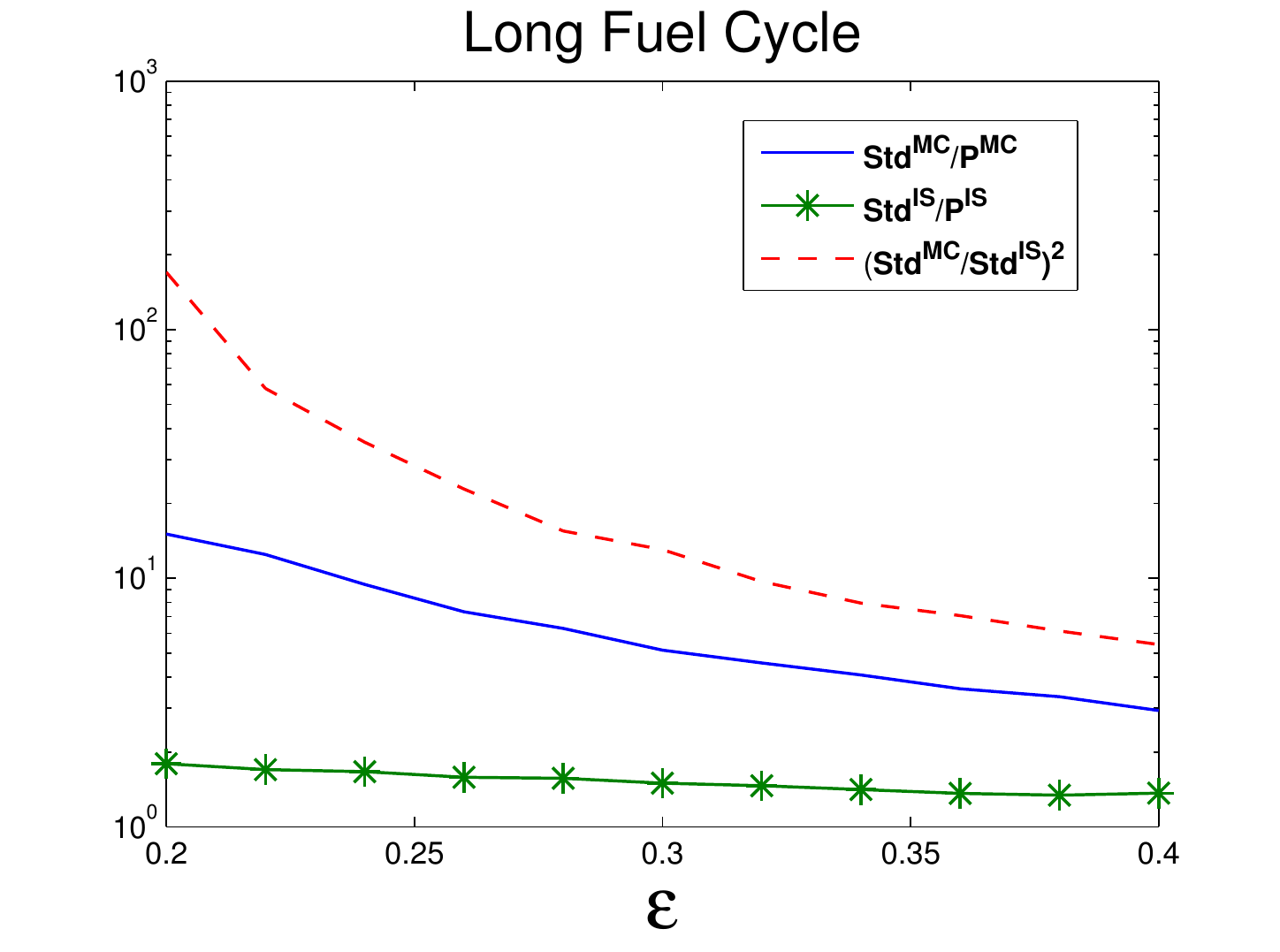}
	\caption{
	\label{fig:Monte Carlo for unstart, long fuel cycle}
	The basic Monte Carlo method and the importance sampling estimator for the probability of the unstart 
	when the long fuel cycle is used (the top figures). The error bar are the $99\%$ confidence intervals. 
	The bottom left figure is the plot of the estimated probabilities in the log scale. The bottom right  
	figure is the plot of the relative errors $\Std^{MC}_J/\hat{P}^{MC}$ and $\Std^{IS}_J/\hat{P}^{IS}$, and 
	the ratio of $\Std^{MC}_J$ to $\Std^{IS}_J$, which is the factor of the improvement of $\Std^{IS}_J$.}
\end{figure}

From Figure \ref{fig:Monte Carlo for unstart, short fuel cycle} and 
\ref{fig:Monte Carlo for unstart, long fuel cycle} we see that $\hat{P}^{IS}$ has the more accurate 
estimates; the improvement of $\hat{P}^{MC}$ increases when the estimated probability decreases. We also 
note that although the importance sampling is intrinsically designed for the small $\eps$ situations, our 
simulations show that it also improves the non-small $\eps$ cases. Because the rare of convergence of 
$\hat{P}^{MC}$ and $\hat{P}^{IS}$ are $1/\sqrt{J}$, the improvement of the speed is the square of the ratio 
of the standard deviations $\Std^{MC}_J/\Std^{IS}_J$. Then we see that the improvement of the speed of 
$\hat{P}^{IS}$ ranges from a factor of $4.2$ ($\eps=0.4$) to a factor of $381.5$ ($\eps=0.2$) when the short 
is used, and from a factor of $5.4$ ($\eps=0.4$) to a factor of $170.3$ ($\eps=0.2$) if we use the long fuel 
cycle. Roughly speaking, the ratio of the improvement of is proportional to $1/\PP(\bfA)$.
\section{Conclusion}
\label{sec:conclusion}

In this paper, we use the large deviation principle to analyze the probability of unstart of a scramjet due 
to the random perturbation of the inflow. The numerical analysis is performed under various comparisons: the 
impact of the fueling schedules, the sensitivity to the constraint sets and the effect of the engine 
geometry, some of which are also confirmed and are consistent in the previous literature using the Monte 
Carlo method (\cite{Iaccarino2011,West2011}); namely, the central analysis and the large deviation analysis 
have the high consistency on the region of operation. Further, the large deviation analysis gives a sharper 
information by providing the most probable inflow perturbation that causes unstart.

We also implement the large-deviation-based importance sampling to overcome the limitation of the basic 
Monte Carlo method. Our numerical results show that when the probability of unstart is small but still not
negligible (for example, the order of $10^{-3}$ or even $10^{-4}$), the importance sampling is significantly 
better than the basic Monte Carlo. The theoretical ratio of the improvement can be up to the reciprocal of 
the estimated probability.

\section*{Acknowledgment}

This work is partly supported by the Department of Energy [National Nuclear Security Administration] under 
Award Number NA28614, and partly by AFOSR grant FA9550-11-1-0266.

\appendix
\section{List of Parameters}
\label{sec:parameters}
The listed parameters are the default values. We will mention the changes in the main text if different 
values are used.

{\small
\centering
\begin{longtable}{|l|l|r|l|}
	\hline
	Variable & Name & Value & Units \\
	\hline
	State Variable &&&\\
	\hline
	$\rho$ & Density & - & $kg/m^3$ \\
	$u$ & Flow Speed & - & $m/s$\\
	$u_{in}$ & Inflow Speed & - & $m/s$\\
	$\rho u$ & Mass Flow & - & $kg / m^2 s$ \\
	$E$ & Energy Density & - & $J/m^3$ \\
	$P$ & Pressure & - & Pascals \\
	$M$ & Mach Number & - &\\
	$M_{in}$ & Inflow Mach Number & - &\\
	$\gamma$ & Ratio of Specific Heats & $1.4$ &  \\
	\hline
	Geometry &&& \\ 
	\hline
	$A_0$ & Minimum Cross Sectional Area of Engine & $0.008$ & $m^2$\\
	$L_I$ & Inlet Length & $0.5$ & $m$ \\
	$L_C$ & Combustor Length & $0.1$ & $m$ \\
	$L_E$ & Expansion Region Length & $0.1$ & $m$ \\
	$\theta_I$ & Angle of Inlet & $0.0$ & Degrees \\
	$\theta_C$ & Angle of Combustor & $7.5$ & Degrees \\
	$\theta_E$ & Angle of Expansion Region & $15.0$ & Degrees \\ \hline
	Fueling &&& \\ 
	\hline
	$\phi$ & Mixing Ratio & $0.78$ &  \\
	$f_{stoch}$ & Stochiometirc Fuel/Air Ratio & $0.029$ & \\
	$H_{prop}$ & Fuel Heating Value & $1.2 \times 10^8$& $J/Kg$\\
	$\rho_0$ & Free Stream Density - Taken as Inflow Density & $0.159$ & $kg/m^3$\\
	$u_0$ & Free Stream Velocity - Take as Inflow Velocity & $1300.0$ & $m/s$\\ 
	$P_0$ & Free Stream Pressure - Take as Inflow Pressure & $47842.0$ & Pascals\\
	$\tau_S$ & Short Fuel Cycle Length & $0.5$ & $ms$ \\
	$b_S$ & Short Fuel Burst Length & $0.1$ & $ms$ \\
	$\tau_L$ & Long Fuel Cycle Length & $2$ & $ms$ \\
	$b_L$ & Long Fuel Burst Length & $0.4$ & $ms$ \\
	\hline
	Numerics &&& \\ 
	\hline
	$T$ & Terminal Time & $0.01$ & $s$ \\
	$K$ & Number of Cells & $100$ & \\
	$N$ & Number of Time Grids for Large Deviations & $10^4$ & \\
	$\Delta t$ & Time Increment of the Euler schemes
	(\ref{eq:general SDE, discrete time}) (\ref{eq:full order u_in}) & $10^{-6}$ & $s$\\
	$\tilde{N}$ & Resolutions of $\tilde{u}_{in}$ in (\ref{eq:reduced order u_in}) & $20$ &\\
	$\sigma_u$ & Volatility of $u_{in}$ & $10^4$ & $m/s^{3/2}$\\
	$\sigma_M$ & Volatility of $M_{in}$ & $96.9020$ & $s^{-1/2}$\\
	\hline
	\caption{Table of Numerical Values}
	\label{tbl:values}
\end{longtable}}

\section{Numerical PDE Methods of the Governing Equation}
\label{sec:component-wise LLF}
In this section, we describe the numerical method to solve the governing equation (\ref{eq:qce}):
\begin{equation*}
	\u_t + \f_x 
	= \frac{A'(x)}{A(x)} \left( \begin{pmatrix} 0 \\ P \\ 0 \\ \end{pmatrix} - \f\right)
	+ \begin{pmatrix} 0 \\ 0 \\ f(x,t) \\ \end{pmatrix}.
\end{equation*}
Given a spatial discretization: $-L_I=x_0<\cdots<x_K=L_C+L_E$, $X^n_k$ denotes the average of the quantity X 
over the cell $(x_k,x_{k+1})$ at time $t_n$. We use the component-wise, first order local Lax-Friedrichs 
(LLF) scheme \cite[Algoritm 4.4]{Shu1999} to solve (\ref{eq:qce}):
\begin{multline}
	\label{eq:first order Euler method for the PDE}
	\begin{pmatrix} \rho_k^{n+1} \\ \rho_k^{n+1} u_k^{n+1} \\ E_k^{n+1} \\ \end{pmatrix}
	= \begin{pmatrix} \rho_k^n \\ \rho_k^n u_k^n \\ E_k^n \\ \end{pmatrix}
	- \frac{h}{\Delta x}(\mathcal{F}^n_{k+1/2}-\mathcal{F}^n_{k-1/2})\\
	+ h\frac{A'(x_{k+1/2})}{A(x_{k+1/2})} \left( \begin{pmatrix} 0 \\ P_k^n \\ 0 \\ \end{pmatrix}
	- \begin{pmatrix} \rho_k^n u_k^n \\ \rho_k^n (u_k^n)^2 + P_k^n \\  (E_k^n + P_k^n)u_k^n \\ \end{pmatrix} 
	\right)
	+ h\begin{pmatrix} 0 \\ 0 \\ f(x_{k+1/2},t_n) \\ \end{pmatrix},
\end{multline}
where $x_{k+1/2}=0.5(x_k+x_{k+1})$, $P_k^n=(\gamma-1)(E_k^n-\rho_k^n (u_k^n)^2/2)$.
$\mathcal{F}^n_{k\pm 1/2}$ are the numerical fluxes generated by the component-wise LLF method:
\begin{align*}
	\label{eq:component-wise LLF}
	\mathcal{F}^n_{k+1/2} 
	&= \frac{1}{2}\left(
	\begin{pmatrix} 
		\rho_{k+1}^n u_{k+1}^n \\ 
		\rho_{k+1}^n (u_{k+1}^n)^2 + P_{k+1}^n \\
		(E_{k+1}^n + P_{k+1}^n)u_{k+1}^n \\ 
	\end{pmatrix}
	+ 
	\begin{pmatrix} 
		\rho_k^n u_k^n \\ 
		\rho_k^n (u_k^n)^2 + P_k^n \\  
		(E_k^n + P_k^n)u_k^n \\ 
	\end{pmatrix}
	\right)\\
	&\quad - \frac{1}{2}\max\{|c_k^n+u_k^n|,|c_{k+1}^n+u_{k+1}^n|\}\left(
	\begin{pmatrix}
		\rho_{k+1}^n \\ 
		\rho_{k+1}^n u_{k+1}^n \\ 
		E_k^n \\ 
	\end{pmatrix}
	- 
	\begin{pmatrix} 
		\rho_k^n \\ 
		\rho_k^n u_k^n \\ 
		E_k^n \\ 
	\end{pmatrix}
	\right),
\end{align*}
where $c_k^n=\sqrt{\gamma P_k^n/\rho_k^n}$ is the speed of sound. Because the component-wise scheme works 
well for the low-order methods (see \cite{Shu1999}), we use it to reduce the computational cost.

For the inflow conditions, we let $\rho_0^n=\rho_0$ and $P_0^n=P_0$ for all $n$, and $u_0^n$ is governed by 
the stochastic inflow speed $u_{in}(n)$. For the outflow condition, the simple extrapolation is used:  
$(\rho_K^{n+1},u_K^{n+1},E_K^{n+1})=(\rho_{K-1}^n,u_{K-1}^n,E_{K-1}^n)$.

The following strategy is used to find the suitable initial condition: we simulate the numerical PDE with 
$\rho(0,x)\equiv\rho(t,-L_I)\equiv\rho_{0}$, $u(0,x)\equiv u(t,-L_I)\equiv u_0$, 
$E(0,x)\equiv E(t,-L_I)\equiv E_0$ and $f(t,x)\equiv 0$. Then with the aforementioned outflow extrapolation,
the numerical solution obtained by this setting has the equilibrium state $(\rho^e(x),u^e(x),E^e(x))$ and 
use this state as the initial condition for the LDP. 

In Section \ref{sec:LDP for unstart} and \ref{sec:numerical result of LD}, the time increment $h$ is taken 
as the uniform constant $\Delta t$ in (\ref{eq:full order u_in}). This is because the uniform time grid 
results in a smoother constraint set $\bfA$ in (\ref{eq:the set of unstart}), which increase the robustness 
of the numerical optimization $\inf_{\tilde{u}_{in}\in\bfA}I(\tilde{u}_{in})$. Because the inflow condition 
$\tilde{u}_{in}$ is a controlled variable, we can choose a sufficiently small $\Delta t$ so that the 
numerical scheme is stable when we solve $\inf_{\tilde{u}_{in}\in\bfA}I(\tilde{u}_{in})$. On the other hand, 
in Section \ref{sec:importance sampling}, as the inflow condition is random and not controlled, we use the 
adaptive time increment:
\[
	h^n = 0.8\times\frac{\Delta x}{\max_k|c_k^n+u_k^n|},
\]
to satisfy the CFL condition and avoid the instabilities due to extraordinary inflow conditions.

\bibliographystyle{siam}
\bibliography{reference}

\end{document}